\documentclass[12pt,twoside,a4paper]{amsart}
\usepackage{amssymb}
\usepackage{amsmath,amsfonts}
\usepackage[showonlyrefs]{mathtools}
\usepackage{bm}
\usepackage{tikz-cd}
\usepackage[
	colorlinks	= true,
	allcolors 	= blue
	]{hyperref}

\usepackage{amsthm}
\theoremstyle	{plain}
\newtheorem		{theorem}					{Theorem}	[section]
\newtheorem		{lemma}			[theorem]	{Lemma}
\newtheorem		{corollary}		[theorem]	{Corollary}
\newtheorem		{proposition}	[theorem]	{Proposition}

\theoremstyle	{definition}
\newtheorem		{definition}	[theorem]	{Definition}
\newtheorem		{remark}		[theorem]	{Remark}
\newtheorem*	{remark*}					{Remark}

\numberwithin{equation}{section}

\newcommand{\cC}{{\mathcal C}}

\newcommand{\cH}{{\mathcal H}}

\newcommand{\cK}{{\mathcal K}}
\newcommand{\cL}{{\mathcal L}}

\def\sA{{\mathfrak A}}      
      
\def\sG{{\mathfrak G}}   \def\sH{{\mathfrak H}}   
      \def\sL{{\mathfrak L}}
   \def\sN{{\mathfrak N}}   
      
\def\sS{{\mathfrak S}}

\newcommand {\ff}{\mathfrak f}
\newcommand {\fg}{\mathfrak g}

\def\bR{{\mathbf R}}
\def\bA{{\mathbf A}}

      \def\dC{{\mathbb C}}

      \def\dR{{\mathbb R}}

   \def\cB{{\mathcal B}}   \def\cC{{\mathcal C}}
      \def\cF{{\mathcal F}}
\def\cG{{\mathcal G}}   \def\cH{{\mathcal H}}   
   \def\cK{{\mathcal K}}   \def\cL{{\mathcal L}}
      
\def\cP{{\mathcal P}}   \def\cQ{{\mathcal Q}}   \def\cR{{\mathcal R}}
      
   \def\cW{{\mathcal W}}

\newcommand{\ptp}{p\times p}

\newcommand{\w}[1]{\widetilde{#1}}
\newcommand{\wt}{\widetilde}
\newcommand{\wh}{\widehat}

\DeclareMathOperator{\sgn}{sgn}

\DeclareMathOperator{\ran}{ran}
\DeclareMathOperator{\rank}{rank}
\DeclareMathOperator{\dom}{dom}
\DeclareMathOperator{\mul}{mul}

\DeclareMathOperator{\gr}{gr}

\DeclareMathOperator{\cdom}{\overline{dom}}

\def\ov{\overline}

\renewcommand{\Re}{\operatorname{Re}}
\renewcommand{\Im}{\operatorname{Im}}

\newcommand*\conj[1]{\overline{#1}}

\def\CR{\color{red} }
\def\CB{\color{blue} }

\newcommand{\dsp}{\displaystyle}

\begin{document}

\title[${\sL}$-resolvents and pseudo-spectral functions of symmetric  relations]
{${\sL}$-resolvents and pseudo-spectral functions of symmetric linear relations in Hilbert spaces}

\author[V.~Derkach]{Volodymyr Derkach}
\address{ 
	Vasyl Stus Donetsk National University\\
	Vinnytsia, Ukraine}
\email{derkach.v@gmail.com}


\dedicatory{In respectful memory of Heinz Langer }

\date{\today}


\begin{abstract}
The set of $\sL$-resolvents of a densely defined symmetric operator
in a Hilbert space $\sH$  with a proper  gauge $\sL(\subset\sH)$ was described by Kre\u{\i}n and Saakyan.
The Kre\u{\i}n--Saakyan  theory of $\sL$-resolvent matrices was extended by
Shmul'yan and Tsekanovskii to the case of improper  gauge $\sL(\not\subset\sH)$
and by Langer and Textorius to the case of symmetric linear relations in Hilbert spaces.
In the present paper we find connections between the theory of boundary triples
and the Kre\u{\i}n--Saakyan  theory of $\sL$-resolvent matrices
for symmetric linear relations with
improper  gauges in Hilbert spaces and extend the known formula for the  $\sL$-resolvent matrix
in terms of boundary operators to this class of relations.
Descriptions of spectral and pseudo-spectral functions of symmetric linear relations with
improper  gauges are given.
The results are applied
to  linear relations generated by a canonical system.

\end{abstract}


\maketitle

\tableofcontents

\section{Introduction} \label{sec:intro}
In \cite{LT77,LaTe78,LaTe82,LaTe84,LaTe85} Heinz Langer and Bjorn Textorius conducted a series of studies on the theory of $\sL$-resolvents and spectral functions
of symmetric linear relations in Hilbert spaces.
In this paper we are reviewing those results from the perspective of the theory
of boundary triples developed in~\cite{Koc75,GG91,DM95,DM25}.
The topics discussed in the paper are closely related to the
research of M.G. Kre\u{\i}n and Sh.N. Saakyan \cite{Kr46,KS70}
 on the theory of  $\sL$-resolvent matrices,
 Yu.L.~Shmul'yan and E.R.~Tsekanovskii \cite{Sh71,ShTs77}
 on representation theory of operators with improper gauges,
L.A.~Sakhnovich~\cite{SakL99},
 A.L.~Sakhnovich~\cite{Sak92}, I.S.~Kac \cite{Kac03}
 and D.Z.~Arov and H. Dym \cite{ArDy12}
 on canonical systems.

To explain the main ideas we will start with
a closed densely defined symmetric linear operator $A$ with the deficiency index $(p,p)$, $p<\infty$,
acting in a Hilbert space $\sH$.
A linear subspace $\sL(\subset\sH)$ of dimension $p$
is called  a {\it proper gauge} for $A$ if the set $\rho(A,\sL)$ of regular type points ${z}\in\dC$ of $A$ such that the space $\sH$ admits the direct decomposition
\begin{equation}\label{eq:Prop_Gauge}
  \sH=\ran(A-{z} I)\dotplus \sL
\end{equation}
is non-empty.
Let  $L$ be a linear bijective mapping from $\dC^p$ onto $\sL$,
let $\wt A$ be a selfadjoint extension of $A$ acting in a (possibly larger) space
$\wt\sH(\supseteq\sH)$,
 and let $\rho(\wt A)$
be the set of regular points of $\wt A$.
Then the matrix-valued function (mvf)
\begin{equation}\label{eq:Lresolv}
  r({z})=L^*(\wt A-{z} I)^{-1}L,\quad {z}\in\rho(\wt A),
\end{equation}
is called the  $\sL$-{\it resolvent} of $A$.
The set of all $\sL$-resolvents of $A$ was described by M.G.~Kre\u{\i}n in \cite{Kr44}
and \cite{Kr46} in the case $p<\infty$
and by S.~Saakyan in the case $p=\infty$ \cite{Sa65}{\CR ,}
by the formula
\begin{equation}\label{eq:LresolvDescr}
  r({z})=T_{W({z})}[\tau({z})]:=
  (w_{11}({{z}})\tau({{z}})+w_{12}({{z}}))
(w_{21}({{z}})\tau({{z}})+w_{22}({{z}}))^{-1},
\end{equation}
where ${z}\in\rho(\wt A)\cap\rho(A,\sL)$ and $\tau$ ranges over the set $\wt\cR^{p\times p}$ of Nevanlinna families, see Definition~\ref{def:Nk-family}.
The block mvf $W({z})=[w_{ij}({z})]_{i,j=1}^2$
generating the linear fractional transform $T_{W({z})}$ is called the
$\sL$-{\it resolvent matrix} of $A$.
The theory of $\sL$-resolvent matrices of symmetric operators $A$ with
proper gauges was developed in \cite{Kr49}, \cite{Sa65} and \cite{KS70}.
In particular, the kernel
\begin{equation}\label{kerK0}
{\mathsf K}_{\zeta}^W({z}):=
\frac{ J_p-W({z})J_p W({\zeta})^*}{-({z}-\ov{\zeta})},
\quad {z}\ne\overline{{\zeta}},\quad\text{where}\quad
J_p=\begin{bmatrix}
    O_{p} & - I_{p}\\
     I_{p} & O_{p}
    \end{bmatrix},
\end{equation}
based on the $\sL$-resolvent matrix $W$ is non-negative on $\rho(A,{\sL})$.
Associated with gauge $\sL$ there  are two operator-valued functions (ovf's):
{ $\cP({z})=L^{-1}\Pi^{z}_{\sL}:\sH\to{\mathbb C}^p$,
where $\Pi^{z}_{\sL}$ is the skew projection} onto $\sL$ along $\ran(A-{z} I)$,
and the ovf { $\cQ({z}):\sH\to{\mathbb C}^p$} given by
\begin{equation}\label{eq:Q_lambda_Intro}
   \cQ({z})={ L^*}(A-{z} I)^{-1}(I-\cP({z})),\quad {z}\in\rho(A,\sL).
\end{equation}
If there exists a real point $a\in\rho(A,\sL)$, the $\sL$-resolvent matrix $W(z)$
can be calculated by
\[
  W({z})=I_{2p}-({z}-a)\cG({z})\cG(a)^*J_p,\ \ \text{where}\ \
  \cG({z})=\begin{bmatrix}
   -\cQ({z})\\
      \cP({z})
    \end{bmatrix},\,\,
 {z}\in\rho(A,\sL).
\]

An effective tool in the study of extension theory of symmetric operators
is the notion of the boundary triple $\Pi=(\cH,\Gamma_0,\Gamma_1)$ (see Definition~\ref{def:btriple}) introduced by J. Calkin~\cite{Cal39} and developed
by A.N. Kochubei \cite{Koc75}, V.M. Bruk \cite{Br76} and M.L. and V.I. Gorbachuks~\cite{GG91}.
In \cite{DM91}  relations between the Kre\u{\i}n's representation theory
and the theory of boundary triples
were investigated and, in particular, it was shown that to a given
boundary triple $\Pi$ and a gauge $\sL$
corresponds a unique $\sL$-resolvent matrix by the formula
\begin{equation}\label{eq:Formula_W_Intro}
    W_{\Pi\sL}({{z}}):=
    \begin{bmatrix}
   -\Gamma_0\cQ({z})^{*}  &   \Gamma_0\cP({z})^{*}\\
   -\Gamma_1\cQ({z})^{*} &   \Gamma_1\cP({z})^{*}
       \end{bmatrix}^*,\quad
 {z}\in\rho(A,\sL).
\end{equation}

However, some problems of the analysis require  consideration of $\sL$-resolvents of symmetric operators $A$ with improper gauges $\sL$ consisting of generalized elements
from a space $\sH_-$ of ``distributions'', see Section~\ref{sec:4.1}.
The theory of $\sL$-resolvent matrices and a description of $\sL$-resolvents of a symmetric operator $A$ with
improper gauge was developed by Yu. Shmul'yan and E. Tsekanovskii, \cite{ShTs77}.
This theory was extended by H.~Langer and B.~Textorius to the case of a symmetric linear relation $A$
and applied to the problem of description of $\sL$-resolvents and spectral functions
of operators associated with a canonical system, see~\cite{LaTe84}, \cite{LaTe85}.

In the present paper we  consider a symmetric linear relation $A$
in a Hilbert space with deficiency index $(p,p)$
that has an improper  gauge $\sL$, i.e. $\sL\not\subset\sH$.
The results of Section \ref{sec:3} on extended generalized resolvents
and extended spectral functions are taken from~\cite{Sh71,Sh71b,ShTs77}, \cite{LaTe78}, and~\cite{DD19,DD21}.
The results about extended boundary triples and extended
Green formulas presented in Lemma~\ref{lem:bf_A0Pi}  and Lemma~\ref{lem:ExtGreenForm} seem to be new.
Using boundary triple's approach we prove an analog of formula~\eqref{eq:Formula_W_Intro} for  the $\sL$-resolvent matrix $W_{\Pi\sL}({z})$ of $A$ in Theorem~\ref{thm:ResM}.
Under an additional assumption it is shown, that the $\sL$-resolvent matrix $W_{\Pi\sL}({z})$ of $A$  belong to the class
$\cW(iJ_p)$, i. e.
the kernel ${\mathsf K}_{\zeta}^W({z})$ 
is non-negative on $\rho(A,{\sL})$, see Definition~\ref{def:11.cP}.
The 
$\sL$-resolvent matrix $W_{\Pi\sL}({z})$ allows to describe the set of
$\sL$-resolvents of $A$ with improper  gauge $\sL$ by the same formula~\eqref{eq:LresolvDescr}, see Theorem~\ref{prop:Gresolv}.
A special interest represent those $\sL$-resolvents of $A$ whose minimal representing linear relation $\wt A$ has the property
\begin{equation}\label{eq:mul_wtA_Intro}
  \mul \wt A=\mul A.
\end{equation}
Parameters $\tau$ corresponding to such $\sL$-resolvents via formula~\eqref{eq:LresolvDescr} are called $\Pi$-admis\-sible. Some $\Pi$-admissibility criteria for $\sL$-resolvents of $A$ were
given in \cite{LT77} and \cite{DHMS2000}.

 $\sL$-resolvents  of a Pontryagin space symmetric operator $A$ with improper gauge $\sL$ and deficiency index $(1,1)$ were studied by M. Kaltenb\"{a}ck and H.~Woracek,~\cite{KW98}.
 In \cite{DD21,DD24} V. Derkach and H. Dym employed  the theory of de Branges--Pontryagin spaces  to describe
 $\sL$-resolvents  of a Pontryagin space symmetric operator $A$ with deficiency index $(p,p)$ and to parameterize solutions
of indefinite truncated matrix moment  problem.
An indefinite version of formula~\eqref{eq:Formula_W_Intro} for  the $\sL$-resolvent matrix $W_{\Pi\sL}({z})$ of $A$ was presented in~\cite{Der25}.

Another question reviewed in this paper is a description of spectral and pseudo-spectral functions of $A$.
Recall the following classification from \cite{LaTe78}, \cite{Sak92}, \cite{Kac03},  \cite{ArDy12} and \cite{Mog15MFAT}.
\begin{definition}\label{def:PseudoSF_Intro}
Let $\sL(\subset\sH_-)$ be a gauge for $A$ such that $\dR\subset\rho(A,\sL)$.
A non-decreasing left-continuous mvf $\dR\ni\mu\mapsto \sigma(\mu)$ with values in $\dC^{p\times p}$
is called:
	\begin{enumerate}
		\item[(1)]
{\it LT-spectral function} (after Langer and Textorius) of the pair
$\langle A;\sL\rangle$ if the mapping
\begin{equation}\label{eq:GenFourierTrIntro}
\cF:\sH\ni f\mapsto \cP({\cdot})f\in L^2(d\sigma)
\end{equation}
is  a contraction from $\sH$ to $ L^2(d\sigma)$
that is isometric on  $\cdom A$;

		\item[(2)]
{\it pseudo-spectral function}  of the pair $\langle A;\sL\rangle$   if
the mapping $\cF$ is  a partial isometry from $\sH$ to $L^2(d\sigma)$
such that $\ker \cF=\mul A$;
\item[(3)]
{\it spectral function}  of  $\langle A;\sL\rangle$  if
$\cF$ is an  isometry from $\sH$ to $L^2(d\sigma)$.
\end{enumerate}
\end{definition}

The mapping $\cF$ is called a generalized Fourier transform
and it is a kind of the directing mapping in the method
of directing functionals invented by M. Krein \cite{Kr48}, and
developed by H. Langer and B. Textorius  \cite{La70,LaTe78, LaTe84}.
This method is modified here to the case when the gauge $\sL$
is considered as a subspace of $\sH_-$ and leads to the following parametrization of
LT-spectral functions (resp., pseudo-spectral functions)  of the pair $\langle A;\sL\rangle$:
\begin{equation}\label{eq:IntRepR2_Intro}
 \int_{\dR}\left(\frac{1}{{\lambda}-z}-\frac{{\lambda}}{1+{\lambda}^2}\right)d\sigma({\lambda})=
 r(z)+K,
\end{equation}
where $r(z)$  ranges over the set of all $\sL$-resolvents of $A$ (resp., such that \eqref{eq:mul_wtA_Intro} holds), see Theorem~\ref{thm:Descr_Pseudo_IS}
  and $K$ is a constant $p\times p$-Hermitian matrix.
The set of spectral functions of $\langle A;\sL\rangle$ is nonempty if and only if $\mul A=\{0\}$ and in this case it coincides with the set of
pseudo-spectral functions of $\langle A;\sL\rangle$.
A combination of the
formulas~\eqref{eq:LresolvDescr} and~\eqref{eq:IntRepR2_Intro}
gives a parametrization  of
pseudo-spectral functions  of the pair $\langle A;\sL\rangle$
via the formula~\eqref{eq:IntRepRMin}
when $\tau$ ranges over the set of all $\Pi$-admissible parameters in $\wt\cR^{p\times p}$.

In Section \ref{sec:6} the results of this paper are applied to canonical systems
   \begin{equation}\label{eq:can,eq-n_Intro}
J_p f'(t)+\cF(t)f(t) = {z}\mathcal H(t)f(t), \quad  t\in(a,b), 
  \end{equation}
where  ${2p\times 2p}$ mvf's $\cF(t)$ and $\cH(t)$  are real Hermitian mvf's
with entries from $L^1_{\rm loc}[a,b)$, $b\le\infty$, $\text{tr }\cH(t)\equiv 1$ and  $\cH(t)\ge 0$ for a.e. $t\in (a,b)$.
We review the known results about  $\sL$-resolvents and boundary triples for regular
and singular canonical systems from \cite{LaTe82,LaTe84,LaTe85}, \cite{BeHaSn20}, \cite{DM25}
and present new formulas for $\sL$-resolvent matrices.
Descriptions of spectral and pseudo-spectral functions  of the minimal
operator $S_{\min}$ generated by the canonical system \eqref{eq:can,eq-n_Intro} in
$L^2_\cH(a,b)$ are given both in the regular case (Theorem~\ref{thm:Descr_Pseudo_IS3})
and in the singular case when the system \eqref{eq:can,eq-n_Intro} is limit point at $b$ (Theorem~\ref{thm:Descr_Pseudo_LP}).
The obtained results complement known results by A. Sakhnovich \cite{Sak92},
and
by D. Arov and H. Dym~\cite{ArDy12}.
In the case when $p=1$ our characterization of pseudo-spectral functions in terms of admissible parameters $\tau$ coincides with  the characterization given by I.S. Kac~\cite{Kac03} in terms of $\cH$-indivisible intervals, see Propositions~\ref{prop:Hii_intermed} and \ref{prop:Indivis_Int}.
Description of LT-spectral functions
of  symmetric relations associated with  canonical  systems was presented
by H.~Langer and B. Textorius in~\cite{LaTe85}.

\section{Preliminaries}\label{sec:2}
\subsection{Linear relations in Hilbert spaces}
\label{subsec:pre:lr}
Let us recall some definitions from  \cite{Arens}, \cite{Ben72}.
 A linear relation $T$ in a Hilbert space
 $(\sH,(\cdot,\cdot)_{\sH})$
is a linear subspace of $\sH \times \sH$.
We use two representations for elements of $\sH \times \sH$: either
$\begin{bmatrix} f \\ g \end{bmatrix}$ or $\text{\rm col} \{f,g\}$, where $f,g\in\sH$.
The \emph{domain}, \emph{kernel}, \emph{range},  and \emph{multivalued part} of a linear relation $T$ are defined as follows:
\begin{align*}
	\dom{T} &\coloneqq \left\{ f \colon \begin{bmatrix} f \\ g \end{bmatrix} \in T \ \text{for some}\ g\in\sH \right\}, &
	\ker{T} &\coloneqq \left\{ f \colon \begin{bmatrix} f \\ 0 \end{bmatrix} \in T \right\},
\\
	\ran{T} &\coloneqq \left\{ g \colon \begin{bmatrix} f \\ g \end{bmatrix} \in T \ \text{for some}\ f\in\sH \right\}, &
	\mul{T} &\coloneqq \left\{ g \colon \begin{bmatrix} 0 \\ g \end{bmatrix} \in T \right\}.
\end{align*}
The \emph{adjoint} linear relation $T^*$ is defined by
\begin{equation}
	T^* \coloneqq \left\{
	\begin{bmatrix} u \\ f \end{bmatrix}
	\in \sH \times \sH \colon ( f,v)_{\sH} =( u,g)_{\sH}\ \text{for any}\
	\begin{bmatrix} v \\ g \end{bmatrix}
	\in T \right\}.
\end{equation}
A linear relation $T$ in $\sH$ is called \emph{closed} if $T$ is closed as a subspace of $\sH \times \sH$.
The set of all closed linear operators (relations) is denoted by $\mathcal{C}(\sH)$ ($\wt{\mathcal{C}}(\sH)$).
Identifying a linear operator $T \in \mathcal{C}(\sH)$ with its graph one can consider $\mathcal{C}(\sH)$ as a part of $\wt{\mathcal{C}}(\sH)$.

Let $T$ be a closed linear relation, ${z} \in \mathbb{C}$, then
	\begin{equation}
		T - {z} I \coloneqq \left\{ \text{\rm col}\{ f , g-{z} f \} \colon
		\text{\rm col}\{  f , g \}
		\in T \right\}.
	\end{equation}
	A point ${z} \in \mathbb{C}$ such that $\ker{\left( T - {z} I \right)} = \{0\}$ and $\ran{\left( T - {z} I \right)} = \sH$ is called a \emph{regular point} of the linear relation $T$. Let $\rho(T)$ be the set of regular points.
	The \emph{point spectrum} $\sigma_p(T)$ of the linear relation $T$
	is defined by
	\begin{equation}\label{eq:Point_s}
		\sigma_p(T) \coloneqq
		\{{z}\in\mathbb{C} \colon \ker(T-{z} I)\ne\{0\}\},
	\end{equation}

A linear relation $A$ is called \emph{symmetric} in a Hilbert space
 $\sH$ if $A \subseteq A^{*}$.
A point ${z} \in \mathbb{C}$ is called a \emph{point of regular type}  for a closed symmetric linear relation $A$ (and is denoted as ${z} \in \widehat{\rho}(A)$), if ${z} \notin \sigma_p(A)$ and the subspace $\ran(A-{z} I)$ is closed in $\sH$.
For ${z} \in \widehat{\rho}(A)$ let us set
$\sN_{{z}} \coloneqq \ker (A^* -{z} I)$ and
\begin{equation}
	\widehat{\sN}_{{z}}  \coloneqq
	\left\{
	\wh f_{{z}} =\text{\rm col} \{f_{{z}} , {z} f_{{z}}\}
	 \colon f_{{z}} \in \sN_{{z}}
	\right\}.
\end{equation}
As is known, see  \cite[Theorem 6.1]{IKL},  the number $\dim{\sN}_{{z}}$ takes a constant value $n_+(A)$ for all ${z}\in\wh\rho(A)\cap\dC_+$, and $n_-(A)$ for all ${z}\in\wh\rho(A)\cap\dC_-$, where
$
 \dC_\pm=\{{z}\in\dC:\, \pm\text{Im }{z}>0 \}.
$ 
The numbers $n_\pm(A)$ are called the \textit{defect numbers} of $A$.

\subsection{$\cR$-families}\label{sec:2.2}
A complex scalar function $\sigma({\lambda} )$ is related to the class $\sS^{(k)}$, $k=0,1,2$, if
it has a bounded variation on every finite interval and
\[
\int_{\dR}\frac{d\sigma({\lambda} )}{1+|{\lambda} |^k}<\infty.
\]
For a $\ptp$ matrix-valued function (mvf) $Q({z})$ holomorphic on
     $\dC_+\cup\dC_-$ let us set $Q^\#({z}):=Q(\overline{z})^*$.
\begin{definition} \label{def:Nk-class}
Let $\cH$ be a Hilbert space. A $\cB(\cH)$-valued ovf
$Q({z})$ holomorphic on
     $\dC_+\cup\dC_-$  is said to belong to the class
     $\cR[\cH]$ if:
\begin{enumerate}
  \item [(a)] the kernel
  \[
  {\mathsf
N}^Q_{\zeta}({z})
:=\left\{\begin{array}{ll}
\frac{\dsp Q({z})-Q({\zeta})^*}
{\dsp{z}-\ov{{\zeta}}}
&:\quad \textrm{for }{z},\,{\zeta}\in\dC_+\cup\dC_-, \ {z}\ne\ov{{\zeta}},\\
Q'({z})
&:\quad \textrm{for }{z},\,{\zeta}\in\dC_+\cup\dC_-, \ {z}=\ov{{\zeta}},
\end{array}\right.
  \]
  is non-negative on  $\dC_+\cup\dC_-$;
  \item [(b)] $Q^\#({z})=Q({z})$ for all ${z}\in\dC_+$.
\end{enumerate}
\end{definition}
Let us set $\cR^{\ptp}:=\cR[\dC^p]$. The class $\cR^{\ptp}$  coincides with the class of mvf's holomorphic on
     $\dC_+\cup\dC_-$ which satisfy the assumption (b) in Definition~\ref{def:Nk-class} and have non-negative imaginary part $\text{Im }Q({z})$ in $\dC_+$.
Recall, see~\cite{KaKr74}, that each  function $Q\in\cR^{\ptp}$ admits the integral representation
\begin{equation}\label{eq:IntRep}
  Q({z})=\alpha + \beta{z}+\int_{\dR}\left(\frac{1}{\lambda-{z}}-\frac{{\lambda}}{1+{\lambda}^2}\right)d\sigma({\lambda}),
\end{equation}
where $\alpha,\beta\in\dC^{\ptp}$, $\beta\ge 0$, $\alpha=\alpha^*$ and $\sigma({\lambda})$ is a non-decreasing
left-continuous ${\ptp}$ mvf of the class $\sS^{(2)}$.

For every $Q\in\cR^{\ptp}$ the corresponding spectral function $\sigma(t)$ can be found by the Stieltjes inversion formula
\begin{equation}\label{eq:IntRep3}
 \sigma({\lambda}_2)-\sigma({\lambda}_1)=\frac{1}{2\pi} \lim_{\nu\downarrow 0}\int_{{\lambda}_1}^{{\lambda}_2}\Im \varphi(\mu+i\nu)d\mu.
\end{equation}

\begin{definition} \label{def:Nk-family}
    A family $\tau=(\tau({z}))_{z\in\dC_+\cup\dC_-}$, where
    $\tau({z})=\ran\begin{bmatrix}
    \varphi({z})\\\psi({z})
    \end{bmatrix}$ and $ \varphi$ and $\psi$ are $\ptp$ mvf's holomorphic on
     $\dC_+\cup\dC_-$,  is called an $\cR^{\ptp}$-{\it family} if:
\begin{enumerate}
  \item [(i)]
  ${\text{Im}{z}}\cdot{\text{Im}\left(\varphi({z})^*\psi({z})\right)}\geq0$
for ${z}\in\dC_+\cup\dC_-$;
  \item [(ii)] $\varphi^\#({z})\psi({z})
      -\psi^\#({z})\varphi({z})$ for all ${z}\in\dC_+\cup\dC_-$;
  \item[(iii)] $\ker \varphi({z})\cap\ker \psi({z})=\{0\}$
  for all ${z}\in\dC_+\cup\dC_-$.
\end{enumerate}
The set of $\cR^{\ptp}$-families is denoted by $\wt\cR^{\ptp}$.
\end{definition}

Two $\cR^{\ptp}$-families
$(\tau_j({z}))_{z\in\dC_+\cup\dC_-}$, where
$\tau_j({z})=\text{\rm col} \{\varphi_j({z}),\psi_j({z})\}$ ($j=1,2$),
coincide, if
$\varphi_2 ({z})=\varphi_1({z})\chi({z})$
and $\psi_2({z})=\psi_1({z})\chi({z})$
for some mvf $\chi({z})$,
which is holomorphic and invertible on $\dC_+\cup\dC_-$.

Every $\cR^{\ptp}$-family $\tau=
\left(\text{\rm col} \{\varphi({z}),\psi({z})\}\right)_{z\in\dC_+\cup\dC_-}$
    can be treated as a family of linear relations in $\dC^p$.
In the case when $\det\varphi({z})\not\equiv 0$
it has no zeros on $\dC_+\cup\dC_-$ and
 $\tau(z)$ coincides on $\dC_+\cup\dC_-$ with
the graph of $\cR^{\ptp}$-function
$Q({z})=\psi({z})\varphi({z})^{-1}$.
Since every mvf from  $\cR^{\ptp}$ admits such a representation,
the set $\cR^{\ptp}$ can be identified with a subset of
$\wt \cR^{\ptp}$.

If for some $z\in\dC\setminus \dR$ and $\tau\in\wt\cR^{\ptp}$ the multivalued part $\mul\tau(z)$ is not trivial, then $\mul\tau(z)(=: \cH_2)$ does not depend on $z\in\dC\setminus \dR$ and, with
$\cH_1:=\dC^p\ominus\cH_2$,
the  family of linear relations $\tau(z)$ admits the decomposition
\begin{equation}\label{eq:tau_decomp}
  \tau(z)=\tau_{\rm op}(z)\oplus \text{col }\{0,\cH_2\},
\end{equation}
where $\tau_{\rm op}\in\cR[\cH_1]$ is
called the operator part of  $\tau\in\wt\cR^{\ptp}$.

Often it is more convenient to consider a dual representation of an $\cR^{\ptp}$-family $\tau$ via an $\cR^{\ptp}$-pair.
\begin{definition} \label{def:Nk-pair}
    A pair $\begin{bmatrix}
    C({z})& D({z})
    \end{bmatrix}$ where $ C$ and $D$ are $\ptp$ mvf's holomorphic on
     $\dC_+\cup\dC_-$  will be called a Nevanlinna pair (or $\cR^{\ptp}$-{\it pair}) if:
\begin{enumerate}
  \item [(i)]
  ${\text{Im}{z}}\cdot{\text{Im}\left(C(z)D(z)^*\right)}{\CB \geq 0}$
for ${z}\in\dC_+\cup\dC_-\,;$
  \item [(ii)] $C({z})D^\#({z})=D({z})C^\#({z})$ for all ${z}\in\dC_+\cup\dC_-$;
  \item[(iii)] $\rank \begin{bmatrix}
    C({z})& D({z})
    \end{bmatrix}=p$
  for all ${z}\in\dC_+\cup\dC_-$.
\end{enumerate}
\end{definition}
\begin{lemma}\label{lem:Nkpairs_fam}
There is a one to one correspondence between  $\cR^{\ptp}$-families
$\tau({z})=\ran\begin{bmatrix}
    \varphi({z})\\\psi({z})
    \end{bmatrix}$
    and  $\cR^{\ptp}$-pairs $\begin{bmatrix}
    C({z})& D({z})
    \end{bmatrix}$
established by the formulas
\begin{equation}\label{eq:Nkpairs_fam}
  C({z})=\psi^\#({z}),\quad D({z})=\varphi^\#({z}), \quad
  {z}\in\dC_+\cup\dC_-.
\end{equation}
The $\cR^{\ptp}$-family $\tau$  admits the following representation
\begin{equation}\label{LR_50}
\begin{split}
\tau({z})&=\ker{\begin{bmatrix} {C({z})} & -{D}({z})\end{bmatrix}}\\
&=\left\{\text{\rm col} \{ u,u' \}\in \dC^p\times\dC^p: {C({z})}u- {D}({z})u'=0\right\},
\quad {z}\in\dC_+\cup\dC_-.
\end{split}
\end{equation}
\end{lemma}
\begin{proof}
With this definition of the pair $\begin{bmatrix}
    C({z})& D({z})
    \end{bmatrix}$
the conditions (i)-(iii) in Definition~\ref{def:Nk-family} are equivalent
to  conditions (i)-(iii) in Definition~\ref{def:Nk-pair}.
  The formula~\eqref{LR_50} follows from the formula (ii) in Definition~\ref{def:Nk-pair}.
\end{proof}

In the case when  $\varphi$ and $\psi$ are constant matrices,
 conditions (ii)-(iii) in Definition~\ref{def:Nk-pair} characterize all selfadjoint
 linear relations in $\dC^p$.
 \begin{lemma}[\cite{RB69}]\label{lem:SALR}
Every selfadjoint
 linear relation in $\dC^p$ admits the representation  $\tau=\ker{\begin{bmatrix} C &D \end{bmatrix}}$
 where $C,\, D\in {\dC}^{\ptp}$ and
\begin{equation}\label{eq:rank_and sym}
{C}{D}^*={D}{C}^* \quad\mbox{and}\quad
\rank \begin{bmatrix}
    C& D
    \end{bmatrix}=p.
\end{equation}
\end{lemma}
\subsection{Boundary triples and Weyl functions}
\label{subsec:pre:triples}
Let $A$ be a closed symmetric linear relation in a Hilbert space ${\sH}$ with equal defect numbers $n_\pm(A)=p<\infty$. We will need the following
\begin{lemma}\label{lem:11_A*}
If $\wt A$ is a selfadjoint extension of $A$ and ${z}\in\rho(\wt A)$, then
  \begin{equation}\label{eq:11_A*wtA}
 A^*=\wt A\dotplus\wh\sN_{{z}}.
  \end{equation}
\end{lemma}
In the case of a densely defined operator the notion of the boundary triple was introduced in \cite{Koc75,GG91} under the name ``space of boundary values''.
This notion  was adapted to the case of a
non-densely defined symmetric operator in 
\cite{M92}, \cite{DM95}.
and to the case of linear relations in~\cite{D99}.
In the next definition we follow~\cite{DM95} and~\cite{D99}.

\begin{definition} \label{def:btriple}
	A set $\Pi = (\dC^p,\Gamma_0,\Gamma_1)$, where $\Gamma_0$ and $\Gamma_1$ are linear mappings from $A^*$ to $\dC^p$,
is called a \emph{boundary triple} for the linear relation $A^*$, if:
	\begin{enumerate}
	\item [(i)]
	for all
	$\wh f= \text{\rm col} \{ f . f' \}$,
	$\wh g =\text{\rm col} \{ g, g' \} \in A^*$
	the following Green's identity holds
	\begin{equation} \label{eq:1.9}
		(f',g )_{\sH} - ( f,g')_{\sH}
=(\wh\Gamma_0\wh g)^*(\wh\Gamma_1\wh f)-(\wh\Gamma_1\wh g)^*(\wh\Gamma_0\wh f);
	\end{equation}
	\item [(ii)]
	the mapping
	$\Gamma=\begin{bmatrix}\Gamma_0 \\ \Gamma_1\end{bmatrix} \colon	A^* \rightarrow \dC^{2p}$
	is surjective.
	\end{enumerate}
\end{definition}

The following linear relations
\begin{equation} \label{e q:A0A1}
	A_0 \coloneqq \ker \Gamma_0, \qquad A_1 \coloneqq \ker \Gamma_1
\end{equation}
are selfadjoint extensions of the symmetric linear relation $A$.

The main analytical tool in description of spectral properties of
selfadjoint extensions of a symmetric linear relation $A$
is the abstract Weyl function, which in the case of a single-valued operator  $A$
was introduced and investigated in~\cite{DM91,DM95}.

\begin{definition} \label{def:11W00}
Let $A$ be a closed symmetric linear relation with $n_+(A)=n_-(A)\le \infty$  and let ${\Pi}=({\dC^p},{\Gamma}_0,{\Gamma}_1)$ be a boundary triple for $A^*$. The {\it Weyl function} $M({\cdot})$ and the $\gamma$-{\it field} $\gamma({\cdot})$ of $A$ corresponding to
the boundary triple~$\Pi$ are defined by
   \begin{equation}\label{eq:11.M}
 M({{z}})\Gamma_0\wh f_{{{z}}}=\Gamma_1\wh f_{{{z}}},\quad \wh f_{{z}}\in \wh\sN_{{{z}}},
 \quad
 {{z}}\in\rho(A_0);
\end{equation}
and
\begin{equation}\label{eq:11_gamma1}
\wh\gamma({{z}})=(\Gamma_0\upharpoonright{\wh\sN_{{{z}}}})^{-1}, \quad
\gamma({{z}})=\pi_1\wh\gamma({{z}}), \quad
 {{z}}\in\rho(A_0),
\end{equation}
where $\pi_1$ is the projection onto the first component in $\sH\times\sH$
and $A_0$ is given by~\eqref{e q:A0A1}.
\end{definition}
By Lemma~\ref{lem:11_A*}, the operator $\Gamma_0\upharpoonright{\wh\sN_{{{z}}}}:{\wh\sN_{{{z}}}}\to\dC^p$ is boundedly invertible, the operator-functions
$\wh\gamma({{z}})$ and  $\gamma({{z}})$ admit the representations
 \begin{equation}\label{eq:11.wh_gamma2}
  \wh\gamma({{z}})=\wh\gamma({{\zeta}})+
  ({{z}}-{{\zeta}})\begin{bmatrix}
                (A_0 - {z})^{-1}\gamma({{\zeta}}) \\
                \gamma({{\zeta}})+{{\zeta}}(A_0 - {z})^{-1}\gamma({{\zeta}})
              \end{bmatrix},\quad  {{z}},{{\zeta}}\in\rho(A_0),
\end{equation}
\begin{equation}\label{eq:11_gamma2}
  \gamma({{z}})=\gamma({{\zeta}})+({{z}}-{{\zeta}})(A_0 - z)^{-1}\gamma({{\zeta}}),\quad  {{z}},{{\zeta}}\in\rho(A_0),
\end{equation}
and so they are holomorphic on $\rho(A_0)$ with values in $\cB(\dC^p,\wh\sN_{{z}})$ and $\cB(\dC^p,\sN_{{z}})$, respectively, see~\cite{DM25}.
Moreover, the following statement holds.
  \begin{theorem}\label{P:11.5}
Let $A$ be a  symmetric linear relation  with $n_{\pm}(A)=p<\infty$, let
$(\dC^p,\Gamma_0,\Gamma_1)$  be a boundary triple for
$A^*$, and let $M(\cdot)$ be the corresponding Weyl function. Then
\begin{enumerate}
    \item [\rm(i)] $M(\cdot)$ is a well-defined $\dC^{\ptp}$-valued function  holomorphic on $\rho(A_0)$.
    \item[\rm(ii)] For all ${{z}},{{\zeta}}\in\rho(A_0)$ the following identity holds:
  \begin{equation}\label{Eq:11.8M}
{\mathsf N}_{\zeta}^{M}({z})
=\frac{M({z})-M({\zeta})^*}{{z}-\ov{\zeta}}
=\gamma({\zeta})^*\gamma({z}).
  \end{equation}
\item[\rm(iii)] 
$M\in \cR^{\ptp}$ and  $\det\Im{M({i})}\ne0$.
  \end{enumerate}
\end{theorem}
\begin{proof}
  The proof of (i) and (ii) is based on the formula~\eqref{eq:11_A*wtA} and the identity~\eqref{eq:1.9}, (iii) follows from \eqref{Eq:11.8M}.
\end{proof}
It follows from~\eqref{Eq:11.8M} that the Weyl function $M(\cdot)$ is a $Q$-function of the linear relation $A$ in the sense of~\cite{KL71}. Recall that a symmetric linear operator $A$ is called simple if
\begin{equation}\label{eq:simple_A}
      \sH=\overline{\textup{span}}\left\{{\sN}_{{z}}:\,\,
    {z}\in\wh\rho(A)\right\}.
\end{equation}

\subsection{Generalized resolvents of a symmetric linear relation $A$}
\begin{definition}\label{def:genres}
Let $\wt A$ be a selfadjoint extension of the symmetric linear relation $A$ in a possibly larger
Hilbert space  $\wt \sH(\supseteq \sH)$. The extension $\wt A$ of
$A$ is called \textit{minimal}, if 
\begin{equation}\label{eq:MinGres}
    \wt{{\sH}}=\overline{\textup{span}}\left\{\sH+(\wt{A}-{\zeta} I_{\wt\sH})^{-1}\sH:\,\,
    {\zeta}\in\rho(\wt A)\right\}.
\end{equation}
Let $P_\sH$ be the orthogonal projection onto $\sH$ in $\wt\sH$.
The operator-valued function
\begin{equation}\label{eq:genres}
  {\zeta} \mapsto {\mathbf R}_{\zeta}:=P_\sH(\wt A-{\zeta} I_{\wt\sH})^{-1}|_\sH,\quad {\zeta}\in\rho(\wt A),
\end{equation}
is called the {\it generalized resolvent} of  $A$.
If \eqref{eq:MinGres} holds, then the representation~\eqref{eq:genres} of the generalized resolvent
${\mathbf R}_{\zeta}$
is called \textit{minimal}
and the selfadjoint relation $\wt A$ is called the
{\it minimal representing relation} of the generalized resolvent $ {\mathbf R}_{\zeta}$.
\end{definition}

  A selfadjoint extension $\wt A$ of  $A$ and its resolvent
  $(\wt{A}-{\zeta} I_{\mathfrak H})^{-1}$ acting in the same space $\sH$ are called canonical.

Every generalized resolvent $ {\mathbf R}_{\zeta}$  of $S$ admits a minimal representation \eqref{eq:genres}; it   is unique up to a unitary equivalence, see \cite{KL71}, \cite{LT77}.

In the following theorem we present a
parametrization of the set of  generalized resolvents of
the symmetric linear relation $A$ in terms of
the boun\-dary triple, Weyl function and $\gamma$-field
established originally in \cite{Kr46,KL71,LT77}.
       \begin{theorem}
       \label{krein}
 Let $A$ be a closed symmetric linear relation in a Hilbert space  with defect numbers $n_{\pm}(A)=p<\infty$,
let $\Pi=({\dC^p},\Gamma_0,\Gamma_1)$ be a boundary triple for
$A^*$,  let $M(\cdot)$ and $\gamma(\cdot)$ be the
corresponding Weyl function and the $\gamma$-field,  let $A_0=\ker\Gamma_0$, $R_{{z}}^0=(A_0-{{{z}}}I_{\sH})^{-1}$. Then
\begin{enumerate}
\item[\rm(i)]
The formula 
\begin{equation} \label{gres0}
{{\mathbf R}_{{{z}}}}=R_{{z}}^0-\gamma({{z}})(M({{z}})
     +\tau({{z}}))^{-1}\gamma(\ov{{{z}}})^*,\quad
     {{z}}\in\dC_+\cup\dC_-
\end{equation}
establishes a bijective correspondence:   ${\mathbf R}_{{z}} \longleftrightarrow \tau$   between the class
of all generalized resolvents   ${\mathbf R}_{{z}}$  of $A$ and  the  class
of all families $\tau\in \wt \cR^{\ptp}$.
\item[\rm(ii)]
${{\mathbf R}_{{z}}}=(A_{-\tau({{z}})}-{{z}}I_{\sH})^{-1}$, that is, for every $h\in\sH$, $f={\mathbf R}_{{z}}h$ is the solution of the boundary value problem with the  eigenvalue dependent boundary condition
\begin{equation}\label{eq:11.Shtr}
\text{\rm col} \{f, h\}\in A^*-{z} I_\sH,\quad
\text{\rm col} \{ \Gamma_0\wh f , \Gamma_1\wh f\}\in -\tau({{z}}),
\end{equation}
where $\wh f=\text{\rm col} \{ f , h +{z} f\}\in A^*$, and ${z}\in\dC_+\cup\dC_-$.
\end{enumerate}
      \end{theorem}
      \begin{proof}
        Statement (i) was proved in  \cite{KL71},  \cite{LT77}, for (ii) see  \cite{DM95,DM25}.
      \end{proof}

Denote by $A_\tau$ the minimal representing selfadjoint relation of the generalized resolvent $\bR_{z}$ corresponding to $\tau\in\wt\cR^{\ptp}$ via \eqref{gres0}.
\begin{definition}\label{def:Adm}
The family $\tau\in \wt \cR^{\ptp}$ associated
to $A_\tau$ via \eqref{gres0} is said to be $\Pi$-{\it admissible},
if $\mul A_\tau=\mul A$. 
\end{definition}
Let us recall some criteria of $\Pi$-admissibility of the family $\tau({z})$ from~\cite{DHMS2000}.
\begin{proposition}\label{prop:Adm_Nev}
Let $\Pi = (\dC^p, \Gamma_0, \Gamma_1)$ be a boundary triple for $A^*$
and let $A_0:=\ker \Gamma_0$ be a selfadjoint extension of $A$
and let $\tau \in\wt\cR^{p\times p}$.
Then
\begin{enumerate}
\def\labelenumi{\rm (\roman{enumi})}
\item
The family $\tau$ is $\Pi$-admissible
if and only if 
\begin{equation}
\label{Adm1}
  (M(iy)+\tau(iy))^{-1}=o(y),\quad y \uparrow \infty,
\end{equation}
and
\begin{equation}
\label{Adm2}
(M(iy)^{-1} +\tau(iy)^{-1})^{-1}=o(y),\quad y \uparrow \infty.
\end{equation}
\item
If, in addition, $A_0=A\dotplus \text{\rm col} \{0,\mul A^*\}$,
then the following equivalences hold
\begin{equation}
\label{tlimit00}
\mul A_\tau=\mul A \Longleftrightarrow \tau \in\cR^{p\times p}\ \text{ and }\ \tau(iy)=o(y),\quad y \uparrow \infty,
\end{equation}
\begin{equation}\label{tlimit01}
\mul A_\tau\subset \sH \Longleftrightarrow\tau_{\rm op}(iy)=o(y),\quad y \uparrow \infty.
\end{equation}
\end{enumerate}
\end{proposition}
\begin{proof}
  Statement (i) and the equivalence \eqref{tlimit00} were proved in~\cite{DHMS2000}
  and~\cite[Corollary 3.8]{DM95}, respectively.

To prove  \eqref{tlimit01} let us assume that
$A_0=A\dotplus \text{\rm col} \{0,\mul A^*\}$
and $\tau(z)$ admits the decomposition
\eqref{eq:tau_decomp}:
  $\tau(z)=\tau_{\rm op}(z)\oplus \text{col }\{0,\cH_2\},
$ where $\cH_2=\mul\tau(z)$, $z\in\dC\setminus\dR$,
$\cH_1=\dC^p\ominus\cH_2$, $\tau_{\rm op}\in\cR(\cH_1)$ and
$\tau_{\rm op}(iy)=o(y)$, as $ y \uparrow \infty$.

Consider a symmetric extension $A':=\Gamma^{-1} \text{col }\{0,\cH_2\}\subset A_0$.
Then $(A')^*=\Gamma^{-1} \text{col }\{\cH_1,\cH\}\subset A^*$
and the set $\Pi'=(\cH_1,\Gamma_0',\Gamma_1')$,                                                                                                                                         where
\[
\Gamma_0'\wh f=\Gamma_0\wh f,\quad
\Gamma_1'\wh f=P_{\cH_1}\Gamma_1\wh f,
\]
is a boundary triple for $(A')^*$ such that
$A_0':=\ker\Gamma_0'=A_0\supset A'$
and
$A_0'=A'\dotplus \text{\rm col} \{0,\mul A^*\ominus\mul A'\}$
see~\cite[Proposition 12.20]{DM25}.
If $A'_{\tau_{\rm op}}$ is the minimal representing selfadjoint relation of the generalized resolvent $\bR_{z}$ of $A'$ corresponding to $\tau_{\rm op}\in\cR(\cH_1)$
and the boundary triple $\Pi'$ via \eqref{gres0},
then for all $z\in\dC\setminus\dR$
\[
A'_{\tau_{\rm op}(z)}=(\Gamma')^{-1}\tau_{\rm op}(z)
=(\Gamma)^{-1}\tau(z)=A_{\tau(z)},
\]
and hence  $\bR_{z}$ is also a  generalized resolvent of $A$
with the minimal representing relation
$A_\tau=A'_{\tau_{\rm op}}$.
By \eqref{tlimit00}, $\mul A_\tau=\mul A'_{\tau_{\rm op}}=\mul A'\subset\sH$.

The converse statement is proved similarly.
\end{proof}
      \begin{corollary}\label{cor:krein}
In the assumptions of Theorem~{\rm\ref{krein}}
the formula
\begin{equation} \label{gres0CD}
{{\mathbf R}_{{{z}}}}=R_{{z}}^0-
\gamma({{z}})(C({{z}})+D({{z}})M({{z}}))^{-1}D({{z}})\gamma(\ov{{{z}}})^*,
\quad
     {{z}}\in\rho(A_0)\cap\rho(\wt A),
\end{equation}
establishes a bijective correspondence  between the class
of all generalized resolvents   ${\mathbf R}_{{z}}$  of $A$ and  the  class
of all $\cR^{\ptp}$-pairs $\begin{bmatrix} C({z}) & D({z})\end{bmatrix}$.
In particular, the formula
\eqref{gres0CD} gives a description of all canonical resolvents of $A$ when
$C$, $D$ range over the set of $\ptp$ matrices such that
\eqref{eq:rank_and sym}  holds.

For every $h\in\sH$, the vector $f={\mathbf R}_{{z}}h$ is the solution of the boundary value problem
\begin{equation}\label{eq:11.ShtrCD}
\text{\rm col} \{
  f , h
\}\in A^*-{z} I_\sH,\quad
C({z}) \Gamma_0\wh f+D({z}) \Gamma_1\wh f=0,
\end{equation}
where $\wh f=\text{\rm col} \{f, h +{z} f\}\in A^*$, and ${{z}}\in\rho(A_0)\cap\rho(\wt A)$.
      \end{corollary}

\section{Rigged Hilbert spaces and generalized resolvents}\label{sec:3}
In this section we present the construction of the rigging $\sH_+\subset\sH\subset\sH_-$ of a Hilbert space associated with a symmetric linear relation $A$ and look at the properties of
extended generalized resolvents of $A$ and extended boundary triples.
\subsection{Rigged Hilbert spaces. }\label{sec:3.1}
\begin{lemma}\label{lem:RiggedPontrSp}
Let $A$ is a closed symmetric linear relation with equal
defect  numbers $n_{\pm}(A)=p<\infty$ in a Hilbert space
$\sH$ and
let $P_{\sH_0}$ be the orthogonal projection  onto $\sH_{0}:={\cdom A}$
in the Hilbert space $\sH$.
Then there exists a pair $\sH_+$ and $\sH_-$ of Hilbert spaces with the duality $\langle\mathfrak{f},  h\rangle_{-,+}$ for $\mathfrak{f}\in\sH_-$ and $h\in\sH_+$ such that
\begin{enumerate}
  \item [(i)] $\sH_+=\dom A^*\oplus \mul A(\subset \sH)$;
  \item [(ii)] $\sH\subset \sH_-$ and for all $f\in \sH$ and $h\in\sH_+$ we have
  $\langle{f},  h\rangle_{-,+}=(f,h)_\sH$.
\end{enumerate}
Moreover, the $\sH_+$-norm in $ \dom A^*$ can be defined by
\begin{equation}\label{eq:A+Norm}
  \| h\|^2_{\sH_+}=\|h\|_\sH^2+\|P_{\sH_0}h'\|_\sH^2\quad
\mbox{for $
\text{\rm col} \{
  h , h'
\}\in A^*$}.
\end{equation}
\end{lemma}
\begin{proof}
If  $\mul A\ne\{0\}$, then the  linear relation  $A$ admits the representation
\begin{equation}\label{LR_22}
A=\gr(A_{\rm op})\widehat\oplus\,A_{\rm mul},
\end{equation}
where $A_{\rm op}$ is  a single-valued symmetric  operator in the Hilbert space $\sH_{\rm op}:=\sH\ominus\mul A$,
while $A_{\rm mul}:=\begin{bmatrix}0\\ \mul A\end{bmatrix}$ is a purely multi-valued linear relation in the subspace $\sH_{1}:=\mul A$,  called {\it the multivalued part} of~$A$.

{\bf Step 1: Construction of a rigging for the operator  $A_{\rm op}$}, see~\cite{Sh71b}, \cite{ArlTs74}.
Let us consider $A_{\rm op}$ as an operator from the Hilbert space $\sH_{0}:={\cdom A}$ to the Hilbert space  $\sH_{\rm op}$ and denote it by $A_0:\sH_0\to \sH_{\rm op}$ and let
$A_0^*\in\cC(\sH_{\rm op},\sH_{0})$ be the adjoint to the operator  $A_0\in\cC(\sH_{0},\sH_{\rm op})$. Then $A_0^*$ is a single-valued operator.
Denote by $\sH_{0,+}$ the linear space $\sH_{0,+}:=\dom A_0^*$ endowed with the inner product
 \begin{equation}\label{E:3.1a}
   ( h,g)_{0,+}:=(h,g)_{\sH}+(A_0^* h,A_0^*g)_\sH,\quad
   h,g\in\sH_{0,+}=\dom A_0^*.
\end{equation}
Since the operator $A_0^*$ is closed, $\sH_{0,+}$ is a Hilbert space
such that $\sH_{0,+}\subset\sH_{\rm op}$. If $\sH_{0}\subsetneq\sH_{\rm op}$, then $A_{\rm op}^*$
is a linear relation in $\sH_{\rm op}$:
 \begin{equation}\label{eq:S_op*}
A_{\rm op}^*=\mbox{gr }A_0^*\ \widehat\oplus\,\begin{bmatrix}
                                     0 \\
                                     \sH_2
                                   \end{bmatrix},\quad\text{where}\quad
\sH_2=\sH_{\rm op}\ominus \sH_0.
\end{equation}

By \cite[Section~1.1.1]{Ber65},
there exists a dual Hilbert space $\sH_{0,-}$ of  bounded conjugate linear functionals on $\sH_{0,+}$ with the duality
\begin{equation}\label{eq:B-+DualityM}
  \langle\mathfrak{f},  h\rangle^{(0)}_{-,+}:=\mathfrak{f}(h)\quad \textrm{for \ $\mathfrak{f}\in\sH_{0,-}$ \ and  \quad $h\in\sH_{0,+}$},
\end{equation}
between $\sH_{0,-}$ and $\sH_{0,+}$.
The embedding ${\imath}:\sH_{\rm op}\hookrightarrow \sH_{0,-}$ is realized by the identification of
any vector $f\in\sH_{\rm op}$ with the functional
\begin{equation}\label{eq:f(h)}
{\imath}f:h\in\sH_{0,+}\mapsto ({\imath}f)(h):=(f,h)_\sH,
\end{equation}
and the space $\sH_{0,-}$ can  be realized as the completion of $\sH_{\rm op}$ with respect to the norm
\[
\|{\imath}f\|_{\sH_{-}}=
\sup_{h\in\sH_{0,+}\setminus\{0\}}\frac{|(f,h)_\sH|}{{\quad}\|h\|_{\sH_{0,+}}}.
\]
For $f\in\sH$ we will identify ${\imath}f$ with $f$. This gives an inclusion $\sH_{\rm op}\subset\sH_{0,-}$.
In view of \eqref{eq:f(h)} and~\eqref{eq:B-+DualityM}
  the expression $\langle\mathfrak{f},  h\rangle_{-,+}$ can be viewed as an extension  of the inner product in the Hilbert space $\sH_{\rm op}$. The triple
\[
\sH_{0,+}\subset\sH_{\rm op}\subset\sH_{0,-}
\]
is called a   {\it rigged Hilbert space} $\sH_{\rm op}$, cf.,  \cite{Ber65}.

{\bf Step 2: Construction of a rigging for the  linear relation  $A$ (\cite{LaTe82}).}
Setting
\begin{equation}\label{eq:H0pm_decom}
  \sH_+=\sH_{0,+}\oplus\sH_1,\quad \sH_-=\sH_{0,-}\oplus\sH_1,\quad
  \text{where $\sH_1=\mul A$},
\end{equation}
 we obtain a rigged Hilbert space $\sH$:
\begin{equation}\label{eq:H_Rigging}
  \sH_+\subset\sH\subset\sH_-
\end{equation}
with the duality  between $\sH_-$ and $\sH_+$ given by
\begin{equation}\label{eq:H-+S_Duality}
\langle\mathfrak{f},  h\rangle_{-,+}:=\langle\mathfrak{f}_0,  h_0\rangle^{(0)}_{-,+}+
(f_1,h_1)_\sH
\end{equation}
for
\[
\mathfrak{f}=\mathfrak{f}_0+f_1,\quad
{h}=h_0+h_1,\quad
 \ff_0\in\sH_{0,-},  \quad h_0\in\sH_{0,+},\quad f_1,h_1\in\sH_1.
 \]
 Then it follows from \eqref{eq:B-+DualityM} and \eqref{eq:f(h)} that
 \begin{equation}\label{eq:S_duality}
\langle\mathfrak{f},  h\rangle_{-,+}=(f,h)_\sH\quad
 \text{ for all \ $f\in \sH$\, \ and \ $h\in\sH_+$}.
\end{equation}

Since $A^*=A_{\rm op}^*\widehat\oplus\,\begin{bmatrix}
                                     0 \\
                                     \mul A
                                   \end{bmatrix}$,
we obtain, by \eqref{eq:S_op*}, that every $\wh h\in A^*$ admits the representation
\[
\wh h=\begin{bmatrix}
                        h_0\\
                        h'
\end{bmatrix}=\begin{bmatrix}
                        h_0\\
                        A_0^* h_0+h_1'+h_2'
\end{bmatrix},\quad\text{where}\quad
h_0\in\sH_{0,+},\,\,h_1'\in\sH_1,\,\, h_2'\in\sH_2,
\]
and hence, by \eqref{eq:H0pm_decom} and \eqref{E:3.1a},
\begin{equation}\label{eq:S+Norm}
\|h_0\|^2_{\sH_+}=\|h_0\|^2_{\sH_{0,+}}=\|h_0\|^2_{\sH}+\|A_0^*h_0\|_{\sH_+}^2=
\|h_0\|^2_{\sH}+\|P_{\sH_0}h'\|_\sH^2.
\end{equation}
This proves \eqref{eq:A+Norm} for $A^*$.
\end{proof}

The triple $\sH_+\subset\sH\subset\sH_-$ constructed in Lemma~\ref{lem:RiggedPontrSp} is called a   {\it rigged Hilbert space}.

For a linear operator $T\in\cB(\sH_-,\dC^p)$ we denote by $T^{\langle*\rangle}\in\cB(\dC^p,\sH_+)$  its adjoint with respect to the duality $\langle \cdot, \cdot\rangle_{-,+}$, i.e.,
\begin{equation}\label{eq:<*>}
  \langle \ff, T^{\langle*\rangle}\xi\rangle_{-,+}
  =\xi^*(T\ff),\quad
\xi\in\dC^p,\quad \ff\in\sH_-.
\end{equation}
For  $T\in\cB(\sH,\dC^p)$ we set $T^{\langle*\rangle}:=T^{*}\in\cB(\dC^p,\sH)$.
For the operator   $\gamma(\ov{{z}})\in\cB(\dC^p,\sH_+)$
its adjoint $\gamma(\ov{{z}})^{\langle*\rangle}\in\cB(\sH_-,\dC^p)$ is defined by
\begin{equation}\label{eq:gamma<*>}
\xi^*(\gamma(\ov{{z}})^{\langle*\rangle}\ff):=\langle \ff, \gamma(\ov{{z}}) \xi \rangle_{-,+},
\quad
\ff\in\sH_-,\quad  \xi\in\dC^p.
\end{equation}

In what follows we also use the notation
\begin{equation}\label{eq+-duality}
  \langle h, \ff  \rangle_{+,-}:=\langle \ff, h \rangle_{-,+}^*,\quad
  \text{for}\quad h\in\sH_+,\quad \ff\in\sH_-,
\end{equation}
and the adjoint to  $L\in\cB(\dC^p,\sH_-)$ with respect to the duality $\langle \cdot, \cdot\rangle_{+,-}$ is defined as the operator $L^{\langle*\rangle}\in\cB(\sH_{+},\dC^p)$ such that
\begin{equation}\label{eq:<*>2}
\xi^*(L^{\langle*\rangle} h):=\langle h, L\xi \rangle_{+,-},
\quad  h\in\sH_+,\quad\xi\in\dC^p.
\end{equation}
\subsection{Extended generalized resolvents}
In the next two lemmas we collect some statements from \cite{Sh71b}, \cite{ArlTs74} and \cite{ShTs77}
and present their short proofs for the convenience of the reader.
\begin{lemma}\label{lem:3.2B}
Let $A$ be a  closed symmetric linear relation in a Hilbert space $\sH$ with equal defect numbers,
and let ${\mathbf R}_{z}$ be a generalized resolvent of  $A$ with a minimal representing relation $\wt A$, ${z}\in\rho(\wt A)$. Then
\begin{enumerate}
  \item [(i)] $\ran {\mathbf R}_{z}\subset \dom A^*\subseteq\sH_+$ and
  $\begin{bmatrix}
     {\mathbf R}_{z} f \\
     (I_\sH+{z}{\mathbf R}_{z})f
   \end{bmatrix}\in A^*$ for all $f\in\sH$.
  \item [(ii)] ${\mathbf R}_{z}\in\cB(\sH,\sH_+)$.
  \item [(iii)] The operator
  $\wt{\mathbf R}_{z}:={\mathbf R}_{\overline{z}}^{\langle*\rangle}\in\cB(\sH_-,\sH)$ is a continuation of
  ${\mathbf R}_{z}$ and is called the extended generalized resolvent.
   \item [(iv)] For every selfadjoint extension $A_0$ of $A$ and $R_z^0:=(A_0-zI)^{-1}$ the extended resolvent
       $\wt R_{z}^0:=(R^0_{\overline{z}})^{\langle*\rangle}$ satisfies the identity
   \begin{equation}\label{eq:HilbertId}
     \wt R_{z}^0-\wt R_\mu^0=({z}-\mu) R_{z}^0\wt R_\mu^0,\quad
     {z},\mu\in\rho (A_0),
   \end{equation}
   and hence
  $\wt R_{z}^0-\wt R_\mu^0\in\cB(\sH_-,\sH_+)$.
\end{enumerate}
\end{lemma}
\begin{proof}
  (i) For all $f\in\sH$ and $\text{\rm col} \{ h,h' \}\in A$ we obtain
 \[
   ((I_\sH+{z}{\mathbf R}_{z})f,h)_\sH - ({\mathbf R}_{z} f,h')_\sH\\
   =
   (f,h)_\sH -(f, {\mathbf R}_{\overline{z}}(h'-\overline{z} h))_\sH =0.
\]
 Here we used the equality ${\mathbf R}_{\overline{z}}(h'-\overline{z} h)=h $.
This proves (i).
  \medskip

(ii) Let $f_n\to 0$ in $\sH$. Then ${\mathbf R}_{z} f_n\to 0$ and, by
\eqref{eq:A+Norm},
 \[
  \| {\mathbf R}_{z} f_n\|^2_{\sH_+}=\|{\mathbf R}_{z} f_n\|_\sH^2+\|P_{\sH_0}(f_n+{z} {\mathbf R}_{z} f_n)\|_\sH^2\to 0.
\]

(iii) Since  ${\mathbf R}_{\overline{z}}^{\langle*\rangle}\in\cB(\sH_-,\sH)$, (iii) follows from the inclusion ${\mathbf R}_{{z}}\subset{\mathbf R}_{\overline{z}}^{\langle*\rangle}$.
  \medskip

(iv) follows from the Hilbert identity
$
R_{\overline{z}}^0-R_{\overline\mu}^0
=({\overline{z}}-{\overline\mu}) R_{\overline\mu}^0 R_{\overline{z}}^0
$
and the equalities $(R_{\overline{z}}^0)^{\langle*\rangle}=\wt R_{{z}}^0$,
$(R_{\overline\mu}^0)^{\langle*\rangle}=\wt R_{\mu}^0$.
\end{proof}
\begin{remark}
Recall, see \cite{LaTe82}, that every 
generalized resolvent ${\mathbf R}_{z}$ of $A$ can be decomposed as
\begin{equation}\label{eqGenResDecomp}
  {\mathbf R}_{z}={\mathbf R}_{z}^{\rm op}\oplus 0_{\mul A}.
\end{equation}
where ${\mathbf R}_{z}^{\rm op}$ is a generalized resolvent of the operator $A_{\rm op}$
in the Hilbert space $\sH_{\rm op}:=\sH\ominus\mul A$. Therefore the statement (ii) of Lemma~\ref{lem:3.2B} can be rewritten as
  ${\mathbf R}_{z}\in\cB(\sH,\sH_{0,+})$.
\end{remark}
By Lemma~\ref{lem:3.2B}, the operators $R^0$ and $\bR^0$ defined by
\begin{equation}\label{eq:cR}
R^0:=\frac12(R_{i}^0+R_{-i}^0),\quad \bR^0:=(R^0)^{\langle*\rangle}.
\end{equation}
have the properties
\begin{enumerate}
  \item[(1)] 
  $R^0\in\cB(\sH,\sH_+)$ and $\bR^0\in\cB(\sH_-,\sH)$;
  \item[(2)] for every  extended generalized resolvent $\wt{\mathbf R}_{z}\in\cB(\sH_-,\sH)$ we have
  $\wt{\mathbf R}_{z}-\bR^0\in\cB(\sH_-,\sH_+)$.
\end{enumerate}
The first claim follows from Lemma~\ref{lem:3.2B}(ii)--(iii) and the second from the equality
\[
\wt{\mathbf R}_{z}-\bR^0=(\wt R_{{z}}^0-\bR^0)-\gamma({z})(M({z})
     +\tau({z}))^{-1}\gamma(\ov{{z}})^{\langle*\rangle},
\]
and the relations $\gamma({z})\in\cB(\dC^p,\sH_+)$ and $\gamma(\ov{{z}})^{\langle*\rangle}\in\cB(\sH_-,\dC^p)$, see~\eqref{eq:gamma<*>}.

The operator $\bR^0$  is called the {\it regularizer} and
\begin{equation}\label{eq:RegExtRes}
  \wh{\mathbf R}_{z}:=\wt{\mathbf R}_{z}-\bR^0
\end{equation}
is called
the {\it regularized extended generalized resolvent}.
\begin{lemma}\label{Lem:3.3}
Let $A$ be a  closed symmetric linear relation with equal defect numbers,
 let ${\mathbf R}_{z}$ be an extended generalized resolvent of  $A$  with a minimal representing relation $\wt A$ and let
\begin{equation}\label{eq:bf_S}
\bA=\left\{\wh f=\begin{bmatrix}
f \\
\ff'
\end{bmatrix}\in \begin{bmatrix}
\sH \\
\sH_-
\end{bmatrix}  :\, \langle \ff', h \rangle_{-,+}=(f,h')_\sH\quad
\text{for}\quad  \begin{bmatrix}
                                               h \\
                                               h'
                                             \end{bmatrix}\in A^*
                                             \right\}.
\end{equation}
Then:
\begin{enumerate}
\item[\rm(i)] $\bA$ is a closed linear relation in $\wt\cC(\sH,\sH_-)$
with $\dom \bA=\sH_{0}:={\cdom A}$.
  \medskip

  \item [(ii)] $\bA$ is an extension of $A$ and $A=\bA\cap\sH^2$.
  \medskip

  \item [(iii)]  $\widetilde{\mathbf R}_{{z}}(\ff' -{{z}} f)=f$ for all
  $\text{\rm col} \{ f,\ff' \}\in \bA$, ${z}\in\rho(\wt A)$.
  \medskip

  \item [(iv)] $\ran({\mathbf{A}} -{{z}} I_\sH)$ is a closed subspace of $\sH_-$ for all ${{z}}\in\wh\rho(A)$.
  \medskip

  \item [(v)] $\ker({\mathbf{A}} -{{z}} I_\sH)=\{0\}$ for ${z}\in\wh\rho(A)$.
   \medskip

   \item [(vi)] The annihilator $\{\ran({\mathbf{A}} -{{z}} I_\sH)\}^\perp$ of  $\ran({\mathbf{A}} -{{z}} I_\sH)$ in $\sH_+$ coincides with $\sN_{\ov{z}}$.
\end{enumerate}
\end{lemma}
\begin{proof}
  (i) $\&$ (ii) are  immediate from \eqref{eq:bf_S}.
  \medskip

  (iii) For  $g\in\sH$ we get, by Lemma~\ref{lem:3.2B},
  $\begin{bmatrix}
     {\mathbf R}_{\overline{{z}}} g \\
     (I_\sH+\ov{z}{\mathbf R}_{\overline{{z}}})g
   \end{bmatrix}\in A^*$
   and so   for  $\text{\rm col} \{ f, \ff' \}\in \bA$ we obtain, by~\eqref{eq:bf_S},
  \[
  \begin{split}
     (\widetilde{\mathbf R}_{{z}}(\ff' -{{z}} f),g)_\sH
      &= \langle \ff' -{{z}} f,{{\mathbf R}_{\overline{{z}}}}g\rangle_{-,+}\\
      &=(f, g+\overline{{z}}{\mathbf R}_{\overline{{z}}}g)_\sH
      -{z}(f, {\mathbf R}_{\overline{{z}}}g)_\sH=
      (f,g)_\sH.
  \end{split}
  \]
  This proves (iii).
  \medskip

  (iv) Let
  $\text{\rm col} \{ f_n , \ff_n' \}\in \bA$, and let
   $\ff_n' -{{z}} f_n\to \fg$ in $\sH_-$ as $n\to\infty$. Then, by (iii) and Lemma~\ref{lem:3.2B}(iii),
   $f_n=\widetilde{\mathbf R}_{{z}}\ff_n' -{{z}} f_n \to f:=\widetilde{\mathbf R}_{{z}}\fg$  in $\sH$. Hence
   $\ff_n' \to\ff':={{z}} f+ \fg$ in $\sH_-$. By (ii),
   we get
   $\text{\rm col} \{ f, \ff' \}\in \bA$ and therefore
 $\fg=\ff'-{z} f\in \ran(\bA-{z} I_\sH)$.
    \medskip

  (v) follows from (ii).
    \medskip

  (vi)  Since $\ran(A -{{z}} I_\sH)\subset \ran({\mathbf{A}} -{{z}} I_\sH)$ for ${z}\in\wh\rho(A)$, the inclusion  $\{\ran({\mathbf{A}} -{{z}} I_\sH)\}^\perp\subseteq\sN_{\ov{z}}$ holds.
  Conversely if $h\in\sN_{\ov{z}}$, then
 $\wh h =\text{\rm col} \{ h, \ov{{z}}h \}\in A^*$
 and for all $\text{\rm col} \{ f, \ff' \}\in \bA$ we get, by \eqref{eq:bf_S},
  \[
 0= \langle \ff'-{z} f, h \rangle_{-,+}=(f,\ov{{z}}h)_\sH-{z}(f,h)_\sH=0
 \]
  Therefore, $ h\in\{\ran({\mathbf{A}} -{{z}} I_\sH)\}^\perp$.
\end{proof}

Let the linear relation $A^{\langle*\rangle}$ be given by
\begin{equation}\label{eq:bf_A*}
A^{\langle*\rangle}=\left\{\wh f=\begin{bmatrix}
f \\
\ff'
\end{bmatrix}\in \begin{bmatrix}
\sH \\
\sH_-
\end{bmatrix}  :\, \langle \ff', h \rangle_{-,+}=(f,h')_\sH\quad
\text{for}\quad  \begin{bmatrix}
                                               h \\
                                               h'
                                             \end{bmatrix}\in A\right\}.
\end{equation}
Similarly, for a selfadjoint extension $A_0$ of $A$ we will set
\begin{equation}\label{eq:bf_A0}
\bA_0=\left\{\wh f=\begin{bmatrix}
f \\
\ff'
\end{bmatrix}\in \begin{bmatrix}
\sH \\
\sH_-
\end{bmatrix}  :\, \langle \ff', h \rangle_{-,+}=(f,h')_\sH\quad
\text{for}\quad  \begin{bmatrix}
                                               h \\
                                               h'
                                             \end{bmatrix}\in A_0\right\}.
\end{equation}
Then, by~\eqref{eq:bf_S},~\eqref{eq:bf_A*} and~\eqref{eq:bf_A0},
\begin{equation}\label{eq:bf_A_A0}
\bA\subset \bA_0\subset A^{\langle*\rangle}.
\end{equation}
\begin{lemma}\label{lem:bf_A0}
Let $A_0$  be a selfadjoint extension of $A$ in $\sH$, let $\bA_0$ be defined by~\eqref{eq:bf_A0},
and let $\widetilde{ R}^0_{{z}}$ be the extended resolvent of $A_0$,  ${{z}}\in\rho(A_0)$.
Then:
\begin{enumerate}
\item [(i)]  $\widetilde{ R}^0_{{z}}(\ff' -{{z}} f)=f$ for all
  $\text{\rm col} \{ f, \ff' \}\in \bA_0$, ${\zeta}\in\rho( A_0)$.
  \medskip

\item [(ii)]
The linear relation $\bA_0$ admits the representation
\begin{equation}\label{eq:bf_A0R}
\bA_0=\left\{\begin{bmatrix}
\wt R_{z}^0\ff \\
\ff+{z}\wt R_{z}^0\ff
\end{bmatrix}:\, \ff\in\sH_-\right\}.
\end{equation}
  \medskip

\item [(iii)]
$A^{\langle*\rangle}$, $\bA_0$ and $\wh\sN_{{z}}$ are closed subspaces of $\wh\sH:=\sH\times\sH_-$ and for ${z}\in\rho(A_0)$
\begin{equation}\label{eq:A+_decom0}
    A^{\langle*\rangle}=\bA_0\dotplus\wh\sN_{{z}}.
\end{equation}
\end{enumerate}
\end{lemma}
\begin{proof}
(i) is proved similarly to  Lemma~\ref{Lem:3.3}(iii).
\medskip

(ii) Notice first that for every ${\zeta}\in\rho(A_0)$
\begin{equation}\label{eq:A0R}
A_0=\left\{\begin{bmatrix}
 R_{\zeta}^0h \\
h+{\zeta} R_{\zeta}^0h
\end{bmatrix}:\, h\in\sH\right\}.
\end{equation}
Then for all $\ff\in\sH_-$ and $h\in\sH$ we get, using the Hilbert identity
\begin{multline*}
\langle \ff+{z}\wt R_{z}^0\ff, R_{\zeta}^0h \rangle_{-,+}
-(\wt R_{z}^0\ff,h+{\zeta} R_{\zeta}^0h)_\sH\\
=\langle \ff, R_{\zeta}^0h+\ov {z} R_{\ov{z}}^0R_{\zeta}^0h \rangle_{-,+}-\langle \ff,R_{\ov{z}}^0h+{\zeta} R_{\ov{z}}^0R_{\zeta}^0h\rangle_{-,+}\\
=\langle \ff, R_{\zeta}^0h-R_{\ov{z}}^0h-({\zeta}-\ov {z}) R_{\ov{z}}^0R_{\zeta}^0h \rangle_{-,+}=0.
\end{multline*}
Therefore, $\wh g=\begin{bmatrix}
\wt R_{z}^0\ff \\
\ff+{z}\wt R_{z}^0\ff
\end{bmatrix}\in\bA_0$ for all $\ff\in\sH_-$.

Conversely, let $\wh g=\text{\rm col} \{ g, \fg' \}\in\bA_0$. By (i), we get for every ${z}\in\rho(A_0)$
\begin{equation}\label{eq:A0R2}
\wt R_{z}^0(\fg'-{z} g)=g.
\end{equation}
Setting $\ff=\fg'-{z} g$ we obtain $g=\wt R_{z}^0\ff$ and hence
$
\fg'=\ff+{z} g=\ff+{z}\wt R_{z}^0\ff.
$
Therefore, the vector $\wh g$ admits the representation $\wh g=\begin{bmatrix}
\wt R_{z}^0\ff \\
\ff'+{z}\wt R_{z}^0\ff'
\end{bmatrix}$, which proves~\eqref{eq:bf_A0R}.

(iii) For
$\wh f=\text{\rm col} \{ f, \ff' \}\in A^{\langle *\rangle}$ and ${z}\in\rho(A_0)$ let us set
\[
g=\wt R_{z}^0(\ff'-{z} f),\quad
\fg'=(\ff'-{z} f)+{z}\wt R_{z}^0(\ff'-{z} f).
\]
Then, by Lemma~\ref{lem:bf_A0} (ii), $\wh g=\text{\rm col} \{ g, \fg' \}\in \bA_0$ and $ f- g\in\sN_{z}$ since for all
$\text{\rm col} \{ h, h' \}\in A$
\[
(f-\wt R_{z}^0(\ff'-{z} f),h'-\ov{z} h)_\sH
=(f,h'-\ov{z} h)_\sH-\langle \ff'-{z} f, h \rangle_{-,+}=0.
\]
Moreover,
\[
\wh f-\wh g=\begin{bmatrix}
f -g\\
\ff'-\fg'
\end{bmatrix}=\begin{bmatrix}
f -\wt R_{z}^0(\ff'-{z} f)\\
{z} f-{z} \wt R_{z}^0(\ff'-{z} f)
\end{bmatrix}\in\wh\sN_{z}.
\]
This proves the inclusion $A^{\langle*\rangle}\subseteq\bA_0\dotplus\wh\sN_{{z}}$.
The converse follows from \eqref{eq:bf_A_A0}.
\end{proof}
\subsection{Extended spectral functions}\label{sec:3.3}
Let  $\wt A $ be a selfadjoint extension of $A$ in a Hilbert space $\wt \sH=\sH\oplus\sH_1$.
Denote by $\wt A_{\rm op}$
the operator part of $\wt A$ acting in the space $\wt \sH_{\rm op}:=\wt \sH\ominus\mul\wt A$
and let $P_{\wt \sH_{\rm op}}$ be the orthogonal projection from $\wt \sH$ onto $\wt \sH_{\rm op}$.
Clearly, $\wt A_{\rm op}$ is a selfadjoint extension of the operator $A_{\rm op}:\sH_{\rm op}=\cdom A\to\sH\ominus\mul A$.
Let $E_{\wt A_{\rm op}}(\Delta)$ be the spectral function   of $\wt A_{\rm op}$.
The spectral function $\wt E(\Delta)$ of the selfadjoint relation $\wt A$ is defined by
$\wt E(\Delta):=E_{\wt A_{\rm op}}(\Delta)P_{\wt \sH_{\rm op}}$,
and the operator function $E(\Delta):=P_{\sH}\wt E(\Delta)\upharpoonright\sH$
is  called the {\it generalized spectral function} of the selfadjoint relation $\wt A$, see~\cite{Sh71b}.
\begin{lemma}[\cite{Sh71b}]\label{lem:ExtSpectralF}
Let $\Delta$ be a finite interval in  $\dR$. Then
\begin{enumerate}
  \item [\rm (i)]  $P_\sH\ran \wt E(\Delta)\subset\sH_+$, $\ran E(\Delta)\subset\sH_+$,  and
 \begin{equation}\label{eq:EwtE}
  \wt E(\Delta)\in\cB(\sH,\sH_+\oplus\sH_1), \quad E(\Delta)\in\cB(\sH,\sH_+).
\end{equation}
  \item [\rm (ii)]  $\ran E(\Delta)A_{0}\subset\sH_+$.

  \item [\rm (iii)] $\wh E(\Delta):=E(\Delta)^{\langle *\rangle}\in\cB(\sH_-,\sH)$.
    \medskip
  \item [\rm (iv)] If $\text{\rm col} \{ f, \ff' \}\in {\mathbf A}$, then
  $\wt A_{\rm op}\wt E(\Delta)f=\wt E(\Delta)^{\langle *\rangle}f'$.
    \medskip
  \item [\rm (v)] For ${\mathfrak f}\in\sH_-$, $g\in\sH$ it holds $(\wh E(\Delta){\mathfrak f},g)\in\sS^{(1)}$ $($see definition in Section~\ref{sec:2.2}$)$.
        \medskip
  \item [\rm (vi)] For ${\mathfrak f},{\mathfrak g}\in\sH_-$ it holds
  $\langle\wh E(\Delta){\mathfrak f},{\mathfrak g}\rangle_{+,-}\in\sS^{(2)}$.
\end{enumerate}
\end{lemma}
\begin{proof}
(i) It is enough to prove the first inclusion in (i).
Notice first that $P_\sH\ran \wt E(\Delta)\subset \sH_{\rm op}$ since for
 $h\in\wt\sH$ we have $\wt E(\Delta) h\perp \mul \wt A\supset\mul A$.
Next for $f\in\dom A_{\rm op}$ 
we obtain
\[
\begin{split}
(A_{0}f,P_\sH\wt E(\Delta)h)_{\sH}&=
(A_{\rm op}f,\wt E(\Delta)h)_{\wt\sH}\\
&=(f,\wt A_{\rm op}\wt E(\Delta)h)_{\wt\sH}=
(f,P_{\sH_0}\wt A_{\rm op}\wt E(\Delta)h)_{\sH_0}.
\end{split}
\]
Therefore,
\begin{equation}\label{eq:AopEDelta}
 \text{$P_\sH\ran \wt E(\Delta)\subset\dom(A_{0}^*)\subset\sH_+$ and
$A_{0}^*P_{\sH}\wt E(\Delta)h=P_{\sH_{0}}\wt A_{\rm op}\wt E(\Delta)h$.}
\end{equation}
Moreover, $E(\Delta)\in\cB(\sH,\sH_+)$ since the operator $A_{0}^*E(\Delta)$ is bounded.

\noindent
(ii) The inclusion $\ran  E(\Delta)A_{0}\subset\sH_+$ follows from the equalities
\[
E(\Delta)A_{0} f= P_\sH\wt E(\Delta)\wt A_{\rm op} f= P_\sH\wt A_{\rm op} \wt E(\Delta)f,
\]
\[
A_{0}^*E(\Delta)A_{0} f=P_{\sH_{0}}\wt A_{\rm op}\wt E(\Delta)A_{\rm op} f
=P_{\sH_{0}}\wt A_{\rm op}^2\wt E(\Delta) f,
\]
where $f\in\dom A$.
  \medskip

\noindent
(iii) follows from definition \eqref{eq:<*>} of $E(\Delta)^{\langle *\rangle}$ and the item (i).
  \medskip

\noindent
(iv) For $h\in\dom \wt A_{\rm op}$ we get, by using \eqref{eq:AopEDelta} and the definition \eqref{eq:bf_S} of ${\mathbf A}$,
\begin{multline*}
   (\wt A_{\rm op} h, \wt E(\Delta)f)_{\wt\sH} = (\wt E(\Delta)\wt A_{\rm op} h, f)_{\wt\sH}
     =(P_{\sH_0}\wt A_{\rm op} \wt E(\Delta)h, f)_{\sH} \\
     =(A_0^*P_{\sH}\wt E(\Delta)h,f)_{\sH}
      =\langle P_{\sH}\wt E(\Delta)h,f'\rangle_{+,-}\\
      =\langle \wt E(\Delta)h,f'\rangle_{\sH_+\oplus\sH_1,\sH_-\oplus\sH_1}=\langle h,\wt E(\Delta)^{\langle *\rangle}f'\rangle_{\wt\sH}.
\end{multline*}
Since $\wt A_{\rm op}=\wt A_{\rm op}^*$, this proves the equality $\wt A_{\rm op}\wt E(\Delta)f=\wt E(\Delta)^{\langle *\rangle}f'$.
  \medskip

 \noindent
(v) \& (vi) are proved similarly to ~\cite{Sh71b}. Let us only note that
the expression $\langle\wh E(\Delta){\mathfrak f},{\mathfrak g}\rangle_{+,-}$
has a sense if ${\mathfrak f},{\mathfrak g}\in\sH_-$. Indeed, for $h\in\dom A_0$
we get, by (ii),
\[
|( A_{0} h, \wh E(\Delta){\mathfrak f})_{\wt\sH}|=
|\langle E(\Delta) A_{0} h, {\mathfrak f}\rangle_{+,-}|
\le \|{\mathfrak f}\|_-\|E(\Delta) A_{0} h\|_+\le
C\|{\mathfrak f}\|_-\|h\|_\sH
\]
and hence
the functional $(A_{0} \cdot, \wh E(\Delta){\mathfrak f})_{\wt\sH}$
is bounded on $\sH_0$. Therefore, $\wh E(\Delta){\mathfrak f}\in\dom A_0^*$,
$\| A_{0}^*\wh E(\Delta){\mathfrak f}\|_{\sH}\le
C\|{\mathfrak f}\|_-$
and hence, $\wh E(\Delta)\in\cB(\sH_-,\sH_+)$.
\end{proof}

The operator-function $\wh E(\Delta):=E(\Delta)^{\langle *\rangle}$ is called the
{\it extended generalized spectral function} of the selfadjoint relation $\wt A$.
Notice that the generalized resolvent ${\mathbf R}_z$ with  the minimal representing relation $\wt A$
has the integral representation
\begin{equation}\label{eq:gen_res_int}
  {\mathbf R}_z=\int_{\dR}\frac{dE_{\lambda}}{\lambda-z},\quad z\in\dC_+\cup\dC_-,
\end{equation}
while the extended regularized generalized resolvent $\wh {\mathbf R}_z$ (see~\eqref{eq:RegExtRes}) takes the form
\begin{equation}\label{eq:gen_res_int}
 \wh {\mathbf R}_z=\int_{\dR}\left(\frac{1}{\lambda-z}-\frac{\lambda}{\lambda^2+1}\right)d\wh E_{\lambda}+K,\quad z\in\dC_+\cup\dC_-,
\end{equation}
where $K=K^*\in\cB(\sH_-,\sH_+)$, see~\cite{Sh71b}.
\subsection{Extended boundary triples}
\begin{lemma}\label{lem:bf_A0Pi}
Let $\Pi=(\dC^p,\Gamma_0,\Gamma_1)$ be a boundary triple for $A^*$,
let the ovf's 
$\gamma$ and  $\wh\gamma$
be given by~\eqref{eq:11_gamma1}
and let
$R_{z}^0=(A_0-{z} I)^{-1}$, ${z}\in\rho(A_0).$
Then the mapping $\Gamma$ admits a continuation to a bounded mapping
$\wh\Gamma: A^{\langle *\rangle}\to\dC^{2p}$ given by
\begin{equation}\label{eq:wh_Gamma}
 \wh\Gamma_0 \begin{bmatrix}
\wt R_{z}^0\ff  \\
\ff +{z}\wt R_{z}^0\ff
\end{bmatrix}= 0, \quad
 \wh\Gamma_1 \begin{bmatrix}
\wt R_{z}^0\ff  \\
\ff +{z}\wt R_{z}^0\ff
\end{bmatrix}= \gamma(\ov{z})^{\langle *\rangle}\ff , \quad \ff \in\sH_-,
\end{equation}
\[
\wh\Gamma\wh\gamma({z})u=\Gamma\wh\gamma({z})u,\quad u\in\dC^p.
\]
Moreover, $\wh\Gamma$ satisfies the relations
\begin{equation}\label{ker_wh_Gamma}
  \ker\wh\Gamma_0=\bA_0,\quad \ker\wh\Gamma_1\cap\ker\wh\Gamma_0 =\bA.
\end{equation}
\end{lemma}
\begin{proof}
  1) Since $A^{\langle*\rangle}$, $\bA_0$ and $\wh\sN_{{z}}$ are closed subspaces of
   $\wh\sH:=\sH\times\sH_-$ and the operators $\Gamma_j\upharpoonright\wh\sN_{z}$, $j=0,1$, are bounded from $\wh\sH$ to $\dC^p$, it is enough to prove that
    the operators $\Gamma_j\upharpoonright\bA_0$, $j=0,1$, are bounded from $\wh\sH$ to $\dC^p$. For $j=0$ this statement is clear, so let us prove that
     $\Gamma_1\upharpoonright\bA_0$ is bounded as an operator from $\wh\sH$ to $\dC^p$.

Assume that
\[
\wt R_{z}^0\ff_n\stackrel{\sH}{\longrightarrow}g,\quad
\ff_n +{z}\wt R_{z}^0\ff_n\stackrel{\sH_-}{\longrightarrow}\fg'\quad\text{as}\quad n\to\infty
\]
with $\ff_n, \fg'\in\sH_-$, $g\in\sH$. Then
$\ff_n\stackrel{\sH_-}{\longrightarrow}\ff:=\fg'-{z} g$ and
$\gamma(\ov{z})^{\langle *\rangle}\ff_n{\longrightarrow}
\gamma(\ov{z})^{\langle *\rangle}\ff$.  This proves that
$\Gamma_1\upharpoonright\bA_0$ is bounded as an operator from $\wh\sH$ to $\dC^p$.

2) The first equality in~\eqref{ker_wh_Gamma} is clear. Assume that
$\wh\Gamma_0\fg=\wh\Gamma_1\fg=0$ for some $\wh g\in A^{\langle*\rangle}$.
Then $\wh g\in \bA_0$ and, by Lemma~\ref{lem:bf_A0} (ii),
\[
\wh g=\begin{bmatrix}
\wt R_{z}^0\ff  \\
\ff +{z}\wt R_{z}^0\ff
\end{bmatrix}\quad\text{for some}\quad \ff\in\sH_-,\quad {z}\in\rho(A_0).
\]
By~\eqref{eq:wh_Gamma}, we get the equality
$\wh\Gamma_1\wh g=\gamma(\ov{z})^{\langle *\rangle}\ff =0$
which, by Lemma~\ref{Lem:3.3}(vi), implies that $\ff\in\ran(\bA-{z} I_\sH)$.
In view of  Lemma~\ref{Lem:3.3} (iii), we get
$\wh g
\in\bA$.
\end{proof}
\noindent
The triple $(\dC^p,\wh\Gamma_0,\wh\Gamma_1)$ will be called an {\it extended boundary triple for} $A^{\langle*\rangle}$.
\begin{corollary}\label{cor:ExtGreen}
For all $\wh f=\text{\rm col} \{ f, \ff' \}\in A^{\langle *\rangle}$ and
$\wh g=\text{\rm col} \{ g, \fg' \}\in A^*$  it holds
\begin{equation}\label{eq:ExtGreen}
  \langle \ff', g \rangle_{-,+}-( f, g')_\sH
  =(\Gamma_0\wh g)^*(\wh\Gamma_1\wh f)-(\Gamma_1\wh g)^*(\wh\Gamma_0\wh f).
\end{equation}
\end{corollary}
In the following lemma we present an extended version of the abstract Green formula.
\begin{lemma}\label{lem:ExtGreenForm}
Let  $\wh\Pi=  (\dC^p,\wh\Gamma_0,\wh\Gamma_1) $ be an extended boundary triple for
$A^{\langle*\rangle}$,
let $\bR^0$ be the regularizer given by~\eqref{eq:cR},
let
$\wh f=\text{\rm col} \{ f, \ff' \}$,
               $\wh g=\text{\rm col} \{ g, \fg' \}\in A^{\langle*\rangle}$
and let ${\wt f},{\wt g}\in\sH_-$ be such that $\ff'+{\wt f}, \fg'+{\wt g}\in \sH$.
Then
\begin{enumerate}
  \item [(i)] $f+\bR^0 {{\wt f}}\in\sH_+$ and $g+\bR^0 {{\wt g}}\in\sH_+$;
  \item [(ii)] the following equality holds:
\begin{multline}\label{eq:ExtGreen2}
  (\ff'+{{\wt f}},g)_\sH-(f, \fg'+{{\wt g}})_\sH
  =(\wh\Gamma_0\wh g)^*(\wh\Gamma_1\wh f)-(\wh\Gamma_1\wh g)^*(\wh\Gamma_0\wh f)\\
  -\left\langle f+\bR^0 {{\wt f}}, {{\wt g}}\right\rangle_{+,-}
  +\left\langle {{\wt f}},g+\bR^0 {{\wt g}}\right\rangle_{-,+}.
\end{multline}
\end{enumerate}
\end{lemma}
\begin{proof}
(i) By \eqref{eq:A+_decom0}, for ${z}\in \rho({A}_0)$ the vector $\wh g$ admits the representation
\[
\begin{bmatrix}
                 g \\
                 \fg'
               \end{bmatrix}=
\begin{bmatrix}
                 g_0 \\
                 \fg_0'
               \end{bmatrix}
               +\begin{bmatrix}
                 g_{z} \\
                {z} g_{z}
               \end{bmatrix}\quad
\text{for some $\wh g_0=\begin{bmatrix}
                 g_0 \\
                 \fg_0'
               \end{bmatrix}\in\bA_0$ and
$\wh g_{z}=\begin{bmatrix}
                 g_{z} \\
                {z} g_{z}
               \end{bmatrix}\in\wh\sN_{z}$}.
\]
Then
\[
h:=\fg'-{z} g +{{\wt g}}=\fg_0'-{z} g_0 +{{\wt g}}\in\sH.
\]
By Lemma~\ref{lem:bf_A0} (i), $ R_{z}^0h=g_0+\wt R_{z}^0{{\wt g}}$ and hence
\[
g+\bR^0 {{\wt g}}=g_0+\bR^0 {{\wt g}}+g_{z}= R_{z}^0h-(\wt R_{z}^0-\bR^0) {{\wt g}}
+g_{z}\in\sH_+.
\]

\noindent
(ii)
If $\wh g=\wh g_{z}\in\wh\sN_{z}$, the equality~\eqref{eq:ExtGreen2} is reduced to~\eqref{eq:ExtGreen}. Therefore, it is enough to prove~\eqref{eq:ExtGreen2} for
$\wh g\in\bA_0$. Let us choose $\wh g_n=\text{\rm col} \{ g_n, g_n' \}\in A_0$ such that $g_n\stackrel{\sH}{\longrightarrow}g$, $g_n'\stackrel{\sH_-}{\longrightarrow}\fg'$.
Then, by Lemma~\ref{lem:bf_A0Pi}, $\Gamma_j\wh g_n\to \wh\Gamma_j\wh g$ for $j=0,1$, and,
 by~\eqref{cor:ExtGreen}, 
 we get
  \begin{equation}\label{eq:ExtGreen3}
  (\Gamma_0\wh g_n)^*(\wh\Gamma_1\wh f)-(\Gamma_1\wh g_n)^*(\wh\Gamma_0\wh f)
  =\langle\ff',g_n\rangle_{-,+}-(f, g_n')_\sH=C_n+D_n,
  \end{equation}
  where
  \begin{equation}\label{eq:ExtGreenCn}
C_n=(\ff'+{{\wt f}},g_n)_\sH-\langle f+\bR^0 {{\wt f}},  g_n'+{{\wt g}}\rangle_{-,+}
  \end{equation}
and
  \begin{equation}\label{eq:ExtGreenDn}
  \begin{split}
D_n&=\langle f+\bR^0 {{\wt f}}, {{\wt g}}\rangle_{+,-}+
(\bR^0 {{\wt f}},g_n')_\sH-\langle  {{\wt f}}, g_n\rangle_{-,+}\\
&=\langle f+\bR^0 {{\wt f}}, {{\wt g}}\rangle_{+,-}+\langle {{\wt f}},R^0 g_n'- g_n\rangle_{-,+}.
\end{split}
  \end{equation}
Since $\begin{bmatrix}
                 g_n \\
                 g_n'-{z} g_n
               \end{bmatrix}\in A_0-{z} I$
and $\begin{bmatrix}
                 g \\
                 \fg'-{z} g
               \end{bmatrix}\in \bA_0-{z} I$
we get, by Lemma~\ref{lem:bf_A0}(i),
\[
R_{z}^0( g_n'-{z} g_n)=g_n,\quad
\wt R_{z}^0( \fg'-{z} g)=g,
\]
and hence, by Lemma~\ref{lem:3.2B}(ii),
\[
     R_{z}^0 g_n'-g_n={z}  \wt R_{z}^0 g_n\stackrel{\sH_+}{\longrightarrow}{z}  R_{z}^0 g
     = \wt R_{z}^0 \fg'-g\in\sH_+ ,\quad\text{for all ${z}\in\rho(A_0)$}.
\]
Therefore,
$
R^0 g_n'-g_n \stackrel{\sH_+}{\longrightarrow}\bR^0 \fg'-g$ as $ n\to\infty$
and so
\[
D_n\to\langle f+\bR^0 {{\wt f}}, {{\wt g}}\rangle_{+,-}+\langle {{\wt f}}, \bR^0 \fg'-g\rangle_{-,+}
\quad\text{as}\quad n\to\infty.
\]
Passing to the limit in \eqref{eq:ExtGreenCn} and \eqref{eq:ExtGreen3} as $n\to\infty$ we obtain
\[
C_n\to
(\ff'+{{\wt f}},g)_\sH-(f, \fg'+{{\wt g}})_\sH-(\bR^0 {{\wt f}},  \fg'+{{\wt g}})_\sH
\quad\text{as}\quad n\to\infty
\]
and thus
  \begin{multline}\label{eq:ExtGreen4}
(\wh\Gamma_0\wh g)^*(\wh\Gamma_1\wh f)-(\wh\Gamma_1\wh g)^*(\wh\Gamma_0\wh f)
  =(\ff'+{{\wt f}},g)_\sH-(f, \fg'+{{\wt g}})_\sH\\
 +\langle f+\bR^0 {{\wt f}}, {{\wt g}}\rangle_{+,-}
 +\langle {{\wt f}}, \bR^0 \fg'-g\rangle_{-,+}
 -\langle {{\wt f}}, \bR^0 \fg'+\bR^0 \wt g\rangle_{-,+}\\
=(\ff'+{{\wt f}},g)_\sH-(f, \fg'+{{\wt g}})_\sH
+\langle f+\bR^0 {{\wt f}}, {{\wt g}}\rangle_{+,-}
-\left\langle {{\wt f}},g+\bR^0 {{\wt g}}\right\rangle_{-,+}.
\end{multline}
The latter equality is equivalent to \eqref{eq:ExtGreen2}.
\end{proof}

\section{${\sL}$-resolvents}\label{sec:4.1}
Here we recall the notion of the ${\sL}$-resolvent  of a symmetric linear relation $A$ in a Hilbert space 
 $\sH$ and present their description
 via the so-called  ${\sL}$-resolvent matrix.
 In the case of  a Hilbert space symmetric operator $A$ with $n_\pm(A)=1$ these notions
were  originally introduced by   M.G. Kre\u{\i}n~\cite{Kr44}.
\subsection{Gauge ${\sL}$ and ovf's $\cP({z})$, $\cQ({z})$.}
The notion of an ${\sL}$-regular point of a Hilbert space symmetric operator $A$  with respect to a proper gauge ${\sL}$,
and the ovf's $\cP({z})$, $\cQ({z})$ of the 1st and 2nd kind
were introduced by M.G. Kre\u{\i}n~\cite{Kr49} in the case $n_\pm(A)=p<\infty$,
 and by S.~Saakyan in~\cite{Sa65} for $n_\pm(A)=\infty$.
In  the case of an improper gauge ${\sL}$
that is contained in the space  $\sH_-$ of distributions associated with $A$ via Lemma~\ref{lem:RiggedPontrSp},
these notions were introduced by Yu.L.~Shmul'yan and E.R.~Tse\-ka\-novskii~\cite{Sh71}, \cite{ShTs77}, and for a linear relation $A$ by H. Langer and B. Tex\-to\-rius \cite{LaTe82}.
In this section we review some results from ~\cite{Sh71}, \cite{ShTs77}, \cite{LaTe82}
and investigate relations between the theories of $\sL$-resolvent matrices
 and of boundary triples which for operators $A$ with  proper gauges were presented in \cite{DM91,DM25}.

\begin{definition}\label{def:L_reg}
Let  $\sH_+\subset\sH\subset\sH_-$ be a rigged Hilbert space associated with
a symmetric linear relation $A$,
and let ${\sL}$ be a closed subspace of $\sH_-$.
\begin{enumerate}
\item[\rm(i)]  A point ${z}\in\wh\rho(A)$ is called ${\sL}$-{\it regular} for the operator $A$, if
\begin{equation}\label{eq:Lres2}
    \sH_-=\ran(\bA-{{z}}I_\sH)\dotplus{\sL}.
\end{equation}
The set of all ${\sL}$-regular points of the operator $A$ is denoted by $\rho(A,{\sL})$.
\item[\rm(ii)] For ${z}\in\rho(A,{\sL})$ denote by $\Pi_{{\sL}}^{z}$ the skew projection in $\sH_-$ onto ${\sL}$ parallel to
$\ran(\bA-{{z}}I_\sH)$.
\end{enumerate}
A  closed subspace ${\sL}\subset\sH_-$ with the property that $\rho(A,{\sL})\ne\emptyset$ will be called a {\it gauge} of $A$.
\end{definition}
\begin{lemma}\label{lem:PQ_prop}
Let  $\sH_+\subset\sH\subset\sH_-$ be a rigged Hilbert space associated with
a symmetric linear relation $A$,
let $\bR^0$ be the regularizer given by~\eqref{eq:cR},
let ${\sL}$ be a subspace of $\sH_-$ of dimension $p$,
let ${L}\in\cB({\dC^p},{\sL})$ be invertible,
and let
${z}\in\wh\rho(A)$. Then
\begin{enumerate}
\item[\rm(i)]
${z}\in\rho(A,{\sL})\,\Longleftrightarrow  \,$ ${L}^{\langle*\rangle}\!\upharpoonright\!\sN_{\overline{z}}:
\sN_{\overline{z}}\to\dC^p$ is invertible.
\medskip
\item[\rm(ii)]
The operator function
\begin{equation}\label{eq:P_lambda}
   \cP({z})={L}^{-1}\Pi^{z}_{\sL}:\sH_-\to\dC^p,\quad {z}\in\rho(A,\sL),
\end{equation}
 takes values in $\cB(\sH_-,{\dC^p})$ and
$\cP({z})^{\langle*\rangle}=({L}^{\langle*\rangle}\!\upharpoonright\!\sN_{\overline{z}})^{-1}
\in\cB({\dC^p},\sH_+)$.
\medskip
 \item[\rm(iii)] For  ${z}\in\rho(A,{\sL})$ we have
\begin{equation}\label{eq:PQ_pr1}
  \cP({z}){L}=I_p,\quad
  {L}^{\langle*\rangle} \cP({z})^{\langle*\rangle}=I_{p}.
  \end{equation}

\item[\rm(iv)]
If ${\mathbf R}_{{z}}$ is a generalized resolvent of $A$ and $\wh{\mathbf R}_{{z}}$  is a regularized extended generalized resolvent 
holomorphic at ${z}\in\rho(A,{\sL})$, then
the ovf
\begin{equation}\label{eq:Q_lambda}
   \cQ({z})={L}^{\langle*\rangle}
   ({\mathbf R}_{{z}}- \wh{\mathbf R}_{{z}}\Pi_{{\sL}}^{z})
   :\sH\to\dC^p,\quad {z}\in\rho(A,\sL),
\end{equation}
 takes values in $\cB(\sH,{\dC^p})$ and does not depend on the choice of generalized resolvent ${\mathbf R}_{{z}}$. Let us set
 $\cQ({z})^{\langle*\rangle}:=\cQ({z})^*\in\cB(\dC^p,\sH)$.
\medskip
\item[\rm(v)] For  ${z}\in\rho(A,{\sL})$  the operator $\cQ({z})-{L}^{\langle*\rangle}R^0\in\cB(\sH,{\dC^p})$ admits a continuation $\wt\cQ({z})={L}^{\langle*\rangle}\wh{\mathbf R}_{{z}}(I-{L}\cP({z}))\in\cB(\sH_-,{\dC^p})$ and
\begin{equation}\label{eq:PQ_pr2}
  \wt\cQ({z}){L}=O_{p},\quad {L}^{\langle*\rangle}\wt\cQ({z})^{\langle*\rangle}={L}^{\langle*\rangle} (\cQ({z})^{\langle*\rangle}-\bR^0 {L})=O_{p}.
\end{equation}
\end{enumerate}
\end{lemma}
\begin{proof}
  (i)--(iii) For ${z}\in\rho(A,{\sL})$ we have $\cP({z}){L}=I_{p}$ and hence
  ${L}^{\langle*\rangle} \cP({z})^{\langle*\rangle}=I_{p}$. Since $\cP({z})^{\langle*\rangle}{\dC^p}\subseteq \sN_{\overline{z}}$, this implies that
  actually $\cP({z})^{\langle*\rangle}{\dC^p}=\sN_{\overline{z}}$ and
  the operator ${L}^{\langle*\rangle}\!\upharpoonright\!\sN_{\overline{z}}\in\cB(\sN_{\overline{z}},{\dC^p})$ 
  is invertible.
\medskip

(iv) The definition~\eqref{eq:Q_lambda} of $Q({z})$ is correct since for $f\in\sH$ we have
${\mathbf R}_{{z}}f\in\sH_+$ and
\[
({\mathbf R}_{{z}}- \wh{\mathbf R}_{{z}}\Pi_{{\sL}}^{z})f=
 \wh{\mathbf R}_{{z}}(I-\Pi_{{\sL}}^{z})f+R^0 f\in\sH_+.
\]
If ${\mathbf R}_{{z}}^{(1)}$ is another generalized resolvent of $A$ and
$\wh{\mathbf R}_{{z}}^{(1)}$  is the corresponding regularized extended generalized resolvent 
holomorphic at ${z}\in\rho(A,{\sL})$ and if
\[
f=(h'-{z} h)+\fg\quad\text{for some}\quad
\text{\rm col} \{ h,h' \}\in\bA
\quad\text{and}\quad \fg\in{\sL},
\]
then, by Lemma~\ref{Lem:3.3}(iii), we get
\[
\begin{split}
 {L}^{\langle*\rangle}({\mathbf R}_{{z}}- \wh{\mathbf R}_{{z}}\Pi_{{\sL}}^{z})f
&-{L}^{\langle*\rangle}({\mathbf R}_{{z}}^{(1)}-\wh{\mathbf R}_{{z}}^{(1)}\Pi_{{\sL}}^{z})f
  ={L}^{\langle*\rangle}(\wh{\mathbf R}_{{z}}-\wh{\mathbf R}_{{z}}^{(1)})(I_{\sH}-\Pi_{{\sL}}^{z})f\\
  &= {L}^{\langle*\rangle}(\wh{\mathbf R}_{{z}}-\wh{\mathbf R}_{{z}}^{(1)})(h'-{z} h)
={L}^{\langle*\rangle}(h-h)=0.
\end{split}
\]

(v) Since $\wt\cQ({z})={L}^{\langle*\rangle}\wh{\mathbf R}_{{z}}(I-\Pi_\sL^{z})\in\cB(\sH_-,{\dC^p})$ it holds that
$\wt\cQ({z}){L}=O_{p}$ for ${z}\in\rho(A,{\sL})$.
The second equality in \eqref{eq:PQ_pr2} follows from the first and the equality
$\wt\cQ({z})=\cQ({z})-L^{\langle*\rangle}\bR^0$.
\end{proof}
In particular, by Lemma~\ref{lem:PQ_prop}(iii), we obtain the following statement.
\begin{corollary}
For all  ${z}\in\rho(A,{\sL})\cap\rho(A_0)$  the ovf $\cQ({z})$ can be calculated by
\begin{equation}\label{eq:Q0}
  \cQ({z})={L}^{\langle*\rangle}
   ( R^0_{{z}}- \wh R^0_{{z}}\,\Pi_{{\sL}}^{z}),
\end{equation}
where $R^0_{{z}}$ is the resolvent of the selfadjoint extension $A_0$ of $A$.
\end{corollary}
\begin{lemma}
For  ${{z}}\in\rho(A,{\sL})$ let us set  $\cQ({z})^{\langle*\rangle}:=\cQ({z})^*\in\cB({\dC^p},\sH)$ and
\begin{equation}\label{eq:whP_whQ}
    \wh\cP({z}):=\begin{bmatrix}
                       \cP({z}) &
                       {z}\cP({z})
                     \end{bmatrix}\quad\text{and}\quad
    \wh\cQ({z}):=\begin{bmatrix}
                       \cQ({z}) &
                       {z}\cQ({z})+{L}^{\langle*\rangle}
                     \end{bmatrix}.
\end{equation}
Then
\begin{equation}\label{eq:PQ*}
    \wh\cP({z})^{\langle*\rangle}=\begin{bmatrix}
                       \cP({z})^{\langle*\rangle} \\
                       \ov{z}\cP({z})^{\langle*\rangle}
                     \end{bmatrix}\quad\text{and}\quad
    \wh\cQ({z})^{\langle*\rangle}=\begin{bmatrix}
                       \cQ({z})^{\langle*\rangle} \\
                       \ov{z}\cQ({z})^{\langle*\rangle}+L
                     \end{bmatrix}, 
\end{equation}
and
\begin{enumerate}
  \item [(i)]  $\wh\cP({z})^{\langle*\rangle}u$ and $\wh\cQ({z})^{\langle*\rangle}u\in A^{\langle*\rangle}$ for all $u\in{\dC^p}$.
  \item [(ii)] For ${{z}}\in\rho_s(A,{\sL}):=\rho(A,{\sL})\cap\overline{\rho(A,{\sL})}$ the following   direct sum decomposition holds
\begin{equation}\label{eq:A+_decom}
    A^{\langle*\rangle}=\bA\dotplus\wh\cP({z})^{\langle*\rangle}{\dC^p}\dotplus\wh\cQ({z})^{\langle*\rangle}{\dC^p}.
\end{equation}
\end{enumerate}
\end{lemma}
\begin{proof}
(i)
The inclusion
$\wh\cP({z})^{\langle*\rangle}{\dC^p}\subseteq A^{\langle*\rangle}$ follows from Lemma~\ref{lem:PQ_prop} (ii).
Next, for $\text{\rm col} \{ h, h' \}\in A$ and $u\in{\dC^p}$ we obtain,
by~\eqref{eq:Q_lambda}, \eqref{eq:Q0} and Lemma~\ref{Lem:3.3}(iii),
\begin{multline*}
  \left\langle h,\ov{z} \cQ({z})^{\langle*\rangle}u+{L}u\right\rangle_{+,-}- (h',\cQ({z})^{\langle*\rangle}u)_\sH
   =u^*\left\{{L}^{\langle*\rangle}h-\cQ({z})(h'-{z} h)\right\} \\
      =u^*\left\{{L}^{\langle*\rangle}h-{L}^{\langle*\rangle}R_{z}^0(h'-{z} h)\right\} =0.
\end{multline*}
Hence, $\wh\cQ({z})^{\langle*\rangle}{\dC^p}\subseteq A^{\langle*\rangle}$.
\medskip

(ii)
The inclusion $A\dotplus\wh\cP({z})^{\langle*\rangle}{\dC^p}\dotplus \wh\cQ({z})^{\langle*\rangle}{\dC^p}\subseteq A^{\langle*\rangle}$ follows from (i) and~\eqref{eq:bf_A_A0}.
Let us prove the converse inclusion. Let
$\text{\rm col} \{ f, \ff' \}\in A^{\langle*\rangle}$.
Since $\ov{{z}}\in\rho(A,{\sL})$, there exist
$\text{\rm col} \{ g, \fg' \}\in \bA$ and $u\in{\dC^p}$ such that
\begin{equation}\label{eq:11.10}
  \ff'-\overline{{z}}f=\fg'-\overline{{z}}g+{L}u.
\end{equation}
It follows from \eqref{eq:11.10} and the relations 
\[
\begin{bmatrix}
f \\  \ff'-\overline{{z}}f
\end{bmatrix},\
\begin{bmatrix}
g \\  \fg'-\overline{{z}}g
\end{bmatrix},\
\begin{bmatrix}
\cQ({z})^{\langle*\rangle}u  \\  {L}u
\end{bmatrix}\in A^{\langle*\rangle}-\overline{{z}}I
\]
that
\begin{equation}\label{eq:11.11}
 \begin{bmatrix}
 f-g-\cQ({z})^{\langle*\rangle}u \\
        0
 \end{bmatrix}\in A^{\langle*\rangle}-\overline{{z}}I.
\end{equation}
Hence $ f-g-\cQ({z})^{\langle*\rangle}u\in\ker(A^{\langle*\rangle}-\overline{{z}}I)
=\sN_{\overline{{z}}}$.
By the assumption ${{z}}\in\rho(A,{\sL})$ and Lemma~\ref{lem:PQ_prop}(ii), the equality $\cP({z})^{\langle*\rangle}{\dC^p}=\sN_{\overline{{z}}}$ holds, and hence
there is a vector $v\in{\dC^p}$ such that
\begin{equation}\label{eq:11.10v}
f-g-\cQ({z})^{\langle*\rangle}u=\cP({z})^{\langle*\rangle}v.
\end{equation}
Therefore,
\[
\begin{bmatrix}
f \\
\ff'
\end{bmatrix}
=\begin{bmatrix}
g+\cP({z})^{\langle*\rangle}v+\cQ({z})^{\langle*\rangle}u\\
\fg'+\overline{{z}}\cP({z})^{\langle*\rangle}v+\overline{{z}}\cQ({z})^{\langle*\rangle}u+Lu
\end{bmatrix}
\in \bA\dotplus\wh\cP({z})^{\langle*\rangle}{\dC^p}\dotplus\wh\cQ({z})^{\langle*\rangle}{\dC^p}.
\]
Moreover, the decomposition \eqref{eq:A+_decom} is direct, since $u$, $v$, $g$ and $\fg'$ are uniquely determined
by~\eqref{eq:11.10} and~\eqref{eq:11.10v}:
\begin{equation}\label{eq:11.12}
\begin{split}
  u&=\cP(\overline{{z}})(\ff'-\ov{{z}} f),\quad
  g=R_{\overline{{z}}}^0(I-\Pi_{{\sL}}^{\overline{{z}}})(\ff'-\ov{{z}} f),\\
 v &={L}^{\langle*\rangle}(f-g),\qquad
 \fg'=\ff'-\overline{{z}}(f-g)-{L}u.
 \end{split}
\end{equation}
\end{proof}

\subsection{The ${\sL}$-preresolvent matrix}
Let $\widetilde A$ be a self-adjoint extension of $A$ in a possibly  larger Hilbert space $\wt\sH(\supseteq\sH)$. Consider $A$ as a linear relation $A'$ in $\wt\sH$. Then the adjoint linear relation to $A'$ in  $\wt\sH$ is equal to
\[
(A')^*=\left\{\begin{bmatrix}
                  f+h \\
                  f'+h'
                \end{bmatrix}:\,
                \begin{bmatrix}
                  f\\
                  f'
                \end{bmatrix}\in A^*,\, h,h'\in \wt\sH[-]\sH\right\}.
\]
Let us consider $\sH^\perp:=\wt\sH\ominus\sH$ as a subspace of the Hilbert space $\wt\sH$.
Consider $\wt\sH_+:=\mul A\oplus\dom (A')^*=\sH_+\oplus \sH^\perp$ as a  Hilbert space endowed with the  norm
\begin{equation}\label{eq:A+Norm2}
  \| \wt f\|^2_{\wt\sH_+}=\|f\|_{\sH_+}^2+\|h\|_{\sH^\perp}^2\quad
\mbox{for }\quad \wt f=f+h,\quad f\in\sH_+,\quad h\in \sH^\perp.
\end{equation}
Let $\sH_-$ be the dual space to $\sH_+$ that was  introduced in Lemma~\ref{lem:RiggedPontrSp}.
Then the Hilbert space $\wt\sH_-:=\sH_-\oplus \sH^\perp$ can be treated as a dual space to $\wt\sH_+$
with respect to the duality
\begin{equation}\label{eq:wtHpm}
\langle\ff+f^\perp, h+h^\perp\rangle^{(\wt\sH)}_{-,+}
:=\langle\ff, h\rangle_{-,+}+(f^\perp, h^\perp)_{\wt\sH}\quad
 \text{ $\ff\in \sH_-$, $h\in\sH_+$, $f^\perp,h^\perp\in\sH^\perp$}.
\end{equation}
Denote the resolvent of $\wt A$ by ${R}_{z}(\wt A):=(\wt A-{z} I_{\wt\sH})^{-1}$ (${z}\in\rho(\wt A)$) and let the extended  resolvent  $\wt{R}_{z}(\wt A)$ of $\wt A$ be defined by
\begin{equation}\label{eq:9.14BGc}
  (\widetilde{R}_{z}(\wt A)\ff,\wt h)_{\wt \sH}=
  \langle\ff, {R}_{\ov{{z}}}(\wt A)\wt h\rangle^{(\wt\sH)}_{-,+}
  \quad \textrm{for}\ \wt h\in\wt\sH,\ \wt\ff\in\wt\sH_-,\  {z}\in\rho(\widetilde A).
\end{equation}
  \begin{lemma}\label{lem:ExtRes_wtA}
  Let  $\widetilde{R}_{z}(\wt A)$ be the extended  resolvent and
  let $\wt{\mathbf R}_{{z}}$ be the extended generalized resolvent of $ A$. Then
 $\widetilde{R}_{z}(\wt A)\in\cB(\wt\sH_-,\wt\sH)$
 and $P_{\sH}\wt{R}_{z}(\wt A)\upharpoonright\sH_-=\wt{\mathbf R}_{{z}} $
 for all ${z}\in\rho(\widetilde A)$.
\end{lemma}
\begin{proof}
  By Lemma~\ref{Lem:3.3}(iii) and \eqref{eq:wtHpm}, we get for $\ff\in\sH_-$, $h\in\sH$
\[
[\wt{R}_{z}(\wt A)\ff,h]_{\wt\sH}=\langle\ff, {R}_{\ov{{z}}}(\wt A)h\rangle^{(\wt\sH)}_{-,+}
=\langle\ff, {\mathbf R}_{\overline{{z}}} h\rangle_{-,+}
=(\wt{\mathbf R}_{{z}}\ff,h)_{\sH}.
\]
\end{proof}

\begin{definition}\label{def:11.Pl-res,matr}
Let 
$\wh{\mathbf R}_{{z}}$ be a regularized extended generalized resolvent of $A$, let  ${\sL}(\subset\sH_-)$ be a gauge for $A$ and let ${L}$ be an invertible mapping from $\cB({\dC^p},{\sL})$.  The
ovf
\begin{equation}\label{eq:G_res}
r({z}):={L}^{\langle*\rangle}\wh{\mathbf R}_{{z}}{L}={L}^{\langle*\rangle}(\wt{\mathbf R}_{{z}}-\bR^0){L}
\end{equation}
is called an ${\sL}${\it-resolvent} of $A$.
  \end{definition}
  Clearly, any ${\sL}$-resolvent $r({z})$ of $A$ belongs to $\cR^{\ptp}$.

\begin{lemma}\label{lem:PreresM}
Let $A$ be a closed symmetric linear relation with equal
defect  numbers $n_{\pm}(A)=p<\infty$ in a Hilbert space
$\sH$,
let $\sL(\subset\sH_-)$ be a gauge for $A$,
let $\Pi=(\dC^p,\Gamma_0,\Gamma_1)$ be a boundary triple for
$A^*$, let $M(\cdot)$ and $\gamma(\cdot)$  be the corresponding Weyl
function and  $\gamma$-field, respectively,
let $A_0 = 
\ker\Gamma_0$, ${z}\in\rho(A_0)$
and let $\wh R_{z}^0=\wt R_{z}^0-\bR^0$ be the  regularized extended  resolvent of $A_0$.
Then  the $2p\times 2p$--mvf ${\mathfrak  A}_{\Pi{\sL}}({\cdot})$  defined by
\begin{equation}\label{eq:Lres2A}
 {\mathfrak A}_{\Pi{\sL}}({z}) = ({\mathfrak a}_{ij}({z}))_{i,j=1}^2  :=
    \begin{bmatrix} M({z}) & \gamma(\ov{z})^{\langle*\rangle}{L}\\
    {L}^{\langle*\rangle}\gamma({z})& {L}^{\langle*\rangle}\wh R_{z}^0 {L}
    \end{bmatrix},
\end{equation}
has the following properties:
\begin{enumerate}
\item[\rm(i)]
 ${\mathfrak A}_{\Pi{\sL}}\in \cR^{2p\times 2p}$ and for all ${{z}},{\zeta}\in\rho(A_0)$ it holds$:$
\[
    {\mathsf{N}}^{{\sA}}_{\zeta}({{z}})
    := \frac{{\mathfrak A}_{\Pi{\sL}}({{z}})
    - {\mathfrak A}_{\Pi{\sL}} ({\zeta})^*}{{{z}} - \overline{{\zeta}}}
    = T({\zeta})^*T({{z}}),\quad 
T({{z}})=\begin{bmatrix}
             \gamma({{z}}) &
             \wt R_{z}^0 {L}
           \end{bmatrix}.
\]
\item[\rm(ii)]
${{z}}\in\rho( A, {\sL})\Longleftrightarrow
\text{\rm det }a_{21}({\ov{z}})\ne 0\Longleftrightarrow
\text{\rm det }a_{12}({{z}})\ne 0$.
  \medskip
\end{enumerate}
\end{lemma}
\begin{proof}
(i) It follows from  \eqref{eq:Lres2A}, \eqref{eq:11_gamma2} and \eqref{Eq:11.8M} that for all ${{z}},{\zeta}\in\rho(A_0)$
\[
\begin{split}
{\mathsf{N}}^{\mathfrak A}_{\zeta}({{z}})  &=\frac{1}{{z} -\ov {{\zeta}}}
\begin{bmatrix}
M({{z}})-M(\ov{{\zeta}}) &
(\gamma (\ov{{z}})^*-\gamma ({\zeta})^*){L}\\
{L}^{\langle*\rangle}({\gamma}({{z}})-{\gamma}(\ov{{\zeta}})) &
\,\,{L}^{\langle*\rangle}(\wt R_{z}^0-\wt R_{\ov{\zeta}}^{0})
{L}
\end{bmatrix}\\
& =\begin{bmatrix}
\gamma ({\zeta})^*\gamma({{z}}) &
\gamma ({\zeta})^*\wt R_{z}^0{L}\\
{L}^{\langle*\rangle}R_{\ov{\zeta}}^{0}\gamma({{z}})&
\,\,{L}^{\langle*\rangle}R_{\ov{\zeta}}^{0}\wt R_{z}^0{L}
\end{bmatrix}
=T({\zeta})^*T({{z}}).
\end{split}
\]
This implies that ${\mathfrak A}_{\Pi{\sL}}\in \cR^{2p\times 2p}$.
\medskip

 (ii)
If  ${{z}}\in\rho(A,{\sL})$, then, by Lemma~\ref{lem:PQ_prop}(i), the operator ${L}^{\langle*\rangle}\upharpoonright\!{\sN_{\overline{{z}}}}$ is an isomorphism from $\sN_{\overline{{z}}}$ onto ${\dC^p}$.  Therefore, the operator $a_{21}({\overline{{z}}})={L}^{\langle*\rangle}\gamma(\overline{{z}})
=\left({L}^{\langle*\rangle}\!\upharpoonright\!{\sN_{\overline{{z}}}}\right) \gamma(\overline{{z}})$ is an isomorphism in $\dC^p$. This implies that $a_{12}({{z}})^{-1}\in\dC^{p\times p}$, since $a_{12}({{z}})=a_{21}({\overline{{z}}})^*$.

Conversely, if $a_{12}({{z}})^{-1}\in\dC^{p\times p}$, then $a_{21}({\overline{{z}}})=a_{12}({{z}})^*={L}^{\langle*\rangle}\gamma(\overline{{z}})$  is an isomorphism in $\dC^p$. Hence,
${L}^{\langle*\rangle}\!\upharpoonright\!\sN_{\overline{z}}$ is invertible
and, by Lemma~\ref{lem:PQ_prop}(i),
 ${{z}}\in\rho(A,{\sL})$.
\end{proof}
\begin{definition} \label{def:LpreresM}
The $\dC^{2p\times 2p}$--valued function ${\mathfrak
A}_{\Pi{\sL}}({{z}})$  defined by~\eqref{eq:Lres2A}
is called the ${\sL}$-preresolvent matrix of $A$ corresponding to the boundary triple $\Pi$, or, shortly,  the $\Pi{\sL}$-{\it preresolvent matrix} of $A$.

\end{definition}
\index{$\Pi{\sL}$-preresolvent matrix}
  The following lemma provides a description of ${\sL}$-resolvents of $A$.
\begin{lemma}\label{lem:Gres_F1}
Let the assumptions of Lemma~\ref{lem:PreresM}  hold and let $\rho_s( A, {\sL})\ne\emptyset$. Then the formula
\begin{equation}\label{eq:11.Lres1}
    {L}^{\langle*\rangle}\wh{\mathbf R}_{{z}}{L}={\mathfrak a}_{22}({{z}})-
    {\mathfrak a}_{21}({{z}})(\tau({{z}})+{\mathfrak a}_{11}({{z}}))^{-1}
    {\mathfrak a}_{12}({{z}}),\quad {z}\in\rho_s( A, {\sL})\setminus\dR,
\end{equation}
establishes a bijective correspondence between the set 
of all ${\sL}$-resolvents of $A$  and the set of
 $\tau\in\wt\cR^{\ptp}$.
\end{lemma}
\begin{proof}
  If ${L}^{\langle*\rangle}\wh{\mathbf R}_{{z}}{L}$ is
  an ${\sL}$-resolvent of $A$,
  then, by Theorem~\ref{krein}, the extended generalized resolvent $ \wt{\mathbf R}_{z}$ admits the representation
\begin{equation}\label{eq:11.Lres2}
\wt{\mathbf R}_{{{z}}}=\wt R_{{z}}^0-\gamma({{z}})(M({{z}})
     +\tau({{z}}))^{-1}\gamma(\ov{{{z}}})^{\langle*\rangle},\quad
     {{z}}\in\dC\setminus\dR,
\end{equation}
for some $\tau\in\wt\cR^{\ptp}$.
Subtracting the regularizer $\bR^0$ and  applying from both sides the operators
$ {L}^{\langle*\rangle}$ and ${L}$ we obtain \eqref{eq:11.Lres1}.
Notice that $\tau\in\wt\cR^{\ptp}$ is uniquely defined by~\eqref{eq:11.Lres1} since,
by Lemma~\ref{lem:PreresM}(ii),
$a_{21}({z})$ and $a_{12}({z})$ are invertible for
${{z}}\in\rho_s( A, {\sL})$.
\end{proof}
\subsection{Right ${\sL}$-resolvent matrix}
\label{sec:11.5}

  In this section, for any  symmetric operator $A$ in a Hilbert space $
  \sH$, boundary triple  $\Pi= (\dC^p,\Gamma_0,\Gamma_1)  $ for   $A^{*}$
and a gauge  $\sL\subset\sH_-$ of $A$, we associate a  $\Pi\sL$-resolvent matrix $W({z})$  and present a formula which describes $\sL$-resolvents
of $A$ with the help of a linear-fractional transformation $T_W$.

\begin{definition}\label{def:11.cP}
Let  $J_p$ be given by~\eqref{kerK0} and let $\wt J_p:=iJ_p$.
An mvf $W(z)$ holomorphic on a domain $\Omega\subset\dC_+$ 
is said to belong
to the class $\cW(\wt J_p)$
if the kernel
\begin{equation}\label{kerK}
{\mathsf K}_{\zeta}^W({z}):=
\frac{\wt J_p-W({z})\wt J_p W({\zeta})^*}{-i({z}-\ov{\zeta})} \quad
{z},{\zeta}\in \Omega,\quad {z}\ne\overline{{\zeta}}
\end{equation}
is non-negative on $\Omega$.
\end{definition}
In particular, it follows from Definition~\ref{def:11.cP} that for
$W\in\cW(\wt J_p)$
\begin{equation}\label{eq:WJnonneg3}
\wt J_p-W(z)\wt J_pW(z)^*\ge 0\quad\text{for all}\quad z\in\Omega\subset\dC_+.
\end{equation}
For a  block mvf
$W({z})=\left[w_{ij}\right]_{i,j=1}^2\in \cW(\wt J_p)$
with blocks $w_{ij}({z})$ of size
${p\times p}$
we define a transformation $T_W$
in the set 
of linear relations $\tau$ in $\dC^p$ via
\begin{equation}
\label{eq:App2:LFT}
 T_W[\tau]
  =\Big\{\,\begin{bmatrix}w_{22}h+w_{21}h' \\w_{12}h+w_{11}h'\end{bmatrix}:\,
  \begin{bmatrix} h \\ h'\end{bmatrix}\in \tau\,\Big\},
\end{equation}
see Yu. Shmul'yan~\cite{Sh80}.
For an $\cR^{\ptp}$-family $\tau({z})
=(\text{\rm col} \{ \varphi({z}),\psi({z}) \})_{z\in\dC\setminus\dR}$
    (see Definition~\ref{def:Nk-family})
    let us set
\begin{equation}\label{eq:Lambda}
    \Lambda_{\varphi,\psi}=\{{z}\in \rho(A,{\sL})\cap\rho(A_0):\,
{\det}\left(w_{21}({z})\psi({z})+w_{22}({z})\varphi({z})\right)\ne0\}.
\end{equation}
If $\Lambda_{\varphi,\psi}\ne\emptyset$, then the linear fractional transform in~\eqref{eq:App2:LFT} for
${{z}}\in\Lambda_{\varphi,\psi}$ takes the form
\begin{equation}\label{eq:Lres_phipsi}
 T_{W}[\tau({{z}})]=
(w_{11}({{z}})\psi({{z}})+w_{12}({{z}})\varphi({{z}}))
(w_{21}({{z}})\psi({{z}})+w_{22}({{z}})\varphi({{z}}))^{-1}.
\end{equation}

\begin{definition}\label{def:11.Pl-res,matr2}
Let $\Pi=  (\dC^p,\Gamma_0,\Gamma_1) $ be a boundary triple for
$A^{*}$, let ${\sL}\subset\sH_-$ be a gauge for $A$ such that
$\rho(A,{\sL})\cap\rho({A}_0)\ne\emptyset$,  and let
${\mathfrak A}_{\Pi{\sL}}({{z}})=({\mathfrak a}_{ij}({{z}}))_{i,j=1}^2$
be  the $\Pi{\sL}$-preresolvent matrix of $A$, see~\eqref{eq:Lres2A}. The $2p\times 2p$--mvf $W_{\Pi{\sL}}({z})$ defined for $z\in\rho(A,{\sL})\cap\rho({A}_0)$ by
\begin{equation}\label{eq:LresM}
    W_{\Pi\sL}({z})=\begin{bmatrix}
    {\mathfrak a}_{22}({z}){\mathfrak a}_{12}({z})^{-1} &     {\mathfrak a}_{22}({z}){\mathfrak a}_{12}({z})^{-1}    {\mathfrak a}_{11}({z})-{\mathfrak a}_{21}({z})\\
        {\mathfrak a}_{12}({z})^{-1}                 &   {\mathfrak a}_{12}({z})^{-1}    {\mathfrak a}_{11}({z})
    \end{bmatrix}, 
\end{equation}
is called the {\it right ${\sL}$-resolvent matrix of $A$} corresponding to the boundary triple $\Pi$ or, briefly, the right $\Pi{\sL}$-resolvent matrix of $A$.
  \end{definition}

In the following theorem we show that the $\Pi{\sL}$-resolvent matrix
$W=W_{\Pi\sL}$ belongs to the class $\cW(\wt J_p)$ (see Definition~\ref{def:11.cP}).

\begin{theorem}\label{thm:Gres_F1}
Let the assumptions of Lemma~\ref{lem:PreresM}  hold and let the mvf's
$\sA({z}):={\mathfrak A}_{\Pi{\sL}}({{z}})
=({\mathfrak a}_{ij}({{z}}))_{i,j=1}^2$ and $W({z}):=W_{\Pi\sL}({z})$
be given  by~\eqref{eq:Lres2A} and~\eqref{eq:LresM}. Then
\begin{enumerate}\def\labelenumi{\rm (\roman{enumi})}
\item $W({z})$ and $\sA({z})$ are related by
\begin{equation}\label{eq:W_A}
  W({z})=\begin{bmatrix}
   0 & {\mathfrak a}_{12}({{z}})\\
  -I_p & {\mathfrak a}_{22}({{z}})
    \end{bmatrix}^{-1}
    \begin{bmatrix}
   I_p & {\mathfrak a}_{11}({{z}})\\
  0  & {\mathfrak a}_{21}({{z}})
    \end{bmatrix}, \quad {z}\in\rho(A,{\sL})\cap\rho({A}_0).
\end{equation}
 \item 
 The kernel ${\mathsf K}_{\zeta}^W({z})$
 is related to the kernel ${\mathsf{N}}^{\sA}_{\zeta}({{z}})$ by
\begin{equation}\label{eq:KW_NA}
{\mathsf K}_{\zeta}^W({z})=\begin{bmatrix}
   0 & {\mathfrak a}_{12}({{z}})\\
  -I_p & {\mathfrak a}_{22}({{z}})
    \end{bmatrix}^{-1}
    {\mathsf{N}}^{\sA}_{\zeta}({{z}})
    \begin{bmatrix}
   0 & {\mathfrak a}_{12}({\zeta})\\
  -I_p & {\mathfrak a}_{22}({\zeta})
    \end{bmatrix}^{-*}.
\end{equation}
\item
$W\in \cW(\wt J_p)$.

\item
For every $\tau\in\wt\cR^{p\times p}$ it holds $\wt\tau=T_{W}[\tau]\in\wt\cR^{p\times p}$.
\end{enumerate}
\end{theorem}
\begin{proof}
  The items (i) and (ii) are checked by straightforward calculations.

  (iii) The formula \eqref{eq:KW_NA} ensures that the kernels
  ${\mathsf K}_{\zeta}^W({z})$ and ${\mathsf{N}}^{\sA}_{\zeta}({{z}})$
  are  non-negative simultaneously on $\rho(A,\sL)\cap \rho(A_0)$.

\medskip

  (iv) See the proof for instance in~\cite[Proposition B10]{DM25}.
\end{proof}

 In the following theorem we provide an explicit formula for the    $\Pi{\sL}$-resolvent matrix $W_{\Pi{\sL}}({{z}})$ of $A$
in the sense of Definition \ref{def:11.Pl-res,matr2} that expresses it in terms of the boundary mappings $\Gamma_j$ and the family $\cG({\cdot})$  given by  
\begin{equation}\label{eq:11.cG}
    {\cG}({{z}})=\begin{bmatrix}
   -\cQ({{z}})\\
      \cP({{z}})
    \end{bmatrix},\quad
    \wh{\cG}({{z}})=\begin{bmatrix}
   -\wh\cQ({{z}})\\
      \wh\cP({{z}})
    \end{bmatrix}.
\end{equation}
\begin{theorem}\label{thm:ResM}
Let $(\dC^p,\wh\Gamma_0,\wh\Gamma_1)$ be an extended boundary triple for $A^{\langle*\rangle}$, let ${\sL}$ be a subspace of $\sH_-$
such that $\rho(A,{\sL})\ne\emptyset$ and let the mvf  $\wh W$ be defined by
\begin{equation}\label{eq:Formula_W}
    \wh W({{z}}):=\left(\wh\Gamma(\wh{\cG}({z})^{\langle*\rangle})\right)^*
 =\begin{bmatrix}
   -\wh\Gamma_0\wh\cQ({z})^{\langle*\rangle}  &  \wh \Gamma_0\wh\cP({z})^{\langle*\rangle}\\
   -\wh\Gamma_1\wh\cQ({z})^{\langle*\rangle} &   \wh\Gamma_1\wh\cP({z})^{\langle*\rangle}
       \end{bmatrix}^*,\quad
 {{z}}\in\rho(A,{\sL}).
\end{equation}
Then:
\begin{enumerate}\def\labelenumi{\rm (\roman{enumi})}
  \item 
 The kernel
\begin{equation}\label{kerK1}
{\mathsf K}_{\zeta}^{\wh W}({z}):=
\frac{\wt J_p-\wh W({z})\wt J_p \wh W({\zeta})^*}{-i({z}-\ov{\zeta})}, \quad{z},{\zeta}\in \rho(A,{\sL}),\quad {z}\ne\overline{{\zeta}},
\end{equation}
admits the factorization
\begin{equation}\label{eq:Christ_Id}
    {\mathsf K}_{\zeta}^{\wh W}({z})={\cG}({{z}}){\cG}({\zeta})^{\langle*\rangle},\qquad
 {{z}},{\zeta}\in\rho(A,{\sL}).
\end{equation}
\medskip
  \item 
  If, in addition, $\rho_s(A,{\sL}):=\rho(A,{\sL})\cap\overline{\rho(A,{\sL})}\ne
\emptyset$, then $\wh W({{z}})$ is invertible for ${z}\in \rho_s(A,{\sL})$ and
coincides  on $\rho_s(A,{\sL})\cap\rho(A_0)$ with the $\Pi{\sL}$-resolvent matrix
 $W_{\Pi{\sL}}({{z}})$ of $A$.
 \end{enumerate}
\end{theorem}
\begin{proof}
(i) Let us  decompose the kernel ${\mathsf K}_{\zeta}^{\wh W}({z})$ into four $p\times p$-blocks: ${\mathsf K}_{\zeta}^W({z})=\left[{\mathsf K}^{ij}_{\zeta}({z})\right]_{i,j=1}^2$.
Then the identity~\eqref{eq:Christ_Id} is splitted into four identities
\begin{equation}\label{eq:KrSaakId1}
  {\mathsf K}^{11}_{\zeta}({z})=\cQ({{z}})\cQ({\zeta})^{\langle*\rangle},\quad
  {\mathsf K}^{12}_{\zeta}({z})=-\cQ({{z}})\cP({\zeta})^{\langle*\rangle},
\end{equation}
\begin{equation}\label{eq:KrSaakId2}
  {\mathsf K}^{21}_{\zeta}({z})=-\cP({{z}})\cQ({\zeta})^{\langle*\rangle},\quad
  {\mathsf K}^{22}_{\zeta}({z})=\cP({{z}})\cP({\zeta})^{\langle*\rangle}.
\end{equation}

Setting ${z},{\zeta}\in{\rho(A,{\sL})}$,
\[
\wh f=\begin{bmatrix}
                          f\\ f'
                        \end{bmatrix}=\wh\cP({{\zeta}})^{\langle*\rangle}u_2,\,\,
                        \wh g=\begin{bmatrix}
                          g\\ g'
                        \end{bmatrix}=\wh\cP({z})^{\langle*\rangle}v_2,\,\quad
u=\begin{bmatrix}
                          0\\ u_2
                        \end{bmatrix},\,\,
                        v=\begin{bmatrix}
                          0 \\ v_2
                        \end{bmatrix}\in\dC^{2p},
\]
 we can rewrite the left hand part of~\eqref{eq:ExtGreen} as
\begin{equation}\label{eq:11.24}
\begin{split}
  (f',g)_\sH-(f,g')_\sH=(\ov{\zeta}-{z})(f,g)_\sH
   =(\ov{\zeta}-{z})(\cP({\zeta})^{\langle*\rangle}u_2,\cP({z})^{\langle*\rangle}v_2)_\sH.
   \end{split}
\end{equation}
In view of~\eqref{eq:Formula_W} the right hand part of~\eqref{eq:ExtGreen} takes the form
\begin{multline}\label{eq:11.25}
(\Gamma_0\wh g)^*(\Gamma_1\wh f)-(\Gamma_1\wh g)^*(\Gamma_0\wh f)
=-\left(\wh\Gamma\wh\cP({z})^{\langle*\rangle}v_2\right)^* J_p\left(\wh\Gamma\wh\cP({\zeta})^{\langle*\rangle}u_2\right)\\
  =-v^*\wh W({{z}})J_p\wh W({\zeta})^*u
  =(\ov{\zeta}-{z})v_2^*{\mathsf K}^{22}_{\zeta}({z})u_2.
   \end{multline}
Comparing of \eqref{eq:11.24} and \eqref{eq:11.25} proves the second identity in~\eqref{eq:KrSaakId2}.

Similarly, setting
$\wh f=\wh\cP({{\zeta}})^{\langle*\rangle}u_2$,
$\wh g=\wh\cQ({{z}})^{\langle*\rangle}v_1$,
for
$u_2,v_1\in{\dC^p}$,
${z},{\zeta}\in{\rho(A,{\sL})}$, we can rewrite the left hand part of~\eqref{eq:ExtGreen} as
\begin{equation}\label{eq:W21A}
\begin{split}
  (f',g)_\sH-\langle f,g'\rangle_{+,-}&=(\ov{\zeta}-{z})( f,g)_\sH-\langle f,{L}v_1\rangle_{+,-}\\
   &=(\ov{\zeta}-{z})v_1^*
   \cQ({z})\cP({\zeta})^*u_2-v_1^*u_2,
   \end{split}
\end{equation}
while the right hand part of~\eqref{eq:ExtGreen} takes the form
\begin{equation}\label{eq:W21B}
\begin{split}
&(\wh\Gamma_0\wh g)^*(\Gamma_1\wh f)-(\wh\Gamma_1\wh g)^*(\Gamma_0\wh f)
=-v_1^*\left( \Gamma\wh\cQ({{z}})^{\langle*\rangle}\right)^* J_p\left(\Gamma\wh\cP({\zeta})^{\langle*\rangle}\right)u_2\\
  &=\begin{bmatrix}
                          v_1^* & 0
                        \end{bmatrix}
  \wh W({{z}})J_p\wh W({\zeta})^*\begin{bmatrix}
                         0 \\ u_2
                        \end{bmatrix}
  =-(\ov{\zeta}-{z})v_1^*{\mathsf K}^{12}_{\zeta}({z})u_2
  -v_1^*u_2.
   \end{split}
\end{equation}
Comparing of \eqref{eq:W21A} and \eqref{eq:W21B} proves
the second identity  in~\eqref{eq:KrSaakId1}
as well as the first identity in~\eqref{eq:KrSaakId2} since
${\mathsf K}^{21}_{\zeta}({z})=\left({\mathsf K}^{12}_{z}({\zeta})\right)^*$.

To prove the first identity  in~\eqref{eq:KrSaakId1} let us set
$  {z},{\zeta}\in{\rho(A,{\sL})}$ and
\[
\wh f=\begin{bmatrix}
                          f\\ f'
                        \end{bmatrix}=-\wh\cQ({{\zeta}})^{\langle*\rangle}u_1,\,\,
                        \wh g=\begin{bmatrix}
                          g\\ g'
                        \end{bmatrix}=-\wh\cQ({{z}})^{\langle*\rangle}v_1,\,\,
u=\begin{bmatrix}
                          u_1 \\ 0
                        \end{bmatrix},
                        \,\, v=\begin{bmatrix}
                          v_1 \\ 0
                        \end{bmatrix}\in\dC^{2p}.
\]
By Lemma~\ref{lem:ExtGreenForm}, the left hand part (LHP) of~\eqref{eq:ExtGreen2} equals to
 \[
 \begin{split}
 LHP=
    (f'+{L}u_1,g)_\sH-(f, g'+{L}v_1)_\sH
  =(\ov{\zeta}-{z})(\cQ({{\zeta}})^{\langle*\rangle}u_1,\cQ({{z}})^{\langle*\rangle}v_1)_\sH.
  \end{split}
\]
Notice that, by~\eqref{eq:PQ_pr2},
\[
\left\langle {L}u_1,\cQ({{z}})^{\langle*\rangle}v_1-\bR^0 {L}v_1\right\rangle_{-,+}
   =\left\langle \cQ({{\zeta}})^{\langle*\rangle}u_1-\bR^0 {L}u_1, {L}v_1\right\rangle_{+,-}=0
\]
and hence the right hand part (RHP) of~\eqref{eq:ExtGreen2} equals to
 \begin{multline*}
 RHP=-v^*\wh W({{z}})J_p\wh W({\zeta})^*u
   +\left\langle {L}u_1,-\cQ({{z}})^{\langle*\rangle}v_1+\bR^0 {L}v_1\right\rangle_{-,+}\\
   -\left\langle -\cQ({{\zeta}})^{\langle*\rangle}u_1+\bR^0 {L}u_1, {L}v_1\right\rangle_{+,-}
   =-v^*\wh W({{z}})J_p\wh W({\zeta})^*u
   =(\ov{\zeta}-{z})v_1^*{\mathsf K}^{11}_{\zeta}({z})u_1.
   \end{multline*}
This implies the first identity  in~\eqref{eq:KrSaakId1}.
\medskip

(ii) Assume that ${z}\in\rho_s(A,{\sL})$.
Then~\eqref{eq:Christ_Id}  implies the identities
\begin{equation}\label{eq:11.W*W}
\wh W({{z}})J_p\wh W(\ov{{z}})^*=J_p,\quad  \wh W(\ov{{z}})^*J_p\wh W({{z}})=J_p, \quad {z}\in\rho_s(A,{\sL}).
\end{equation}
Let us  decompose the matrices $W_{\Pi{\sL}}({{z}})$ and $\wh W({{z}})$ into four $p\times p$-blocks:
  \[
  W_{\Pi{\sL}}({{z}})=[w_{ij}({{z}})]_{i,j=1}^2,\qquad
  \wh W({{z}})=[\wh w_{ij}({{z}})]_{i,j=1}^2.
  \]
The equalities in~\eqref{eq:11.W*W} are equivalent to the following conditions on the components of the matrix $ \wh W $:
\begin{equation}\label{LR_39}
\wh w_{21}\wh w_{22}^\#=\wh w_{22}\wh w_{21}^\#,\ \
\wh w_{11}\wh w_{12}^\#=\wh w_{12}\wh w_{11}^\#,\ \
\wh w_{11}\wh w_{22}^\#-\wh w_{12}\wh w_{21}^\#=I_{p}.
\end{equation}
\begin{equation}\label{LR_38}
 \wh w_{12}^\#\wh w_{22}=\wh w_{22}^\#\wh w_{12},\ \
 \wh w_{11}^\#\wh w_{21}=\wh w_{21}^\#\wh w_{11},\ \
\wh w_{11}^\#\wh w_{22}-\wh w_{21}^\#\wh w_{12}=I_p,
\end{equation}

By~\eqref{eq:Formula_W}, we have
\begin{equation}\label{eq:11.Wh_w}
  \begin{split}
     \wh w_{11}({{z}})^* & =-\wh \Gamma_0\wh\cQ({{z}})^{\langle*\rangle},\qquad
     \wh w_{21}({{z}})^*  =\wh \Gamma_0\wh\cP({z})^{\langle*\rangle},\\
     \wh w_{12}({{z}})^* & =-\Gamma_1\wh\cQ({{z}})^{\langle*\rangle},\qquad
     \wh w_{22}({{z}})^*  =\Gamma_1\wh\cP({z})^{\langle*\rangle}.
  \end{split}
\end{equation}
Hence, we obtain explicit formulas for the elements of the matrix $\sA_{\Pi{\sL}}({{z}})=({\mathfrak a}_{ij}({{z}}))_{i,j=1}^2$ by means of $ \wh w_{ij}({{z}})$. By~\eqref{eq:Lres2A} and \eqref{eq:Formula_W},
\begin{equation} \label{eq:11.a11}
  {\mathfrak a}_{11}({{z}})=M({{z}})=\wh w_{22}(\ov{{z}})^*\wh w_{21}(\ov{{z}})^{-*}
  =\wh w_{22}^\#({{z}})\wh w_{21}^\#({{z}})^{-1}.
\end{equation}
Next, it follows from \eqref{eq:Lres2A}, \eqref{eq:11.Wh_w} and~\eqref{eq:PQ_pr1} that for all $u\in\dC^p$
\[
  {\mathfrak a}_{21}({{z}})\wh w_{21}(\ov{{z}})^*u= {L}^{\langle*\rangle}\gamma({{z}})\Gamma_0\wh\cP({z})^{\langle*\rangle}u
  = {L}^{\langle*\rangle}\cP({z})^{\langle*\rangle}u=u,
\]
whence
\begin{equation} \label{eq:11.a21}
  {\mathfrak a}_{21}({{z}})=\wh w_{21}^\#({{z}})^{-1},\quad
  {\mathfrak a}_{12}({{z}})={\mathfrak a}_{21}^\#({{z}})=\wh w_{21}( {z})^{-1}.
\end{equation}
To find the expression of ${\mathfrak a}_{22}({{z}})= {L}^{\langle*\rangle}\wh R^0_{z} {L}$ we consider the problem
\[
\wh f=\begin{bmatrix}
  f \\
  {L}u
\end{bmatrix}\in A^{\langle*\rangle}-{z} I,\quad
\Gamma_0\wh f=0,\quad
\wh f=\begin{bmatrix}
  f \\
  {L}u+{z} f
\end{bmatrix}\in A^{\langle*\rangle},\quad
\begin{array}{l}
f\in\sH,\\
u\in\dC^p.
\end{array}
\]
It follows from \eqref{eq:PQ*} that this problem has a solution of the form
\[
\wh f=\wh\cQ(\ov{{z}})^{\langle*\rangle}u-\wh\cP(\ov{{z}})^{\langle*\rangle}v
\quad\text{with some}\quad
v\in\dC^p.
\]
In view of \eqref{eq:11.Wh_w},  the equality $\wh\Gamma_0\wh f=0$ implies that
\[
\wh\Gamma_0\wh f=\wh \Gamma_0\wh\cQ(\ov{{z}})^{\langle*\rangle}u-\Gamma_0\wh\cP(\ov{{z}})^{\langle*\rangle}v=
-\wh w_{11}^\#({{z}})u-\wh w_{21}^\#({{z}})v=0
\]
and hence $v=-\wh w_{21}^\#({{z}})^{-1}\wh w_{11}^\#({{z}})u$. Therefore,
\[
f=\cQ(\ov{{z}})^{\langle*\rangle}u
+\cP(\ov{{z}})^{\langle*\rangle}\wh w_{21}^\#({z})^{-1}\wh w_{11}(\ov{{z}})^*u
=\wt R^0_{z} {L}u,
\]
and, by~\eqref{eq:11.Wh_w}, \eqref{eq:PQ_pr1}, \eqref{eq:PQ_pr2},
\begin{equation}\label{eq:11.a22}
\begin{split}
  {\mathfrak a}_{22}({{z}})u&={L}^{\langle*\rangle}(\wt R^0_{z}-\bR^0){L}\\
  &={L}^{\langle*\rangle}(\cQ(\ov{{z}})^{\langle*\rangle}u-\bR^0 {L})u
  +{L}^{\langle*\rangle}\cP(\ov{{z}})^{\langle*\rangle}
  \wh w_{21}^\#({{z}})^{-1}\wh w_{11}^\#({{z}})u\\
  &=\wh w_{21}^\#({{z}})^{-1}\wh w_{11}^\#({{z}})u.
\end{split}
\end{equation}
It follows from  \eqref{eq:11.a11}--\eqref{eq:11.a22} that the $\Pi{\sL}$-preresolvent matrix takes the form
\begin{equation}\label{eq:11.Preres_A}
  \sA_{\Pi{\sL}}({{z}})=\begin{bmatrix}
  \wh w_{22}^\#({{z}})\wh w_{21}^\#({{z}})^{-1} & \wh w_{21}({z})^{-1} \\
  \wh w_{21}^\#({{z}})^{-1}
  & \wh w_{21}^\#({{z}})^{-1}\wh w_{11}^\#({{z}})
         \end{bmatrix},\quad {z}\in\rho_s(A,{\sL}).
\end{equation}
Finally, \eqref{eq:LresM}, 
\eqref{eq:11.Preres_A}, \eqref{LR_39} and \eqref{LR_38} imply that for all ${z}\in\rho_s(A,{\sL})\cap\rho(A_0)$
\[
\begin{split}
   w_{11}({{z}}) & ={\mathfrak a}_{22}({{z}}){\mathfrak a}_{12}({{z}})^{-1}=
   \wh w_{21}^\#({{z}})^{-1}\wh w_{11}^\#({{z}})\wh w_{21}( {z})=\wh w_{11}({{z}}), \\
   w_{12}({{z}}) & ={\mathfrak a}_{22}({{z}}){\mathfrak a}_{12}({{z}})^{-1}{\mathfrak a}_{11}({{z}}) -{\mathfrak a}_{21}({{z}})\\
      &=\wh w_{11}({{z}})\wh w_{22}^\#({{z}})\wh w_{21}^\#({{z}})^{-1}-\wh w_{21}^\#({{z}})^{-1}
     =\wh w_{12}({{z}}), \\
   w_{21}({{z}}) & ={\mathfrak a}_{12}({{z}})^{-1}=\wh w_{21}({{z}}), \\
   w_{22}({{z}}) & ={\mathfrak a}_{12}({{z}})^{-1}{\mathfrak a}_{11}({{z}})=
   \wh w_{21}({{z}})\wh w_{22}^\#({{z}})\wh w_{21}^\#({{z}})^{-1}
   =\wh w_{22}({{z}}), \\
\end{split}
\]
and hence, $W_{\Pi{\sL}}({{z}})=\wh W({{z}})$ for all ${z}\in\rho_s(A,{\sL})\cap\rho(A_0)$.
\end{proof}

\begin{theorem}\label{prop:Gresolv}
Let $\Pi=  (\dC^p,\Gamma_0,\Gamma_1) $ be a boundary triple for
$A^{*}$ and let ${\sL}\subset\sH_-$ be a gauge for $A$,
 let  $W_{\Pi{\sL}}({z})$ be the $\Pi{\sL}$-resolvent matrix of $A$ defined on $\rho(A,{\sL})\cap\rho({A}_0)$ by
\eqref{eq:LresM}.
 Then
\begin{enumerate}
\item[\rm(i)] The formula
\begin{equation}\label{eq:Lres}
(r({z})=){L}^{\langle*\rangle}\wh{\mathbf R}_{{z}}{L}
 =T_{W_{\Pi{\sL}}({z})}[\tau({{z}})], \quad
 {z}\in\rho(A,{\sL})\cap\rho(A_0)\cap\rho(\wt A),
\end{equation}
establishes a one--to--one correspondence
$r(\cdot)\longleftrightarrow  \tau(\cdot)$   between the set  of ${\sL}$-resolvents $r$ of $A$ and the set of
 $\tau\in\wt\cR^{\ptp}$.
 \item[\rm(ii)]
If, in addition, $A_0=A\dotplus \text{\rm col} \{0,\mul A^*\}$,
then the $\Pi$-admissibility of $\tau({z})$ (see Definition~\ref{def:Adm}) is equivalent
to the single condition \eqref{tlimit00}.
\end{enumerate}
  \end{theorem}
\begin{proof} (i) By Lemma~\ref{lem:PreresM}(iii), ${\mathfrak a}_{12}({z})$ is invertible for ${z}\in\rho(A,{\sL})\cap\rho(A_0)\cap\rho(\wt A)$.
  It follows from~\eqref{eq:11.Lres1} and~\eqref{eq:LresM} that
\begin{equation}\label{eq:Lres.2}
  \begin{split}
{L}^{\langle*\rangle}\wh{\mathbf R}_{{z}}{L}
&={\mathfrak a}_{22}({{z}})-{\mathfrak a}_{21}({{z}})(\tau({{z}})+{\mathfrak a}_{11}({{z}}))^{-1}
{\mathfrak a}_{12}({{z}})\\
&={\mathfrak a}_{22}({{z}})-{\mathfrak a}_{21}({{z}})({\mathfrak a}_{12}({{z}})^{-1}\tau({{z}})+
{\mathfrak a}_{12}({{z}})^{-1}{\mathfrak a}_{11}({{z}}))^{-1}\\
&=(w_{11}({{z}})\tau({{z}})+w_{12}({{z}}))
(w_{21}({{z}})\tau({{z}})+w_{22}({{z}}))^{-1}.
\end{split}
\end{equation}
Now (i) follows from~Lemma~\ref{lem:Gres_F1}(i) and
(ii) follows from Theorem~\ref{krein}(iii).
\end{proof}

\subsection{Left ${\sL}$-resolvent matrix}
For a right $\Pi{\sL}$-resolvent matrix  $W({z}):=W_{\Pi{\sL}}({{z}})=\left[w_{i,j}({z})\right]_{i,j=1}^2$ of $A$
and  $\tau\in \wt\cR(\cH)$ we define the left $\Pi{\sL}$-resolvent matrix of $A$
\begin{equation}\label{eq:LeftResMat}
  W^\ell_{\Pi{\sL}}({z})=[w^\ell_{ij}({z})]_{i,j=1}^2:=W^\#({z}) 
\end{equation}
and the left linear-fractional transformation
\begin{equation}\label{eq:2.2}
  T_W^{\ell}[\tau]:=(\tau({z})w^{\ell}_{ 12}({z})+w_{22}^{\ell}({z}))^{-1}
  (\tau({z})w_{11}^{\ell}({z})+w^{\ell}_{ 21}({z})).
\end{equation}
It follows from~\eqref{eq:LresM} and the formulas
\[
{\mathfrak a}_{11}^\#({{z}})={\mathfrak a}_{11}({{z}}),\quad
{\mathfrak a}_{12}^\#({{z}})={\mathfrak a}_{21}({{z}}),\quad
{\mathfrak a}_{22}^\#({{z}})={\mathfrak a}_{22}({{z}}),\quad
{z}\in\rho_s(A,{\sL}),
\]
that the left $\Pi{\sL}$-resolvent matrix of $A$ takes the form
\begin{equation}\label{eq:LresM_Left}
    W_{\Pi{\sL}}^\ell({{z}})=\begin{bmatrix}
        {\mathfrak a}_{21}({{z}})^{-1}{\mathfrak a}_{22}({{z}}) & {\mathfrak a}_{21}({{z}})^{-1}\\
    {\mathfrak a}_{11}({z}){\mathfrak a}_{21}({{z}})^{-1}{\mathfrak a}_{22}({{z}})-{\mathfrak a}_{12}({{z}})
      &
    {\mathfrak a}_{11}({{z}}){\mathfrak a}_{21}({{z}})^{-1} \\
    \end{bmatrix},\quad {z}\in\rho_s(A,{\sL}).
\end{equation}
Moreover, \eqref{eq:Formula_W} and \eqref{eq:LeftResMat} yield another formula for  the left $\Pi{\sL}$-resolvent matrix
\begin{equation}\label{eq:Left_WPQ}
   W_{\Pi{\sL}}^\ell({{z}}):=\wh\Gamma(\wh{\sL}(\ov{z}))^{\langle*\rangle})
 =\begin{bmatrix}
   -\wh\Gamma_0\wh\cQ(\ov{z})^{\langle*\rangle}  &  \Gamma_0\wh\cP(\ov{z})^{\langle*\rangle}\\
   -\wh\Gamma_1\wh\cQ(\ov{z})^{\langle*\rangle} &   \Gamma_1\wh\cP(\ov{z})^{\langle*\rangle}
       \end{bmatrix},\quad
 {{z}}\in\rho_s(A,{\sL}).
\end{equation}

By Theorem~\ref{thm:Gres_F1}, for every $\tau\in \wt\cR^{\ptp}$
 the family $\wt\tau=T_W[\tau]$
belongs to $\wt\cR^{\ptp}$.
Since  $\tau^\#({z})=\tau({z})$ and $\wt\tau^\#({z})=\wt\tau({z})$,
the left and right linear-fractional transformations coincide:
\begin{equation}\label{eq:LeftLFT2}
  (T_W^{\ell}[\tau])({z})=(T_W[\tau])^\#({z})=(T_W[\tau])({z}),\quad {z}\in\rho_s(A,{\sL})\setminus \dR.
\end{equation}

 If $\tau\in \wt\cR^{\ptp}\setminus \cR^{\ptp}$ it is convenient to consider the following  kernel representation of $\tau$,
see~\eqref{LR_50},
\[
\tau({z})=\ker{\begin{bmatrix} {C(z)} & -{D}(z)\end{bmatrix}},
\]
where $\begin{bmatrix}
    C({z})& D({z})
    \end{bmatrix}$ is an $\cR^{\ptp}$-pair connected with
    $\cR^{\ptp}$-family $\tau$
by \eqref{eq:Nkpairs_fam}.

\begin{theorem}\label{thm:Gres_CD}
Let assumptions of Theorem~\ref{prop:Gresolv} hold and let $W_{\Pi{\sL}}^\ell({{z}})$
be the left $\Pi{\sL}$-resolvent matrix of $A$ defined by~\eqref{eq:LresM_Left}. Then the formula
\begin{equation}\label{eq:2.2L}
  {L}^{\langle*\rangle}\wh {\mathbf R}_{{z}}{L}=(C({z})w_{12}^\ell({z})+D({z})w_{22}^\ell({z}))^{-1}
  (C({z})w_{11}^\ell({z})+D({z})w_{21}^\ell({z}))
\end{equation}
establishes a one--to--one correspondence  between the set
of all ${\sL}$-resol\-vents of $A$  and the set of
$\cR^{\ptp}$-pairs $\begin{bmatrix}
    C({z})& D({z})
    \end{bmatrix}$.
\end{theorem}
\begin{proof} Using the
connection $\tau({z})=\ker{\begin{bmatrix} {C({z})} & -{D}({z})\end{bmatrix}}$ between all $\cR^{\ptp}$-families
$\tau({z})$
and $\cR^{\ptp}$-pairs $\begin{bmatrix} {C({z})} & {D}({z})\end{bmatrix}$,
see Lemma~\ref{lem:Nkpairs_fam},
and the formula
\[
(\tau({{z}})+{\mathfrak a}_{11}({{z}}))^{-1}=(C({z})+D({z}){\mathfrak a}_{11}({z}))^{-1}D({z})
\]
we obtain from
\eqref{eq:11.Lres1} that for $ {z}\in\rho(A,{\sL})\cap\rho(A_0)\cap\rho(\wt A)$
\begin{multline}\label{eq:Lres.2A}
{L}^{\langle*\rangle}\wh{\mathbf R}_{{z}} {L}
={\mathfrak a}_{22}({z})-{\mathfrak a}_{21}({z})(C({z})+D({z}){\mathfrak a}_{11}({z}))^{-1}
D({z}){\mathfrak a}_{12}({z})\\
={\mathfrak a}_{22}({z})-
(C({z}){\mathfrak a}_{21}({z})^{-1}+D({z}){\mathfrak a}_{11}({z}){\mathfrak a}_{21}({z})^{-1})^{-1}D({z}){\mathfrak a}_{12}({z})\\
=\left(C{\mathfrak a}_{21}^{-1}
+D{\mathfrak a}_{11}{\mathfrak a}_{21}^{-1}\right)^{-1}
\left(C{\mathfrak a}_{21}^{-1}{\mathfrak a}_{22}
+ D\left\{{\mathfrak a}_{11}{\mathfrak a}_{21}^{-1}{\mathfrak a}_{22}-{\mathfrak a}_{12}\right\}\right)
\end{multline}
which, by~\eqref{eq:LresM_Left}, yields~\eqref{eq:2.2}.
\end{proof}

\section{Spectral and pseudo-spectral functions}\label{sec:5}
In this paper we use the following classification from \cite{LaTe78}, \cite{LaTe84}, \cite{Sak92}, \cite{Kac03}, \cite{ArDy12} and \cite{Mog15MFAT}.

\begin{definition}\label{def:PseudoSF}
Let $\sL(\subset\sH_-)$ be a gauge for $A$ such that $\dR\subset\rho(A,\sL)$,
 and let the mapping $\cP({z}):\sH_-\to\dC^p$ be defined by \eqref{eq:P_lambda}.
A non-decreasing left-continuous mvf $\dR\ni{\lambda}\mapsto \sigma({\lambda})$ with values in $\dC^{p\times p}$
is called
	\begin{enumerate}
		\item[(1)]
$LT$-{\it spectral function}  of the pair $\langle A;\sL\rangle$ if
\begin{equation}\label{eq:BesselIneq}
  \int_{\dR}(\cP({\lambda}) f)^*d\sigma({\lambda})\cP({\lambda}) f
  \le\langle  f, f\rangle_{\sH}, \quad \text{for all}\quad f\in\sH
\end{equation}
with equality sign for all $f\in\cdom{A}$.
The set of $LT$-spectral functions of $\langle A;\sL\rangle$ is denoted by $\Sigma_{LT}(A,\sL)$.
		\item[(2)]
$q$-{\it pseudo-spectral function}  of the pair $\langle A;\sL\rangle$   if
the mapping 
\begin{equation}\label{eq:GenFourier}
\cF:\sH\ni f\mapsto \cP({\cdot})f\in L^2(d\sigma)
\end{equation}
is  a partial isometry from $\sH$ to $L^2(d\sigma)$.
The set of $q$-pseudo-spectral functions of $\langle A;\sL\rangle$ is denoted by $\Sigma_{qpsf}(A,\sL)$.

		\item[(3)]
{\it pseudo-spectral function } of the pair $\langle A;\sL\rangle$   if
the mapping
$\cF:\sH\to L^2(d\sigma)$ in \eqref{eq:GenFourier}
is  a partial isometry  from $\sH$ to $L^2(d\sigma)$
such that $\ker \cF=\mul A$. 
The set of pseudo-spectral functions of $\langle A;\sL\rangle$ is denoted by $\Sigma_{psf}(A,\sL)$.
\item[(4)]
{\it spectral function}  for  $\langle A;\sL\rangle$  if 
the mapping $\cF$ is an  isometry from $\sH$ into $L^2(d\sigma)$.
The set of spectral functions of $\langle A;\sL\rangle$ is denoted by $\Sigma_{sf}(A,\sL)$.
\item[(5)]
a pseudo-spectral function
(resp., spectral function)  $\sigma$ for  $\langle A;\sL\rangle$
is called {\it orthogonal} if $
\ran \cF=L^2(d\sigma)$.
\end{enumerate}
\end{definition}
This classification is not common. It is consistent with those in~\cite{Sak92} and \cite{Mog15MFAT}, while differs from that in  \cite{ArDy12},
where the term pseudo-spectral is used for $q$-pseudo-spectral,
and from \cite{Kac03}, where the term quasi-spectral is used for
pseudo-spectral functions in the sense of Definition~\ref{def:PseudoSF}.

\begin{remark}\label{rem:Directing}
If a gauge $\sL(\subset\sH_-)$  for $A$ is such that $\dR\subset\rho(A,\sL)$,
then the mapping
\begin{equation}\label{eq:GenFourierTr}
\cF:\sH\ni f\mapsto \cP(\cdot)f
\end{equation}
is a {\it directing mapping} for $A$ 
in the sense of \cite{Kr48}, \cite{La70,LaTe78, LaTe84}, i.e.
\begin{enumerate}
  \item [(1)] for each ${z}\in \dR$ the mapping $h\mapsto \cP({z}) h$ is linear;
  \item [(2)] for each $h\in \sH$ the mapping ${z}\mapsto \cP({z})h$ is holomorphic on $\dR$;
  \item [(3)] for ${z}\in \dR$ and $h\in \sH$ we have $\cP({z}) h=0$ if and only if
 $h\in \ran(A-{z} I)$;
  \item [(4)] for at least one ${z}\in \dR$ it holds $\dim\cP(z)\sH=p$.
\end{enumerate}

Indeed, (1)--(2) are evident, (3) follows  from the definition~\eqref{eq:P_lambda} of $\cP({z})$
and the equality $\ran({\mathbf A}-{z} I)\cap\sH=\ran(A-{z} I)$.

Assume that $\text{dim }\cP({z})\sH=k<p$ for some ${z}\in\dR$. Since $\sH=\ran(A-{z} I)\dotplus \sN_{\overline{z}}$ and
$\text{dim }\sN_{\overline{z}}=p$,
we get
$\text{dim }\cP({z})\sN_{\overline{z}}=k<p$ and so
there exists a vector $h\in\sN_{\overline{z}}\setminus\{0\}$ such that
$\cP({z})h=0$. Hence $h\in\sN_{\overline{z}}\cap \ran(A-{z} I)=\{0\}$ what contradicts the fact that $h\ne 0$.
\end{remark}

\begin{lemma}\label{lem:P-dir_func}
Let $\sL(\subset\sH_-)$ be a gauge for $A$ such that $\dR\subset\rho(A,\sL)$,
 let  $L:\dC^p\to\sL$ be a one-to-one linear mapping from $\dC^p$ onto $\sL$,
  let the mapping $\cP({z}):\sH_-\to\dC^p$ be defined by \eqref{eq:P_lambda},
 let  $E$ (resp., $\wh E$) be the generalized (resp., extended generalized) spectral function of a minimal selfadjoint extensions $\wt A$ of $A$, see Subsection~\ref{sec:3.3}, and let $\wt R_z=(\wt A-zI_{\wt\sH})^{-1}$.
Then
\begin{enumerate}
\def\labelenumi{\rm (\roman{enumi})}
\item
For every finite interval $\Delta\subset\dR$ and $f\in\sH$
\begin{equation}\label{eq:Expansion_f}
  \wh E(\Delta)f=\int_{\Delta}\,d\wh E_{\lambda}
  L\cP(\lambda)f.
\end{equation}
\item
For $\sigma(\Delta)$ given for a finite interval $\Delta\subset\dR$ by
\begin{equation}\label{eq:sigma}
  \sigma(\Delta)=L^{\langle*\rangle} \wh E(\Delta)L
\end{equation}
(see Lemma~\ref{lem:ExtSpectralF}(vi)) the following equality holds for all $f\in\sH$
\begin{equation}\label{eq:Parseval}
  (f, E(\dR)f)_\sH =\int_{\dR}\,(\cP({\lambda})f)^*d\sigma({\lambda})\cP({\lambda})f.
\end{equation}

\item
For all $f\in\sH$ it holds
\begin{equation}\label{eq:IntRepR3}
(\wt R_z f,g)_\sH= \int_{\dR}\frac{1}{{\lambda}-z}
(\cP({\lambda})g)^*d\sigma({\lambda})\cP({\lambda})f.
\end{equation}
\end{enumerate}
\end{lemma}
\begin{proof}
(i) The proof of the statement (i) is standard and follows the ideas of \cite{Kr48} (see also \cite[Section 129]{AG66})
except that here we use  the geometry of rigged Hilbert spaces.
For $\alpha<\mu<+\infty$ let us consider
\[
w_\mu=\int_\alpha^\mu d\wh E_{\lambda} f-\int_{\alpha}^\mu\,d\wh E_{{\lambda}} L\cP({\lambda})f
\]
as a vector-function of $\mu$. Clearly, $w_\mu$ is left-continuous.
To prove that $w_\mu$ is right-continuous represent the vector
$f-L\cP(\mu)f$ in the form
\[
f-L\cP(\mu)f=g'-\mu g,\quad\text{where}\quad
\text{\rm col} \{g, g'\}\in {\mathbf A}.
\]
Let $\Delta_\mu^\varepsilon=[\mu,\mu+\varepsilon)$.
By Lemma~\ref{lem:ExtSpectralF}(iv) we get
\begin{equation}\label{eq:wtEdelta1}
\wh E(\Delta)(f-L\cP(\mu)f)=\wh E(\Delta)(g'-\mu g)
=P_\sH(\wt A_{\rm op}-\mu)\wt E(\Delta)g.
\end{equation}
Then for $\Delta=\Delta_\mu^\varepsilon=[\mu,\mu+\varepsilon)$ we obtain
\[
\lim_{\varepsilon\downarrow 0}\wh E(\Delta_\mu^\varepsilon)(f-L\cP(\mu)f)
=\lim_{\varepsilon\downarrow 0}(\wt A_{\rm op}-\mu)\wt E(\Delta_\mu^\varepsilon)g=0
\]
and hence $w_\mu$ is right-continuous.

Now let us prove that $\frac{d}{d\mu}w_\mu\equiv 0$ in the sense that
\begin{equation}\label{eq:Estim0}
\lim_{\varepsilon\downarrow 0}\frac{1}{\varepsilon^2}\|w_{\mu+\varepsilon}-w_{\mu}\|^2=0.
\end{equation}
Assume that $\varepsilon>0$. Using the fact that $w_{\mu+}=w_{\mu}$ we obtain
\[
    \|w_{\mu+\varepsilon}-w_{\mu}\|  =
    \| \wt E(\Delta_{\mu+}^\varepsilon)f-\int_{\mu+}^{\mu+\varepsilon}\,d\wh E_{{\lambda}} L\cP({\lambda})f \|\le X_1+X_2,
 \]
where
\[
X_1=\| \wt E(\Delta_{\mu+}^\varepsilon)(f-L\cP(\mu)f)\|,\quad
X_2=\| \wt E(\Delta_{\mu+}^\varepsilon)L\cP(\mu)f-\int_{\mu+}^{\mu+\varepsilon}\!d\wh E_{{\lambda}} L\cP({\lambda})f \|.
\]
By~\eqref{eq:wtEdelta1},
\begin{equation}\label{eq:Estim1}
 \begin{split}
  X_1^2&=\|P_\sH(\wt A_{\rm op}-\mu)\wt E(\Delta_{\mu+}^\varepsilon)g\|^2\le\int_{\mu+}^{\mu+\varepsilon}|{\lambda}-\mu|^2(d\wt E_{\lambda} g,g)
  \le \varepsilon^2\| \wt E(\Delta_{\mu+}^\varepsilon)g \|^2,
  \end{split}
\end{equation}
and
\begin{equation}\label{eq:Estim2}
 \begin{split}
 X_2^2&=\| \int_{\mu+}^{\mu+\varepsilon}\,d\wh E_{{\lambda}}L(\cP(\mu)f- \cP({\lambda})f) \|^2\le C^2 \varepsilon^2\|\wt E(\Delta_{\mu+}^\varepsilon)L\cP(\mu)f\|^2,
 \end{split}
 \end{equation}
 where $C$ is some positive constant. The equality \eqref{eq:Estim0} follows from \eqref{eq:Estim1} and \eqref{eq:Estim2}. This proves the identity $w_\mu\equiv 0$ and hence also (i).

(ii)  For $f,g\in\sH$ and a finite interval $\Delta\subset\dR$ we get from \eqref{eq:Expansion_f} and \eqref{eq:sigma}
 \begin{multline}\label{eq:Isom1}
    (f,\wt E(\Delta)g) = \int_{\Delta}\,d(f,\wh E_{\lambda} L\cP(\lambda)g)_\sH
      = \int_{\Delta}\,d(\wt E_{\lambda}f, L\cP(\lambda)g)_\sH \\
      = \int_{\Delta}\,d{\langle\wh E_{\lambda} L\cP(\lambda)f, L\cP(\lambda)g)_\sH\rangle} = \int_{\Delta}\,(\cP(\lambda)g)^*d\sigma(\lambda)\cP(\lambda)f.
 \end{multline}
The equality~\eqref{eq:Parseval} is obtained by setting in \eqref{eq:Isom1} $g=f$ and applying the Fatou's lemma.

(iii) follows from \eqref{eq:Parseval} in a similar way when replacing $f$ with $\wt R_zf$.
\end{proof}

\begin{theorem}
  \label{thm:Descr_Pseudo-spect}
Let the assumptions of Lemma~\ref{lem:P-dir_func}  hold and let $\sigma$
be given by~\eqref{eq:sigma}. Then
\begin{enumerate}
\def\labelenumi{\rm (\roman{enumi})}
\item 
$\sigma$ is an LT-spectral function of $\langle A;\sL\rangle$.
    \item
    $\sigma\in \Sigma_{qpsf}(A;\sL)$ if and only if
\begin{equation}\label{eq:Adm_q}
(\mul A\subseteq)\mul \wt A\subset\sH.
\end{equation}
\item $\sigma\in \Sigma_{psf}(A;\sL)$ if and only if
\begin{equation}\label{eq:Adm1}
\mul \wt A=\mul A.
\end{equation}
\item The set $\Sigma_{sf}(A;\sL)$ of spectral functions of $\langle A;\sL\rangle$ is nonempty if and only if $\mul A=\{0\}$ and in this case  $\Sigma_{sf}(A;\sL)=\Sigma_{psf}(A;\sL)$.
    \item
    If \eqref{eq:Adm1} holds then $\sigma$ is an orthogonal pseudo-spectral function  of $\langle A;\sL\rangle$ if and only if
        $\wt A$ is a canonical selfadjoint extension of $A$.
\end{enumerate}
\end{theorem}
\begin{proof}
(i) Let $\wt E(\cdot)$ be the resolution of the identity of the operator part
of the linear relation $\wt A$.
Notice that \eqref{eq:Parseval} yields
 the inequality 
\begin{equation}\label{eq:BesselIneq2}
    \int_{\dR}(\cP({\lambda}) f)^*d\sigma({\lambda})\cP({\lambda}) f
    =(  \wt E(\dR)f,  f)_{\wt\sH}
  \le ( f,  f)_{\sH},
\end{equation}
which is valid for all $f\in\sH$. This inequality can also be rewritten as
\begin{equation}\label{eq:BesselIneq3}
  \|\cF f\|_{L^2(d\sigma)}= \| \wt E(\dR) f\|_{\wt\sH}^2\le
   \|  f\|_{\sH}^2 ,\quad f\in\sH,
\end{equation}
which implies (i).

(ii) If $\mul \wt A\subset\sH$ then for all $f\in\mul \wt A$ we get
$\wt E(\dR) f=0$ and hence, by~\eqref{eq:BesselIneq3}, $\mul \wt A\subset\ker\cF$.
For $f\in\sH\ominus\mul \wt A$ we get
$\wt E(\dR) f=f$,
$\|\cF f\|_{L^2(d\sigma)}=\|f\|_{\sH}$ and hence
$\cF:L^2(d\sigma)\to \sH$ is a partial isometry with the kernel
$\ker\cF=\mul \wt A$.

Conversely, if  $\sigma\in \Sigma_{qpsf}(A;\sL)$, then $\cF:\sH\to L^2(d\sigma)$
is a partial isometry with the kernel $\cK_0=\mul\wt A\cap\sH$.
 Then for $h\in\cK_0^\perp:=\sH\ominus \cK_0$ we have the equality $\|\cF h\|_{L^2(d\sigma)}=\|h\|_\sH$ and hence $\cK_0^\perp\perp\mul\wt A$.
 Decomposing $\mul\wt A=\cK_0\oplus\cK_1$ we obtain
 \[
 \cK_1\perp\cK_0\quad\text{and}\quad \cK_1\perp\cK_0^\perp\quad
 \Longrightarrow\quad \cK_1\subset\wt\sH\ominus\sH.
 \]
Since $\wt A$ is a minimal extension of $A$, we get $\cK_1=\{0\}$
and hence  $\mul\wt A\subset\sH$.

(iii) Since $\ker\cF=\mul \wt A$, we have the equivalence
\[
\ker\cF=\mul A\quad\Longleftrightarrow \quad\mul \wt A=\mul  A.
\]
Therefore, $\sigma$ is a pseudo-spectral function of $\langle A;\sL\rangle$ if and only if
\eqref{eq:Adm1} holds.

(iv) Since $\mul A\subseteq \mul \wt A$, the condition $\mul A=\{0\}$ is necessary for
$\sigma$ to be a spectral function of $\langle A;\sL\rangle$. The condition $\mul A=\{0\}$ is also sufficient for $\Sigma_{sp}(A;\sL)$ to be non-empty since a single-valued operator $A$ with equal defect numbers admits a single-valued selfadjoint extension $\wt A$.

(v) If $\sigma\in\Sigma_{psf}(A;\sL)$ and $\ran \cF=L^2(d\sigma)$, then
$\cF_0=\cF\upharpoonright \sH_{\rm op}$ is a unitary operator
 from $\sH_{\rm op}=\sH\ominus\mul A$ to $L^2(d\sigma)$ and
 $A_{\rm op}=\cF_0^* \Lambda\cF_0\upharpoonright \dom A_{\rm op}$,
 where $\Lambda$ is the multiplication operator in $L^2(d\sigma)$.
 Therefore, $\cF_0^* \Lambda\cF_0$ is a selfadjoint extension of
 $A_{\rm op}$ in $\sH_{\rm op}$ and the linear relation $\wt A$
with the operator part $\wt A_{\rm op}=\cF_0^* \Lambda\cF_0$
is a canonical  selfadjoint extension of $A$.

 Conversely, let $\wt A$ be a canonical  selfadjoint extension of $A$
 such that \eqref{eq:Adm1} holds, denote
 $\sG=\ran \cF\subset L^2(d\sigma)$ and consider
 $\cF_0=\cF\upharpoonright \sH_{\rm op}$ as a unitary operator
 from $\sH_{\rm op}=\sH\ominus\mul A$ to $\sG\subseteq L^2(d\sigma)$.
 Then the operator $G:=\cF_0\wt A_{\rm op}\cF_0^*$ is
  selfadjoint in $\sG$ and for its resolvent
  $R_z^G=(G-z I_{\sG})^{-1}$
  and functions $\varphi=\cF_0 f$ and $h=\cF_0g$, $f,g\in\sH_{\rm op}$
  we obtain
\begin{equation}\label{eq:IntRepRG}
(R_z^G \varphi,h)_{\sG}=
(\wt R_z f,g)_\sH= \int_{\dR}\frac{1}{{\lambda}-z}
\psi({\lambda})^*d\sigma({\lambda})\varphi({\lambda}).
\end{equation}

Assuming that $h\in L^2(d\sigma)\ominus\sG$ we obtain
$(R_z^G \varphi,h)_{\sG}=0$ for all $\varphi\in\sG$.
By~\eqref{eq:IntRepRG}, we get
\[
\int_{\dR}\frac{1}{{\lambda}-z}
h({\lambda})^*d\sigma({\lambda})\varphi({\lambda})=0\quad
\text{for all $\varphi\in\sG$}
\]
and hence, by the Stieltjes inversion formula, for every interval $\Delta\subset\dR$
\[
\int_{\Delta}\,
h({\lambda})^*d\sigma({\lambda})\varphi({\lambda})=0\quad
\text{for all $\varphi\in\sG$}.
\]
Since the set $\{\cP(\cdot)f{\mathbf 1}_\Delta:f\in\sH, \,\Delta\subset\dR\}$
is dense in $L^2(d\sigma)$, we obtain $h=0$ in $L^2(d\sigma)$,
and thus $\sG=\ran \cF= L^2(d\sigma)$.
\end{proof}

\begin{remark}
\begin{itemize}
  \item   For a single-valued operator $A$ formulas \eqref{eq:Expansion_f} and \eqref{eq:Isom1} were proved in \cite{Sh71}
      under a  weaker assumption that $\cP({\lambda})f$ admits an analytic continuation into a neighborhood of $\Delta$.
      \item For a  single-valued operator $A$ with defect numbers $n_\pm(A)=1$
      the statement (v) in Theorem~\ref{thm:Descr_Pseudo-spect} was proved in~\cite{Kr48} as a criterion of completness of the set $\{\cP(\lambda)f\}$
      in $L^2(d\sigma)$,
      where $f$ ranges over the set of finite functions.
  \item If $\sigma$ is an orthogonal pseudo-spectral function for
  $\langle A,\sL \rangle$, then Lemma~\ref{lem:P-dir_func} and formula \eqref{eq:BesselIneq3} show that the mappings
\begin{equation}\label{eq:Expansion_f2}
  f=\int_{\Delta}\,d\wh E_{{\lambda}} L(F({\lambda})),\quad F\in L^2(d\sigma),
\end{equation}
and
\begin{equation}\label{eq:Expansion_f3}
  F({\lambda})=\cP({\lambda})f,\quad f\in \sH\ominus\mul A,
\end{equation}
define an isometry from $L^2(d\sigma)$ onto $\sH\ominus\mul A$.
\end{itemize}
\end{remark}

\begin{theorem}
  \label{thm:Descr_Pseudo_IS}
  Let $\sL(\subset\sH_-)$ be a gauge for $A$ such that $\dR\subset\rho(A,\sL)$,
 let $L:\dC^p\to\sL$ be a one-to-one linear mapping from $\dC^p$ onto $\sL$.
Then
\begin{enumerate}
\def\labelenumi{\rm (\roman{enumi})}
\item
The formula
\begin{equation}\label{eq:IntRepR2}
 \int_{\dR}\left(\frac{1}{{\lambda}-z}-\frac{{\lambda}}{1+{\lambda}^2}\right)d\sigma({\lambda})=
 L^{\langle*\rangle}\wh{\bm R}_{z} L-\wh K
\end{equation}
establishes a bijective correspondence between
all LT-spectral functions $\sigma$  of $\langle A;\sL\rangle$
and all $\sL$-resolvents
$L^{\langle*\rangle}\wh{\bm R}_{z} L$ of $A$. Here $\wh K$ is a Hermitian ${p\times p}$ matrix depending on $\sigma.$
\item
$\sigma\in \Sigma_{qpsf}(A;\sL)$
if and only if the corresponding
minimal representing relations $\wt A$ satisfy the condition~\eqref{eq:Adm_q}.

\item
$\sigma\in \Sigma_{psf}(A;\sL)$
if and only if the corresponding
minimal representing relations $\wt A$ satisfy the condition~\eqref{eq:Adm1}.

\item
If $\mul A=\{0\}$, then $\sigma\in \Sigma_{sf}(A;\sL)$
if and only if the corresponding
minimal representing relations $\wt A$ satisfy the condition~\eqref{eq:Adm1}.
\end{enumerate}
\end{theorem}

\begin{proof}
(i) Let $\sigma\in \Sigma_{LT}(A;\sL)$.
  Since, by Remark~\ref{rem:Directing}, the mapping
  $\sH\ni f\mapsto \cP({z})f\in\dC^p$ with $\cP({z})$ defined by \eqref{eq:P_lambda} is a directing mapping,
  it follows from \cite[Theorem~1]{LaTe84} that there exists
  a minimal selfadjoint extensions $\wt A$ of $A$ such that
\eqref{eq:Adm1} holds and
\begin{equation}\label{eq:Descr_LTResolv3}
(\wt{ R}_z f,f)_{\wt\sH}= \int_{\dR}\frac{1}{{\lambda}-z}
(\cP({\lambda})f)^*d\sigma({\lambda})\cP({\lambda})f \quad \text{for all}\quad f\in\sH,
\end{equation}
where $\wt R_z=(\wt A-z I_{\wt\sH})^{-1}$.
Hence we obtain for the regularized resolvent
$\wh R_z=\wt R_z-\wt R$ with $\wt R:=\frac12(\wt R_i+\wt R_{-i})$
\begin{equation}\label{eq:Descr_LTResolv4}
  \langle \wh R_z  f,  f\rangle_{\sH}=
  \int_{\dR}\left(\frac{1}{{\lambda}-z}-\frac{{\lambda}}{{\lambda}^2+1}\right)
  (\cP({\lambda})f)^*d\sigma({\lambda})\cP({\lambda})f
  , \quad \text{for all}\quad f\in\sH.
\end{equation}
Let $\widehat{\bm R}_z \in \cB(\sH_-, \sH_+)$ be the regularized extended generalized resolvent of $A$, see Lemma~\ref{lem:3.2B}(iii).
The operator $\widehat{\bm R}_z $ is related with  $\wh R_z$ by the formula
\begin{equation}\label{eq:Regul_0-wt}
  \widehat{\bm R}_z\upharpoonright\sH=P_{\sH}\wh R_z\upharpoonright\sH
  +\wh R,\quad\text{where}\quad
  \widehat{ R}=P_{\sH}\wt R\upharpoonright\sH- R^0\in \cB(\sH, \sH_+).
\end{equation}
By Lemma~\ref{lem:3.2B}(iv),  $\wh{\bm R}=\wh R^{\langle*\rangle}\in \cB(\sH_-, \sH_+)$.
 Since
{ $\sH$ is dense in $\sH_-$}, for every ${\xi}\in\dC^p$ there exists a sequence $f_n\in\sH$ such that $f_n\to L{\xi}$ in $\sH_-$. Therefore,
there exists the limit
\begin{equation}\label{eq:LimRegRes}
 N_z({\xi}):= \lim_{n\to\infty} \langle \wh R_z  f_n,  f_n\rangle_{+,-}
 =\langle \wh {\bm R}_z L{\xi},L{\xi}\rangle_{\sH}
 -\langle  \wh{\bm R} L{\xi},L{\xi}\rangle_{\sH}.
\end{equation}

On the other hand, since $\cP({\lambda})\in \cB(\sH_-,\dC^p)$, we have
\[
\cP({\lambda})( f_n)\to  \cP({\lambda})L{\xi}={\xi}.
\]
By \eqref{eq:Descr_LTResolv4}, we get
\[
\int_{\dR}\frac{(\cP({\lambda}) f_n)^*d\sigma({\lambda})\cP({\lambda}) f_n}{{\lambda}^2+1}
= \Im \langle \wh R_i f_n,f_n\rangle_{\sH}.
\]
Next by the Fatou's lemma  and \eqref{eq:LimRegRes}, we get for every $c>0$
\[
\int_{-c}^c\frac{{\xi}^* d\sigma({\lambda}){\xi}}{{\lambda}^2+1}
=\lim_{n\to\infty}\int_{-c}^c
\frac{(\cP({\lambda}) f_n)^*d\sigma({\lambda})\cP({\lambda}) f_n}{{\lambda}^2+1}
\le \Im \langle \wh {\bm R}_i L{\xi},L{\xi}\rangle_{\sH}.
\]
Now applying  the Lebesgue's dominated convergence theorem to
the equality
\begin{equation}\label{eq:Descr_LTResolv5}
  \langle \wh R_z f_n, f_n\rangle_{\sH}=
  \int_{\dR}\left(\frac{1}{{\lambda}-z}-\frac{{\lambda}}{{\lambda}^2+1}\right)
  (\cP({\lambda}) f_n)^*d\sigma({\lambda})\cP({\lambda}) f_n
\end{equation}
and taking into account \eqref{eq:LimRegRes} we obtain
\begin{equation}\label{eq:Descr_LTResolv6}
  \langle \wh{\bm R}_z L{\xi}, L{\xi}\rangle_{\sH}
  -\langle \wh{\bm R} L{\xi},L{\xi}\rangle_{\sH}=
  \int_{\dR}\left(\frac{1}{{\lambda}-z}-\frac{{\lambda}}{{\lambda}^2+1}\right)
  {\xi}^* d\sigma({\lambda}){\xi}.
\end{equation}
Setting $\wh K=L^{\langle*\rangle}\wh{\bm R} L$ we obtain~\eqref{eq:IntRepR2}.

Conversely,  let $\wt A$  be a minimal selfadjoint extensions of $A$.
By Lemma~\ref{lem:P-dir_func}, the mvf $\sigma$ given by \eqref{eq:sigma} is an LT-spectral function of $\langle A;\sL\rangle$.
Therefore, by the 1-st part of the proof, $\sigma$ satisfies~\eqref{eq:IntRepR2}.

 (ii)--(iv) 
 follow from statements (ii)--(iv) of
Theorem~\ref{thm:Descr_Pseudo-spect}.
\end{proof}

\section{Canonical systems of differential equations}\label{sec:6}
Consider the canonical differential equation  
   \begin{equation}\label{eq:can,eq-n}
J_p f'(t)+\cF(t)f(t) = {z}\mathcal H(t)f(t), \quad  t\in(a,b),\quad
  \end{equation}
where $ {z} \in {\dC}$,  $f(\cdot)$  is a ${\dC}^{n}$-vector function, $n=2p$,
$J_p$ is defined by~\eqref{kerK0}, and ${n\times n}$ matrix-functions $\cF(t)$ and $\cH(t)$ satisfy the assumptions
\begin{enumerate}
  \item [(A1)] $\cF(t)$ and $\cH(t)$ are real Hermitian matrix-functions
with entries from $L^1_{\rm loc}(a,b)$, $\text{tr }\cH(t)\equiv 1$ and  $\cH(t)\ge 0$ for a.e. $t\in I:=(a,b)$.
  \item [(A2)]  For each absolutely continuous $f$ such that $J_p f'(t)+\cF(t)f(t)=0$ the following implication holds
  \[
  \cH(t)f(t)=0\quad \text{a.e. }\Longrightarrow f\equiv 0 \quad \text{on}\quad I=(a,b).
  \]
\end{enumerate}
 An endpoint of the interval $(a,b)$ is said to be a {\it quasiregular} endpoint
of the canonical system~\eqref{eq:can,eq-n} if $\cF$ and $G$
 are integrable up to that endpoint. A finite quasiregular endpoint is called
{\it regular}.

Let $\cL^2_\cH(a,b)$ be the semi-Hilbert space of measurable $\dC^{n}$-valued functions $f$, such that $\langle f,f\rangle_\cH :=\int_a^b f(t)^*\cH(t)f(t)dt<\infty$.
The semi-definite inner product in $\cL^2_\cH(a,b)$  corresponding to the semi-norm $\|f\|_\cH :=\langle f,f\rangle_\cH^{1/2}$ is defined by
\begin{equation}\label{eq:InnerPr}
  \langle f,g\rangle_\cH:=\int_a^b g(t)^*\cH(t)f(t)dt.
\end{equation}
Let $L^2_\cH(a,b)$ be the factor-space
\[
L^2_\cH(a,b)=\cL^2_\cH(a,b)/\{f\in \cL^2_\cH(a,b):\langle f,f\rangle_\cH=0\}.
\]
 For a function $f\in \cL^2_\cH(a,b)$ we denote by $\pi_{\cH} f$ the corresponding class in $L^2_\cH(a,b)$.
  Clearly,  $L^2_\cH(a,b)$ is a Hilbert space with respect to the inner product
  \[
  ( \pi_{\cH}f, \pi_{\cH}g)_{L^2_{\cH}} = \langle f,g\rangle_\cH,\quad f,g\in \cL^2_\cH(a,b).
  \]
   Define  the maximal relation $S_{\rm max}$ in  $L^2_\cH(a,b)$ by
  \[
 S_{\rm max}=\left\{
\text{\rm col}\{\pi_{\cH}f,\pi_{\cH}g\}\in L^2_\cH(a,b)\times L^2_\cH(a,b):\,
                    J_p f'+\cF f = \cH g\right\}.
  \]
As is known,  
for
  $\text{\rm col}\{\pi_{\cH}f, \pi_{\cH}g\}\in S_{\rm max}$ the  class $\pi_{\cH}f$
  has a unique absolutely
 continuous representative $f\in AC[a,b]$, \cite{Orc67}, see also~\cite{LaTe82}.
For not necessarily definite $2\times 2$  canonical systems such regularity result was
proven in~\cite{Kac83} and for $n\times n$ systems by another method in \cite{LesMal03}.

Let  the pre-minimal relation $S'$ be defined as the
 restriction of the maximal relation $S_{\rm max}$ to the elements
 $\text{\rm col} \{\pi_{\cH}f,\pi_{\cH}g\}$ such that $f$ has compact support on $(a,b)$
and let the
 minimal relation $S_{\rm min}$ be defined as  $S_{\rm min}=\overline{S'}$.

\subsection{Spectral and pseudo-spectral functions of $S_{\min}$ in the regular case}
 Assume that  $a$ and $b$  are regular or
quasiregular endpoints of the canonical system~\eqref{eq:can,eq-n}.
Then for $g\in\cL^2_{\cH,{\rm loc}}(a,b) $  each solution $f$ of the equation
$J_p f'+\cF f = \mathcal H g$
 is square-integrable with respect to
$\cF$ on $I$ and the limits
\[
f(a)=\lim_{t\downarrow a}f(t)\quad\text{and}\quad
f(b)=\lim_{t\uparrow b}f(t)
\]
exist. As is known, see~\cite{{LaTe82}},
 $A:=S_{\rm min}$ is a symmetric relation with defect numbers $n_\pm(A)=n=2p$,
 $S_{\rm max}=S_{\rm min}^*$ and
  \[
 S_{\rm min}=\big\{\text{\rm col}\{\pi_{\cH}f,\pi_{\cH}g\}\in S_{\rm max}:\,
                     f(a)= f(b) = 0\big\}.
  \]

Denote by $U(\cdot,{z})$ the fundamental ${n\times n}$ matrix solution of the initial problem
\begin{equation}\label{Intro_canon_system_second}
   J_p\frac{dU(t,{z})}{dt}+\cF(t)U(t)= {{z}}\cH U(t,{z}),\quad \text{a.e. on}\quad(a,b),
   \quad  U(a,{z})= I_n.
\end{equation}
The matrix function $U({z}):=U(b,{z})$ is called the monodromy matrix of the system~\eqref{eq:can,eq-n}.
We will drop $\pi_\cH$ sometimes if it does not lead to a confusion.

\begin{proposition}
 \label{prop:14.51}
 Let the assumptions (A1)-(A2) hold, let $A=S_{\rm min}$ be the minimal relation associated with the canonical system~\eqref{eq:can,eq-n} and let
 $U({z})$ be the monodromy matrix of the system~\eqref{eq:can,eq-n}. Then
\begin{enumerate}
  \item [\rm(i)]
As a boundary triple
$\Pi = ({\dC}^n,\Gamma_0,\Gamma_1)$ for $A^*=S_{\rm max}$ one can take
\begin{equation}\label{eq:BT_Can}
\Gamma_0 \begin{bmatrix}
              f \\
               g
             \end{bmatrix}=\frac{1}{\sqrt{2}}\bigl(f(a) + f(b)\bigr),\quad
\Gamma_1 \begin{bmatrix}
              f \\
               g
             \end{bmatrix}= -\frac{1}{\sqrt{2}}J_p\bigl(f(a) - f(b)\bigr), \quad
\begin{bmatrix}
              f \\
               g
             \end{bmatrix}\in S_{\rm max}.
\end{equation}
  \item [\rm(ii)] The corresponding Weyl function is given by
    \begin{equation}\label{eq:M_lambda}
M({z}) = -J_p(I_n - U({z}))(I_n+U({z}))^{-1}.
 \end{equation}
 \item [\rm(iii)] The $\gamma$-field is
     \begin{equation}\label{eq:gamma_lambda}
\gamma({z}) = \sqrt{2}U(\cdot,{z})(I_n+U({z}))^{-1}.
 \end{equation}
  \item [\rm(iv)] The adjoint to the $\gamma$-field is
     \begin{equation}\label{eq:gamma_*}
\gamma(\ov{z})^*f = \sqrt{2}\int_a^b(I_n+U^\#({z}))^{-1}U^\#(s,{z})\cH(s)f(s)ds.
 \end{equation}
   \item [\rm(v)] The resolvent $R^0_{z}=(S_0-{z} I_\sH)^{-1}$ of the linear relation $S_0:=\ker \Gamma_0$ takes the form
     \begin{equation}\label{eq:R0lambda}
(R^0_{z} f)(t) = \frac12 U(t,{z})\int_a^b\left\{\sgn(s-t) J_p-J_p M({z})J_p\right\}U^\#(s,{z})\cH(s)f(s)ds.
 \end{equation}
\end{enumerate}
\end{proposition}
\begin{proof}
  Let us give a sketch of the proof. For more details, see \cite[Theorem 7.7.2]{BeHaSn20}.

(i) is straightforward.

(ii)--(iii) The defect subspace $\sN_{z}$ of $A$ consists of vvf's
$
f_{z}(\cdot)=U(\cdot,{z})u$, $u\in\dC^{n}.
$
Since
\[
\Gamma_0 \wh f_{z}= \frac{1}{\sqrt{2}}\bigl(I_n + U({z})\bigr)u,\quad
\Gamma_1 \wh f_{z}= -\frac{1}{\sqrt{2}}J_p\bigl(I_n - U({z})\bigr)u,
\]
we obtain \eqref{eq:M_lambda} and \eqref{eq:gamma_lambda}.

(iv) To get \eqref{eq:gamma_*} let us take $f\in\cL^2_\cH(I)$ and $v\in\dC^n$. Then
\[
v^*\left(\gamma(\ov{z})^*f\right)=(f,\gamma(\ov{z})v)_{\cL^2_\cH}
=\sqrt{2}v^*\left(\int_I (I_n+U(\ov{z})^*)^{-1}U(s,\ov{z})^*\cH(s)f(s)ds)\right).
\]

(v) follows from ~\cite[Theorem II.3.5]{Orc67}, see also~\cite[eq. (3.1)]{LaTe82}.
\end{proof}

\begin{proposition}
 \label{prop:PreresolvMatrix}
Let the assumptions (A1)-(A2) hold, let $A=S_{\rm min}$ be the minimal relation associated with the canonical system~\eqref{eq:can,eq-n}, let
 $U({z})$ be the monodromy matrix of the system~\eqref{eq:can,eq-n}, let
 ${\sL}$ be given by
     \begin{equation}\label{eq:G_frak}
{\sL}:=\left\{\delta_a\otimes \xi:\, \xi\in\dC^n\right\}
 \end{equation}
 and let the operator ${L}:\dC^n\to{\sL}$
be defined by ${L}u={\sqrt2}\delta_a\otimes u$, $u\in\dC^n$. Then
\begin{enumerate}
\item[(i)] $\rho(S_{\rm min},{\sL})=\dC$.

\item[(ii)] The operator-function $\cP({z})$ is given  by
\begin{equation}\label{eq:P_lambda_Min}
  (\cP({z}))(\pi_{\cH}f)
		= \frac{1}{\sqrt2}
\int_a^b U^\#(s,{z})\, dH(s)\, f(s).
\end{equation}

\item[(iii)]
The $\Pi{\sL}$-preresolvent matrix $\sA_{\Pi{\sL}}({z})$ takes the form
    \begin{equation}\label{eq:PiG_Res_M}
\sA_{\Pi{\sL}}({z}) =\begin{bmatrix}
                         M({z})         & {2}(I_n+U^\#({z}))^{-1} \\
                         {2}(I_n+U({z}))^{-1} & J_p (-M({z})+\Re  M(i))J_p
                       \end{bmatrix}, \quad {z}\in\rho(S_0),
 \end{equation}
 where $M({z}) = -J_p(I_n - U({z}))(I_n+U({z}))^{-1}$.
 \end{enumerate}
\end{proposition}
\begin{proof}
(i)
The operators  $\Gamma_0$ and $\Gamma_1$ are bounded as operators from
$A^*$ to $\dC^n$, see~\cite{DM95}, and hence the mapping
$f\in A^*\mapsto f(a)\in\dC^p$ is bounded.
Let $\sH_{0,+}=\dom A^*=\dom A_0^*$ be the Hilbert space with the norm~\eqref{E:3.1a}.
Since $\mbox{gr }A_0^*\subset A^*$, for every $\xi\in\dC^n$, see Section~\ref{sec:3.1} the
anti-linear functional
\[
\langle \delta_a\otimes \xi,f\rangle_{-,+}=f(a)^*\xi,\quad f\in \sH_{0,+},
\]
is bounded on $\sH_{0,+}$, i.e. $\sL\subset\sH_{0,-}\subset \sH_-$, see also~\cite[Lemma II.4.1]{{LaTe82}}. The subspace ${\sL}$
of $\sH_-$ is disjoint with $\ran(\bA-{z} I_\sH)$ since, otherwise, there is $\xi\in\dC^n$
such that $\delta_a\otimes \xi\in \ran(\bA-{z} I_\sH)$ and, by Lemma~\ref{Lem:3.3} (vi), we get
\[
0=\langle \delta_a\otimes \xi,f_{\ov{z}}\rangle_{-,+}=f_{\ov{z}}(a)^*\xi,\quad\text{for all }\quad
f_{\ov{z}}\in \sN_{\ov{z}}.
\]
This implies $\xi=0$, and therefore the subspaces
${\sL}$ and  $\ran(\bA-{z} I_\sH)$ are disjoint.
Hence $\rho(A,{\sL})=\dC$ and so ${\sL}$ is a gauge for $A$.

(ii) 	It follows from the definition of $\cP({z})$ that $\cP({z})^* \xi \in \sN_{\conj{z}}$. Hence there exists $\eta \in \dC^n$ such that $\cP({z})^* \xi = \pi_{\cH} U(\cdot, \conj{z}) \eta$.
	From the equalities
	\begin{equation*}
		{L}^{\langle*\rangle} \cP({z})^* \xi = \xi,
		\quad
		L^{\langle*\rangle} \pi_{\cH}U(\cdot, \conj{z}) \eta = {\sqrt2} \eta,
	\end{equation*}
	we get $ {\sqrt2}\eta = \xi$.
	\eqref{eq:P_lambda_Min} now follows from the following chain of equalities
	\begin{multline*}
		\xi^* (\cP({z}))(\pi_{\cH}f)
		=(\pi_{\cH}f, \cP({z})^* \xi)_{\sH}\\
		= \frac{1}{\sqrt2}(f,\, U(\cdot,\conj{z})\,\xi )_{\sH}
		= \frac{1}{\sqrt2}\xi^* \int_a^b U^\#(s,{z})\, dH(s)\, f(s).
	\end{multline*}

(iii) It follows from the equality
\[
u^*\left({L}^{\langle *\rangle}f\right)=\langle f,Lu\rangle_{-,+}={\sqrt2} u^* f(a), \quad f\in\sH_+,\quad u\in\dC^p,
\]
that ${L}^{\langle *\rangle}f= {\sqrt2}f(a)$.
  This  equality and \eqref{eq:gamma_lambda} yield the formulas
  for ${\mathfrak a}_{21}({z})$ and ${\mathfrak a}_{12}({z})$
    \begin{equation}\label{eq:a21}
  {\mathfrak a}_{21}({z})={L}^{\langle *\rangle}\gamma({z})={\sqrt2}(I_n+U({z}))^{-1}, \quad {z}\in\rho(S_0),
\end{equation}
    \begin{equation}\label{eq:a12}
  {\mathfrak a}_{12}({z})={\mathfrak a}_{21}^\#({z})={\sqrt2}(I_n+U^\#({z}))^{-1}, \quad {z}\in\rho(S_0).
\end{equation}

Next, by~\eqref{eq:R0lambda} and~\eqref{eq:cR},  we get
for $u\in\dC^p$ and ${z}\in\rho(S_0)$
\begin{equation}\label{eq:Rlambda_delta}
  \wt R^0_{z}(\delta_a\otimes u)
=-\frac12 U(t,{z})(J_p+J_p M({z})J_p)u,
\end{equation}
\begin{multline}\label{eq:Rlambda_delta2}
(\wt R^0_{z}-{\mathbf R}^0)(\delta_a\otimes u)
= -\frac12 U(t,{z})(J_p+J_p M({z})J_p)u\\
+
\frac14\left\{ U(t,i)(J_p+J_p M(i)J_p)+U(t,-i)(J_p+J_p M(-i)J_p)\right\}u
\end{multline}
and hence, by~\eqref{eq:Lres2A} and~\eqref{eq:RegExtRes},
    \[
  {\mathfrak a}_{22}({z})={L}^{\langle *\rangle}\wh R^0_{z} {L}u
  ={L}^{\langle *\rangle} (\wt R^0_{z}-{\mathbf R}^0){L}
  =-J_p M({z})J_p+J_p \text{Re }M(i)J_p.
\]
\end{proof}

\begin{theorem}
 \label{thm:ResolvMatrix}
 Let  the assumptions of Proposition~\ref{prop:PreresolvMatrix} hold. Then
\begin{enumerate}
\item[(i)]
The left $\Pi{\sL}$-resolvent matrix $W_{\Pi{\sL}}^{\ell}({z})$ takes the form
    \begin{equation}\label{eq:PiG_Res_M2}
W_{\Pi{\sL}}^{\ell}({z}) =\frac{1}{2}\begin{bmatrix}
(U({z})-I_n)J_p +(U({z})+I_n)K          & (U({z})+I_n)\\
J_p(U({z})+I_n)J_p +J_p(U({z})-I_n)K    & J_p(U({z})-I_n)
                       \end{bmatrix}, 
 \end{equation}
 where
 $
   K:=J_p \Re M(i)J_p.
 $
 \item[(ii)] The formula \eqref{eq:2.2L}
 establishes a one--to--one correspondence  between the set
of all ${\sL}$-resolvents of $A$  and the set of all
$\cR^{n\times n}$-pairs $\begin{bmatrix}
    C({z})& D({z})
    \end{bmatrix}$.
\medskip
    \end{enumerate}
 \end{theorem}
\begin{proof}
(i)   By~\eqref{eq:LresM_Left} and~\eqref{eq:PiG_Res_M},
we get
\begin{multline}\label{eq:w11}
  w_{11}^{\ell}({z})={\mathfrak a}_{21}({{z}})^{-1}{\mathfrak a}_{22}({{z}})
  =\frac{1}{2}(I+U({z}))
  \left\{(U({z})-I)(I+U({z}))^{-1}J_p+K\right\}\\
=\frac{1}{2}\left\{(U({z})-I_n)J_p +(U({z})+I_n)K\right\}.
\end{multline}
Since $U^\#({z})=-J_p U({z})^{-1}J_p$, we have
\[
(I_n+U^\#({z}))^{-1}=-J_p(I_n+U({z})^{-1})^{-1}J_p
=-J_p U({z})(I_n+U({z}))^{-1}J_p.
\]
Hence, by~\eqref{eq:LresM_Left} and~\eqref{eq:PiG_Res_M},
\begin{multline}\label{eq:w21}
  w_{21}^{\ell}({z})
  ={\mathfrak a}_{11}({{z}}){\mathfrak a}_{21}({{z}})^{-1}{\mathfrak a}_{22}({{z}})-{\mathfrak a}_{12}({{z}})\\
  =\frac{1}{2} J_p
  \left\{\left[(U({z})-I_n)^2+4U({z})\right](U({z})+I_n)^{-1}J_p
  +(U({z})-I_n)K\right\}\\
=\frac{1}{2}\left\{J_p(U({z})+I_n)J_p+J_p(U({z})-I_n)K\right\}.
\end{multline}
Similarly, we get from~\eqref{eq:LresM_Left} and~\eqref{eq:PiG_Res_M}
\begin{equation}\label{eq:w12}
\begin{split}
  w_{12}^{\ell}({z})= {\mathfrak a}_{21}({{z}})^{-1}=\frac{1}{2}(I_n+U({z}))
\end{split}
\end{equation}
and
\begin{equation}\label{eq:w22}
\begin{split}
  w_{22}^{\ell}({z})= {\mathfrak a}_{11}({{z}}){\mathfrak a}_{21}({{z}})^{-1}=\frac{1}{2}J_p(U({z})-I_n).
\end{split}
\end{equation}
Now \eqref{eq:PiG_Res_M2} follows from~\eqref{eq:w11}--\eqref{eq:w22}.

\medskip
(ii) follows from Theorem~\ref{thm:Gres_CD}.
\end{proof}
\begin{theorem}
  \label{thm:Descr_Pseudo_IS3}
  Let $\Pi = (\dC^n, \Gamma_0, \Gamma_1)$ be the boundary triple for $S_{\max}$
  given by~\eqref{eq:BT_Can},
let $\sL=\{\delta_a\otimes v:v\in\dC^p\}$ be a gauge for $S_{\min}$,
 and let
$W_{\Pi{\sL}}^{\ell}({z})$  be the left $\Pi{\sL}$-resolvent matrix of $S_{\min}$ given by
\eqref{eq:PiG_Res_M2}.
Then
\begin{enumerate}
  \item [\rm(i)]
The formula
\begin{equation}\label{eq:IntRepRMin}
  \int_{\dR}\left(\frac{1}{{\lambda}-z}-\frac{{\lambda}}{1+{\lambda}^2}\right)d\sigma({\lambda})=
  T_{W_{\Pi{\sL}}^{\ell}(z)}[\tau(z)]-\wh K
\end{equation}
establishes a bijective correspondence between
all LT-spectral
functions $\sigma$  of $\langle A;\sL\rangle$
and all  families $\tau\in\wt\cR^{n\times n}$. Here $\wh K=\wh K^*$ is a 
${n\times n}$ matrix depending on $\sigma.$
\medskip
\item [\rm(ii)]	$\sigma\in\Sigma_{\rm psf}(A;\sL)$ if and only if $\tau\in\wt\cR^{n\times n}$ is $\Pi$-admissible.
    \medskip
\item [\rm(iii)]	For every pseudo-spectral function $\sigma$ the generalized Fourier transform
\begin{equation}\label{eq:FourierMin}
\cF:\sH\ni f\mapsto F({\lambda})=\frac{1}{\sqrt{2}}\int_a^b U^\#(s,{\lambda})\cH(s)f(s)ds
\end{equation}
is a partial isometry from $\sH$ to $L^2(d\sigma)$
such that $\ker\cF=\mul A$ and for all $f,g\in \sH\ominus\mul A$ the Parseval equality holds
\begin{equation}\label{eq:ParsevalMin}
  \int_a^b g(s)^*\cH(s)f(s)ds=\int_\dR G({\lambda})^*d\sigma({\lambda})F({\lambda}),\quad\text{where}\quad
  G=\cF g.
\end{equation}
\end{enumerate}
\end{theorem}
\begin{proof}
  (i) \& (ii) follow from Theorem~\ref{thm:Descr_Pseudo_IS}(i) and (iii).
\medskip

  (iii) By Theorem~\ref{thm:Descr_Pseudo_IS}(ii) and Definition~\ref{def:PseudoSF}, the mapping $\cF:\sH\ni f\mapsto \cP(\cdot)f\in L^2(d\sigma)$ is a partial isometry with $\ker \cF=\mul {A}$
  and such that \eqref{eq:ParsevalMin} holds for all $f\in\sH\ominus\mul {A}$.
\end{proof}

\begin{remark} With every $\cR^{n\times n}$-pair $\begin{bmatrix}
    C({z})& D({z})
    \end{bmatrix}$
    let us associate a pair
\[
\begin{bmatrix}
    A({z})& B({z})
    \end{bmatrix}:=
\begin{bmatrix}
    C({z})& D({z})
    \end{bmatrix}
    \begin{bmatrix}
    I_n & I_n\\
    -J_p & J_p
    \end{bmatrix}.
\]
of $n\times n$ mvf's $A({z})$ and $ B({z})$. The conditions (i)-(iii) of Definition~\ref{def:Nk-pair} are equivalent to the conditions:
\begin{enumerate}
  \item [(a)] $i(\dsp A({z})J_p A({z})^*-B({z})J_p B({z})^*)\ge 0$ for all ${z}\in\dC_+$;
\smallskip
  \item [(b)] $A({z})J_p A^\#({z})-B({z})J_p B^\#({z})=0$ for all ${z}\in\dC_+\cup\dC_-$;
  \smallskip
  \item[(c)] $\rank \begin{bmatrix}
    A({z})& B({z})
    \end{bmatrix}=n$
  for all ${z}\in\dC_+\cup\dC_-$,
\end{enumerate}
and the boundary conditions \eqref{eq:11.ShtrCD} can be rewritten in terms of $A$ and $B$
as follows
\begin{equation}\label{eq:11.ShtrAB}
A({z})f(a)+B({z})f(b)=0.
\end{equation}
The formula \eqref{eq:2.2L} takes the form
\begin{equation}\label{eq:GRes_AB}
  {L}^{\langle*\rangle}\wh {\mathbf R}_{{z}}{L}
  =(A({z})+B({z})U({z}))^{-1}
  (-A({z})+B({z})U({z}))J_p+K/2,
\end{equation}
and coincides with the formula
for ${\sL}$-resolvents of $A$ from ~\cite{LaTe82}.
The admissibility conditions \eqref{Adm1} and \eqref{Adm2} take the form
\begin{equation}\label{Adm1CS}
  (I+U(iy))(A(iy)+B(iy)U(iy))^{-1}(A(iy)-B(iy))=o(y),\quad y \uparrow \infty,
\end{equation}
\begin{equation}\label{Adm2CS}
  (I-U(iy))(A(iy)+B(iy)U(iy))^{-1}(A(iy)+B(iy))=o(y),\quad y \uparrow \infty.
\end{equation}
\end{remark}
In the case when $p=1$ one can get an implicit characterization of admissible parameters $\tau$.  For this purpose let us recall first the definition of the $\cH$-indivisible interval from~\cite{Kac03}.
\begin{definition} \label{def:Hii}
Let  $\xi_\psi:=\begin{bmatrix}
                       \cos\psi \\
                       \sin \psi
                     \end{bmatrix}$ for $\psi\in[0,\pi)$.
  An interval $(\alpha,\beta)(\subset (a,b))$ is called $\cH$-indivisible ($\cH$-i.i.) of type $\psi$,
  if $H(t)=\xi_\psi\xi_\psi^*$ for a.e. $t\in(\alpha,\beta)$.
  \end{definition}
  The following observation was made in~\cite[Lemma 2.1]{Kac03}.
  \begin{lemma}\label{lem:IndInt}
	Let $\imath \subsetneq [a,b)$ be $\cH$-i.i. of type $\psi$. Then for any
$\text{\rm col}\{\pi_{\cH}{u}, \pi_{\cH}f\} \in S_{max}$
	\begin{equation}\label{eq:xi-const}
		\xi_\psi^*  {u}(x) \equiv \text{const}, \quad x \in \imath.
	\end{equation}
	If in addition $\pi_{\cH}{u} = 0$ on $\imath$, then this constant is zero.
\end{lemma}

Let  symmetric extensions $S_{\gamma,\cdot}$, $S_{\cdot,\psi}$ and a selfadjoint extension $S_{\gamma,\psi}$ of $S_{min}$ ($\gamma,\psi\in[0,\pi)$) be defined by
\begin{equation}\label{eq:Sgamma}
  S_{\gamma,\cdot}=\left\{\text{\rm col}\{\pi_{\cH}{u}, \pi_{\cH}f\}\in S_{\max}:\, \xi_\gamma^*{u}(a)=0,\,{u}(b)=0\,\right\},
\end{equation}
\begin{equation}\label{eq:Spsi}
  S_{\cdot,\psi}=\left\{\text{\rm col}\{\pi_{\cH}{u}; \pi_{\cH}f\}\in S_{\max}:\, {u}(a)=0,\,\xi_\psi^*{u}(b)=0\right\},
\end{equation}
\begin{equation}\label{eq:Sgamma_psi}
  S_{\gamma,\psi}=\left\{\text{\rm col}\{\pi_{\cH}{u}, \pi_{\cH}f\}\in S_{\max}:\, \xi_\gamma^*{u}(a)=0,\,\xi_\psi^*{u}(b)=0\,\right\}.
\end{equation}
The following statements follow from the results of \cite[Theorem 9.1]{Kac03}.
\begin{lemma}\label{lem:mul_Sgamma_psi}
\begin{enumerate}
  \item [\rm (i)]
       The equation~\eqref{eq:can,eq-n} has no $\cH$-i.i. at $a$ and at $b$ if and only $\mul S_{\min}=\mul S_{\max}$.
  \item [\rm (ii)]  If the system \eqref{eq:can,eq-n} has no $\cH$-i.i. at $b$, then
$
    \mul S_{\gamma,\cdot}=\mul S_{\gamma,\cdot}^* $
for all $\gamma\in[0,\pi)$.
  \item [\rm (iii)]      If the  equation~\eqref{eq:can,eq-n} has an $\cH$-i.i. at $b$ of type $\psi$, then we have
  \begin{equation}\label{eq:mulA_gamma_psi}
    \mul S_{\gamma,\psi}=\mul S_{\gamma,\cdot}^*\quad\text{for all $\gamma\in[0,\pi)$}.
  \end{equation}
If in addition, \eqref{eq:can,eq-n} has no  $\cH$-i.i. at $a$, then
  \begin{equation}\label{eq:mulA_gamma_psi3}
 \mul S_{\gamma,\psi}=\mul S_{\max}\quad  \text{for all $\gamma\in[0,\pi)$\quad
 and \ $\text{\rm dim}(\mul S_{\max}/\mul S_{\min})=1$.}
 \end{equation}
  \item [\rm (iv)] If the equation~\eqref{eq:can,eq-n} has $\cH$-i.i. at $a$ of type $\gamma$ and at $b$  of type $\psi$, then
  \begin{equation}\label{eq:mulA_gamma_psi4}
 \mul S_{\gamma,\psi}=\mul S_{\max}\quad
 \text{and \ $\text{\rm dim}(\mul S_{\max}/\mul S_{\min})=2$}.
 \end{equation}
      \end{enumerate}
\end{lemma}
\begin{proof}
  (i) \& (ii) If $\text{\rm col}\{\pi_{\cH}{u}, \pi_{\cH}f\} \in S_{\gamma,\cdot}^*$ and $\pi_{\cH}{u}= 0$, then
${u}(b)= 0$.
Indeed, if ${u}(b)\ne 0$, then since ${u}$ is continuous, there exists $\varepsilon>0$ such that ${u}(x)\ne 0$ for all $x\in\imath=(b-\varepsilon,b)$. By \cite[Theorem 9.1(a)]{Kac03}, $\imath$ is an $\cH$-i.i.  for system \eqref{eq:can,eq-n}. This contradicts the assumption of item (ii).

The proof of (i) is similar.
\medskip

(iii) Let  $\text{\rm col}\{\pi_{\cH}{u}, \pi_{\cH}f\} \in S_{\gamma,\cdot}^*$ and $\pi_{\cH}{u}=0$. Then, by Lemma~\ref{lem:IndInt},
$\xi_\psi^*  {u}(x) \equiv 0$ and hence
$\text{\rm col}\{\pi_{\cH}{u}, \pi_{\cH}f\} \in S_{\gamma,\psi}$.

If in addition, \eqref{eq:can,eq-n} has no  $\cH$-i.i. at $a$, then for
 $\text{\rm col}\{\pi_{\cH}{u}, \pi_{\cH}f\} \in S_{\max}$, $\pi_{\cH}{u}=0$
the reasoning from (ii)
yield ${u}(a)=0$ and now we get
$\text{\rm col}\{\pi_{\cH}{u}, \pi_{\cH}f\} \in S_{\cdot,\psi}\subset S_{\gamma,\psi}$.
This implies that
  \begin{equation}\label{eq:mulA_gamma_psi5}
\mul S_{\min} \subsetneq\mul S_{\cdot,\psi}=\mul S_{\gamma,\psi}=\mul S_{\gamma,\cdot}^*=\mul S_{\max}
\end{equation}
and thus proves the second relation in~\eqref{eq:mulA_gamma_psi3}.

(iv) is proved similarly.
\end{proof}
\begin{proposition}\label{prop:Hii_intermed}
Let the assumptions of Theorem \ref{thm:Descr_Pseudo_IS3} hold. Then
\begin{enumerate}
  \item [\rm(i)]
 If the equation~\eqref{eq:can,eq-n} has no $\cH$-i.i. and only in this case,  the set 	of spectral functions $\sigma$  of $\langle A,\sL\rangle$
 is non-empty
 and in this case  the set 	$\Sigma_{\rm sf}(A;\sL)$ is parametrized by the formula \eqref{eq:IntRepRMin},
 where $\tau \in\wt\cR^{2\times 2}$.

  \item [(ii)]    If the equation~\eqref{eq:can,eq-n} has  no $\cH$-i.i.  at the endpoints $a$ and $b$, then in
\eqref{eq:IntRepRMin} it holds
 $\sigma\in\Sigma_{\rm psf}(A;\sL)\Leftrightarrow\tau \in\wt\cR^{2\times 2}$.

\item [(iii)]   If~\eqref{eq:can,eq-n} has  $\cH$-i.i. at $a$ and at $b$  of type $\pi/2$, then
   $\sigma\in\Sigma_{\rm psf}(A;\sL)$ if and only if
   $\tau$ admits the representation
    \begin{equation}\label{eq:Admis_tau}
      \tau=T_X[\varepsilon], \quad
      X=\begin{bmatrix}
          0 & -1 & -1 &  0 \\
          0 &  1 & -1 &  0 \\
          1 &  0 &  0 & -1 \\
          1 &  0 &  0 &  1
        \end{bmatrix},
    \end{equation}
where $\varepsilon\in\cR^{2\times 2}$
and $\varepsilon(iy)=o(y)$ as $ y\to\infty$.
\item [(iv)]   If~\eqref{eq:can,eq-n} has  $\cH$-i.i. at $a$ and at $b$  of type $\pi/2$, then
   $\sigma\in\Sigma_{\rm qpsf}(A;\sL)$ if and only if $\tau$ admits the representation \eqref{eq:Admis_tau},
where $\varepsilon\in\wt\cR^{2\times 2}$ and $\varepsilon_{\rm op}(iy)=o(y)$ as $ y\to\infty$.
\end{enumerate}

\end{proposition}
\begin{proof}
  (i) By Definition~\ref{def:PseudoSF}, the set of spectral functions $\sigma$  of $\langle A,\sL\rangle$ is non-empty if and only if $\ker \cF=\mul A=\{0\}$.
By \cite[Theorem 9.2]{Kac03}, $\mul A=\{0\}$ if and only if   the equation~\eqref{eq:can,eq-n}  has no $\cH$-i.i. In this case the operator $S_{\min}$ is densely defined, the set  of spectral functions $\sigma$  of $\langle A,\sL\rangle$ coincides with the set of pseudo-spectral functions $\sigma$  of $\langle A,\sL\rangle$ and, by Definition \ref{def:Adm}, all the parameters  $\tau \in\wt\cR^{2\times 2}$ are $\Pi$-admissible.

(ii) If the equation~\eqref{eq:can,eq-n}  has no $\cH$-i.i. at the endpoints $a$ and $b$, then by Lemma~\ref{lem:mul_Sgamma_psi},
$    \mul S_{\min}=\mul S_{\min}^*$
and hence      $\mul \wt S=\mul S_{\min}$ for every minimal
selfadjoint extension $\wt S$ of $S_{N,\cdot}$.
Therefore, every parameter $\tau \in\wt\cR^{2\times 2}$ in \eqref{eq:IntRepRMin} is $\Pi$-admissible.

(iii) Let the  equation~\eqref{eq:can,eq-n} have $\cH$-i.i. at $a$ and at $b$  of type $\pi/2$.
Consider a new boundary triple  $\Pi' = (\dC^2, \Gamma_0', \Gamma_1')$ for $S_{\max}$
\begin{equation}\label{eq:BT'_S_max}
  \Gamma_0'f=\begin{bmatrix}
               -f_2(a) \\
               f_2(b)
             \end{bmatrix}, \quad
  \Gamma_1'f=\begin{bmatrix}
               f_1(a) \\
               f_1(b)
             \end{bmatrix},
\end{equation}
which is connected to the boundary triple  $\Pi= (\dC^2, \Gamma_0, \Gamma_1)$
of the form~\eqref{eq:BT_Can}
by the formula $\Gamma'=X\Gamma$, where $X$ is a $iJ_p$-unitary matrix given by~\eqref{eq:Admis_tau}. Then the corresponding $\Pi'\sL$-resolvent matrix takes the form
\[
W'({z}):=W_{\Pi'{\sL}}^{\ell}({z})=XW_{\Pi{\sL}}^{\ell}({z})=XW({z}).
\]

By Lemma~\ref{lem:mul_Sgamma_psi}, the linear relation $S_0'=\ker \Gamma_0'$ has the property
$\text{dim}\mul S_0'/\mul S_{\min}=2$ and hence, by Theorem~\ref{krein}, the set of
$\sL$-resolvents of $S_{\min}$ is described by the formula
\[
{L}^{\langle*\rangle}\wh{\mathbf R}_{z}{L}=T^\ell_{W'}[\varepsilon({z})],
\quad\text{where $\varepsilon\in\wt\cR^{2\times 2}$.}
\]
By \eqref{tlimit00}, the parameter $\varepsilon$ is admissible
if $\varepsilon\in\cR^{2\times 2}$ and $\varepsilon(iy)=o(y)$ as $ y\to\infty$.
Since $T^\ell_{W'}[\varepsilon]=T^\ell_{W}[T^\ell_X[\varepsilon]]$, we get the statement (iii)
by setting $\tau=T^\ell_X[\varepsilon]$.

(iv) follows from the reasonings of (iii), the equivalence  \eqref{tlimit01} and  Theorem~\ref{thm:Descr_Pseudo_IS}(ii).
\end{proof}

\subsection{Spectral and pseudo-spectral functions of $S_{N,\cdot}$}
Let $S_{N\cdot}$  be the intermediate extension of the linear relation $S_{\min}$ which is obtained by imposing the  Neumann condition at the left endpoint $a$:
	\begin{equation}\label{eq:ADaN}
		S_{N\cdot} := \{ \text{\rm col}\{\pi_{\cH}u, \pi_{\cH}f\} \in A_{\max} :\, u_2(a) =0, \ u(b) = 0\ \}.
	\end{equation}
$S_{N\cdot}$ is a symmetric linear relation with the defect numbers $n_{\pm}(S_{N\cdot})=p$.
\begin{proposition}\label{prop:RM_for_N}
	Let the right endpoint $b$ be quasiregular for the system \eqref{eq:can,eq-n} and let $S_{N\cdot}$ be the linear relation defined in~\eqref{eq:ADaN}. Then:
	
	\begin{enumerate}
		\item[(i)]
		The adjoint linear relation $S_{N\cdot}^*$ takes the form
		\begin{equation*}
	S_{N\cdot}^* = \{ \text{\rm col}\{\pi_{\cH}u, \pi_{\cH}f\} \in A_{\max}:\, u_2(a) = 0 \}.
		\end{equation*}

		\item[(ii)]
		The triple $\Pi^{(N)} := (\dC^p, \Gamma_0^{(N)}, \Gamma_1^{(N)})$ where
for $\text{\rm col}\{\pi_{\cH}u, \pi_{\cH}f\} \in S_{N\cdot}^*$
			\begin{equation}\label{eq:btriple-NII}
				\Gamma_0^{(N)} \text{\rm col}\{\pi_{\cH}u, \pi_{\cH}f\} := u_2(b),
				\quad
				\Gamma_1^{(N)} \text{\rm col}\{\pi_{\cH}u, \pi_{\cH}f\} := u_1(b),
			\end{equation}
			is a boundary triple for $S_{N\cdot}^*$.

		\item[(iii)]
		The corresponding Weyl function and the $\gamma$-field are
		\begin{equation}\label{eq:WeylF_NII}
			M({z}) =c_1(b,{z})c_2(b,{z})^{-1}, \qquad
			\gamma({z}) =  c(\cdot,{z})c_2(b,{z})^{-1}.
		\end{equation}

		\item[(iv)]
The resolvent  $R^0_{{z}}$ of
	the extension $S_0 = \ker \Gamma_0^{(N)}$ is given by
		\begin{equation}\label{eq:Resolv_NII}
		(R^0_{{z}} f) (x)
		= \frac12{U(x,{z})} \int_a^b
			\left\{ \sgn_+(s-x)I_{2p} -D(z) \right\} J U^\#(s,{z}) \cH(s) f(s)\,ds,
	\end{equation}
	where $f \in \cL^2_\cH(a,b)$, ${z} \in \rho(S_0)$ and
		\begin{equation}\label{eq:Resolv_D}
D(z)=\begin{bmatrix}
			                      -I_p & -2c_2(b,{z})^{-1}s_2(b,{z})\\
			                       0 & I_p
			                         \end{bmatrix}.
\end{equation}
	\item[(v)]
		With the gauge $\sL(\subset \sH_-)$ and mapping $L :\, \dC^p \to \sL $ defined as
		\begin{equation}\label{eq:gauge-AN}
\sL=\left\{L\xi:\xi\in\dC^p\right\},\quad			
L\xi:= \text{\rm col}\{\delta_a\otimes\xi,0\},
		\end{equation}
we have $\rho(A,\sL)=\dC$ and
\begin{equation}\label{eq:P_lambda_SN}
  \cP({z})f= \begin{bmatrix}I_p & 0\end{bmatrix}\int_a^b U^\#(s,{z})\cH(s)f(s)ds.
\end{equation}
\item[(vi)]
        The corresponding  right $\Pi^{(N)}\sL$-resolvent matrix takes the form
		\begin{equation}\label{eq:intermed-WrN}
			W^r_{\Pi^{(N)}\sL}({z}) =
								\begin{bmatrix}
					s_2(b,{z}) & s_1(b,{z}) \\
					c_2(b,{z}) & c_1(b,{z})
				\end{bmatrix}
				-
				\begin{bmatrix}
					Kc_2(b,{z}) & Kc_1(b,{z}) \\
                             0      & 0
				\end{bmatrix} .
		\end{equation}
where $K=K^*\in\dC^{\ptp}$.
	\end{enumerate}
\end{proposition}
\begin{proof}
  (i) \& (ii) are straightforward.

  (iii) The formulas \eqref{eq:WeylF_NII} follow from \eqref{eq:11.M} and \eqref{eq:11_gamma1} since the defect subspace $\sN_z$ of $S_{N\cdot}$ consists of vector-functions
  \[
  f_z(\cdot)=c(\cdot,z)\xi,\quad \xi\in\dC^p.
  \]

  (iv) By \cite[eq. (3.6)]{LaTe82} the resolvent of an extension $\wt S$ of $S_{\min}$ characterized by the boundary condition \eqref{eq:11.ShtrAB} takes the form~\eqref{eq:Resolv_NII}, where
  \[
  D(z)=(A+BU(z))^{-1}(A-BU(z)).
  \]
  For the extension $S_0=\ker\Gamma_0^{N}$ one can take
  $A=\begin{bmatrix}
       0 & 0 \\
       0 & 1
     \end{bmatrix}$,   $B=\begin{bmatrix}
       0 & 1 \\
       0 & 0
     \end{bmatrix}$ and hence the expression for
     \begin{equation*}
  D(z)=\begin{bmatrix}
       c_2(b,z) &s_2(b,z)  \\
       0        & 1
     \end{bmatrix}^{-1}
     \begin{bmatrix}
       -c_2(b,z)& -s_2(b,z)  \\
       0        & 1
     \end{bmatrix}
  \end{equation*}
  is reduced to~\eqref{eq:Resolv_D}.
\end{proof}

\begin{theorem}
  \label{thm:Descr_Pseudo_N}
  Let $\Pi = (\dC^p, \Gamma_0, \Gamma_1)$ be the boundary triple for $S_{N\cdot}^*$
  given by~\eqref{eq:BT_Can},
let the gauge $\sL$ and the mapping
$L\in\cB(\dC^p,\sL)$ be defined by \eqref{eq:gauge-AN}
 and let
$W_{\Pi{\sL}}^{r}({z})$  be the right $\Pi{\sL}$-resolvent matrix of $S_{N\cdot}$ given by
\eqref{eq:intermed-WrN}.
Then
\begin{enumerate}
  \item [\rm(i)]
The formula
\begin{equation}\label{eq:IntRepRN}
  \int_{\dR}\left(\frac{1}{{\lambda}-z}
  -\frac{{\lambda}}{1+{\lambda}^2}\right)d\sigma({\lambda})=
  T_{W_{\Pi{\sL}}^{r}(z)}[\tau(z)]+\wh K
\end{equation}
establishes a bijective correspondence between
all pseudo-spectral functions $\sigma$  of $\langle S_{N\cdot},\sL\rangle$
and all admissible families $\tau\in\wt\cR^{p\times p}$. Here $\wh K$ is a Hermitian ${p\times p}$ matrix depending on $\sigma.$

\item [\rm(ii)]	
For every pseudo-spectral function $\sigma$ the generalized Fourier transform
\begin{equation}\label{eq:FourierN}
\cF:\sH\ni f\mapsto F({\lambda})= \begin{bmatrix}I_p & 0\end{bmatrix}\int_a^b U^\#(s,{\lambda})\cH(s)f(s)ds
\end{equation}
is a partial isometry from $\sH$ to $L^2(d\sigma)$
such that $\ker\cF=\mul S_{N\cdot}$ and for all $f,g\in \sH\ominus\mul S_{N\cdot}$ the Parseval equality holds
\begin{equation}\label{eq:ParsevalN}
  \int_a^b g(s)^*\cH(s)f(s)ds=\int_\dR G({\lambda})^*d\sigma({\lambda})F({\lambda}),\quad\text{where}\quad
  G=\cF g.
\end{equation}
\item [\rm(iii)]
The adjoint to the  generalized Fourier transform is given by
\begin{equation}\label{eq:Fourier_InvN}
(\cF^{*}(F))(x)
=\int_\dR c(x,{\lambda})d\sigma({\lambda})F({\lambda}),
\quad F\in L^2(d\sigma),
\end{equation}
where the integral in \eqref{eq:Fourier_InvN} converges in the norm of $L^2_{\cH}(I)$.

\end{enumerate}
\end{theorem}
\begin{proof}
  (i) follows from Theorem~\ref{thm:Descr_Pseudo_IS}(i).

  (ii) By Theorem~\ref{thm:Descr_Pseudo_IS}(ii) and Definition~\ref{def:PseudoSF}, the mapping $\cF:\sH\ni f\mapsto \cP(\cdot)f\in L^2(d\sigma)$ is a partial isometry with
  $\ker \cF=\mul S_{N\cdot}$
  and such that \eqref{eq:Parseval} holds for all $f\in\sH\ominus \mul S_{N\cdot}$.

  (iii)
  The proof of (iii) for arbitrary $\sigma\in\Sigma_{psf}(S_{\min},\sL)$
   is given in \cite[Proposition 3.4]{Mog15Ufa}. We present here a short proof for an orthogonal pseudo-spectral function  $\sigma$.

    Let $F({z}) $ be a function from $L^2(d\sigma)$ with a compact support
  ($\subset \Delta$).
   Since for orthogonal pseudo-spectral function $\sigma\in\Sigma_{psf}(S_{\min},\sL)$
   it holds $\ran\cF=L^2(d\sigma)$, there exists
    $f\in L^2_{\cH}\ominus\mul S_{N\cdot}$
    such that $\cF f=F$, see~ \eqref{eq:Expansion_f2}. Let $g_t(s)=\mathbf{1}_{[a,t)}$
  and let
  \[
  G({\lambda}):=(\cF g_t)({\lambda})
  =\begin{bmatrix}I_p & 0\end{bmatrix}\int_a^t U^\#(s,{\lambda})\cH(s)\,ds.
  \]
  Then, by~\eqref{eq:Parseval}, the equality~\eqref{eq:ParsevalN} holds
  and hence
  \[
 (F,\cF g_t)_{L^2(d\sigma)}= (F,G)_{L^2(d\sigma)}
  =\int_{\Delta}\left(\int_a^t\begin{bmatrix}I_p & 0\end{bmatrix} U^\#(s,{\lambda})\cH(s)\,ds\right)^*d\sigma({\lambda})F({\lambda}).
  \]
 By the Fubini theorem, we get
  \[
 (F,\cF g_t)_{L^2(d\sigma)}
  =\int_a^t \cH(s) \left(\int_{\Delta} U(s,{\lambda})\begin{bmatrix}I_p \\ 0\end{bmatrix}\,d\sigma({\lambda})F({\lambda})\right)\,ds.
  \]
  Since the set $\{g_t(s):\,t\in[a,b)\}$ is dense in $L^2_\cH(I)$, this proves \eqref{eq:Fourier_InvN}.
\end{proof}

As follows from Definition~\ref{def:PseudoSF},
\begin{equation}\label{eq:FF*}
\cF^*\cF=P_{\sH\ominus\ker \cF}\quad\text{and}
\quad \cF\cF^*=P_{\ran \cF}.
\end{equation}
In particular, if $\sigma$ is an orthogonal spectral function for
$\langle S_{\min},\sL\rangle$,
then $\cF^*$ is the inverse to the generalized Fourier transform $\cF$.

\begin{proposition}\label{prop:Indivis_Int}
Let in the assumptions of Theorem \ref{thm:Descr_Pseudo_N} $p=1$. Then
\begin{enumerate}
  \item [\rm(i)]
 If the system \eqref{eq:can,eq-n} has no $\cH$-i.i. and only in this case,  the set of spectral functions $\sigma$  of $\langle S_{N\cdot},\sL\rangle$ is non-empty and
 $\Sigma_{sf}(S_{N\cdot},\sL)=\Sigma_{psf}(S_{N\cdot},\sL)$.

  \item [(ii)]    If the system \eqref{eq:can,eq-n} has  no $\cH$-i.i.  at the endpoint $b$, then the formula
\eqref{eq:IntRepRMin}
establishes a bijective correspondence between
all pseudo-spectral functions $\sigma$  of $(S_{N\cdot},\sL)$
and all  $\tau\in\wt\cR$.

\item [(iii)]   If the system \eqref{eq:can,eq-n} has an $\cH$-i.i.  at $b$  of type $\pi/2$, then we have in
\eqref{eq:IntRepRMin}
 \[
 \text{$\sigma\in\Sigma_{\rm psf}(A;\sL)\Leftrightarrow\tau \in\cR$ and $\tau(iy)=o(y)$ as $ y\to\infty$.}
 \]

     \item [(iv)]
  If the system \eqref{eq:can,eq-n} has an $\cH$-i.i. at $b$  
  of type $\psi$,
then  $\sigma\in\Sigma_{\rm psf}(A;\sL)$
  if and only if  $\tau\in\cR$
\begin{equation}\label{eq:NevCond_psi}
  (\sin\psi\tau(iy)+\cos\psi)(-\cos\psi\tau(iy)+\sin\psi)^{-1}=o(y)\quad\text{ as }\quad y\to\infty.
\end{equation}
\end{enumerate}
\end{proposition}
\begin{proof}
  (i) By Definition~\ref{def:PseudoSF}, the set of spectral functions $\sigma$  of $\langle A,\sL\rangle$ is non-empty if and only if $\ker \cF=\mul A=\{0\}$.
  By Lemma~\ref{lem:mul_Sgamma_psi}(i),
$\mul A=\{0\}$ if and only if   the system \eqref{eq:can,eq-n} has no $\cH$-i.i. In this case the operator $S_{N,\cdot}$ is densely defined, the set  of spectral functions $\sigma$  of $\langle A,\sL\rangle$ coincides with the set of pseudo-spectral functions $\sigma$  of $\langle A,\sL\rangle$ and, by Definition \ref{def:Adm}, all the parameters $\tau\in\wt\cR$ are $\Pi$-admissible.

(ii) If the system \eqref{eq:can,eq-n}  has no $\cH$-i.i. at point $b$, then by Lemma~\ref{lem:mul_Sgamma_psi}(ii),
$    \mul S_{N,\cdot}=\mul S_{N,\cdot}^*$
and hence      $\mul \wt S=\mul S_{N,\cdot}^*$ for every minimal
selfadjoint extension $\wt S$ of $S_{N,\cdot}$.
Therefore, every parameter $\tau$ in \eqref{eq:IntRepRN} is $\Pi$-admissible.

(iii) If the system \eqref{eq:can,eq-n} has an $\cH$-i.i. $(\beta,b)$ of type $\pi/2$, then, 
by Lemma~\ref{lem:mul_Sgamma_psi}(iii), the extension $S_{N,\pi/2}$ has maximal multi-valued part, i.e.
$
    \mul S_{N,\pi/2}=\mul S_{N}^*.
$
By Proposition~\ref{prop:Adm_Nev}, the $\Pi$-admissibility of  $\tau$ is equivalent to the single condition~\eqref{tlimit00}.

(iv) By 
~\eqref{eq:mulA_gamma_psi5}, we have in this case
  $
    \mul S_{N,\psi}=\mul S_{N}^*.
  $
Consider the triple $\Pi^{(\psi)} := (\dC, \Gamma_0^{(\psi)}, \Gamma_1^{(\psi)})$, where for $\wh u=\text{\rm col}\{\pi_{\cH}u, \pi_{\cH}f\} \in S_{N,\cdot}^*$
			\begin{equation*}\label{eq:btriple-Da}
			\Gamma^{(\psi)}\wh u=\begin{bmatrix}\Gamma_0^{(\psi)} \wh u\\ \Gamma_1^{(\psi)} \wh u\end{bmatrix}
=\begin{bmatrix}\xi_\psi^*u(b) \\ \xi_\psi^*Ju(b)\end{bmatrix}
=X\Gamma\wh u
				\quad
				\text{and}\quad
X=\begin{bmatrix}\sin\psi & \cos\psi\\
                    -\cos\psi & \sin\psi\end{bmatrix}.
\end{equation*}
Then $S_0:=\ker\Gamma_0^{(\psi)}= S_{N,\psi}$.
With the gauge mapping $L : \dC \to \sL \subset \sH_-$ defined by~\eqref{eq:gauge-AN}
		the corresponding   right $\Pi\sL$-resolvent matrix takes the form
		$
			W^r_{\Pi^\psi \sL}({z}) =W^r_{\Pi^{(2)}\sL}({z})X^*.
$
Since $
    \mul S_{\gamma,\psi}=\mul S_{\gamma,\cdot}^*
  $, by Theorem~\ref{thm:Descr_Pseudo_N},
  every pseudo-spectral function $\sigma$  of $\langle A,\sL\rangle$ admits the representation
\begin{equation}\label{eq:PSF_descrN}
  \int_{\dR}\left(\frac{1}{{\lambda}-z}-\frac{{\lambda}}{1+{\lambda}^2}\right)d\sigma({\lambda})=
  T_{W^r_{\Pi^\psi \sL}}[\wt\tau]+k,
  \end{equation}
  where $k\in\dR$ and
$\wt\tau\in\cR$  satisfies the single condition 
\begin{equation}\label{eq:PSF_admN}
  \wt\tau(iy)=o(y)\quad\text{as $ y\to\infty$}.
\end{equation}
The right hand side of~\eqref{eq:PSF_descrN} can be rewritten as
\[
T_{W^r_{\Pi^\psi \sL}}[\wt\tau]=
T_{W_{\Pi{\sL}}}^{r}[\tau],
\quad\text{
where }\quad
\tau=T_{X^*}[\wt\tau]
\]
and since
\[
 \wt\tau=T_{X^{-*}}[\tau],\quad\text{
and }\quad X^{-*}=-JXJ=X,
\]
we obtain that  \eqref{eq:PSF_admN} ensures the admissibility of $\tau$ in \eqref{eq:IntRepRN}.
\end{proof}
\subsection{Limit point case}
Let $p=1$ and let  the differential expression
   \eqref{eq:can,eq-n} be quasi-regular at $a$ and limit point at $b$.
   Let us consider the
   minimal  $S_{\min}$ and the maximal $S_{\max}$ relations  generated by the expression
   \eqref{eq:can,eq-n} in $L^2_\cH(I)$, $I=[a,b]$.
Then $A=S_{\min}$ is a symmetric linear relation with the defect numbers $n_{\pm}(A)=1$
and $A^*=S_{\max}$.
\begin{proposition}\label{prop:RM_for_LP}
Let the canonical system
   \eqref{eq:can,eq-n} be quasi-regular at $a$ and limit point at $b$, let $m({z})$
   be the Weyl-Titchmarsh coefficient defined by the condition
\[
y(\cdot, {z}):= s(\cdot, {z}) - m({z})\, c(\cdot, {z})\in \cL^2_\cH(I)
\]
and let the gauge $\sL$ and the mapping $L:\dC\to\sH_-$ be defined by
		\begin{equation}\label{eq:gauge-LP}
\sL= \left\{\begin{bmatrix}0\\
\xi\delta_a \end{bmatrix}:\,\xi\in\dC\right\},\quad
L\xi=\begin{bmatrix}0\\
\xi\delta_a \end{bmatrix}
			\quad \xi \in \dC.
		\end{equation}

     Then:
	
	\begin{enumerate}
		\item[(i)]
		The triple $\Pi := (\dC, \Gamma_0, \Gamma_1)$ where
			\begin{equation*}\label{eq:btriple-LP}
				\Gamma_0 \{\pi_{\cH}u; \pi_{\cH}f\} := u_2(a),
				\quad
				\Gamma_1 \{\pi_{\cH}u; \pi_{\cH}f\} := -u_1(a),
				\quad
				\{\pi_{\cH}u; \pi_{\cH}f\} \in A^*,
			\end{equation*}
			is a boundary triple for $A^*$.

		\item[(ii)]
		The corresponding Weyl function and the $\gamma$-field are
		\begin{equation}\label{eq:WeylF_LP}
			M({z}) =m({z}), \qquad
			\gamma({z}) = y(\cdot, {z}),\qquad
		\end{equation}

	\item[(iii)]
	$\rho(A,\sL)=\dC$ and	
the right $\Pi\sL$-resolvent matrix takes the form
		\begin{equation}\label{eq:intermed-WrLP}
			W^r_{\Pi\sL}({z}) =
								\begin{bmatrix}
					0      & -1 \\
					1   & m({z})
				\end{bmatrix}.
		\end{equation}
	\end{enumerate}
\end{proposition}
\begin{proof}
(i) \& (ii) are straightforward, cf. \cite[Theorem 7.8.2]{BeHaSn20}.

  (iii) The proof of the equality $\rho(A,\sL)=\dC$ is similar to that in Proposition~\ref{prop:PreresolvMatrix}. By~\cite[Theorem 7.8.5]{BeHaSn20},
  the resolvent of the extension $S_0=\ker\Gamma_0$ is given by
\[
\begin{split}
(R_z^0f)(x)&=(-s(x,z)+c(x,z)m(z))\int_a^x c^\#(t,z)\cH(t)f(t)\,dt\\
&+c(x,z)\int_x^b(-s^\#(t,z) +m(z)c^\#(t,z))\cH(t)f(t)\,dt.
\end{split}
\]
Hence we get $\wt R_z^0L\xi=0$, since for all $f\in\sH$
\[
(f,\wt R_z^0L\xi)_{\sH}=\langle R_{\bar z}^0f, \text{\rm col}\{0, \xi\delta_a\} \rangle_{+,-}=0.
\]
Therefore, by Lemma~\ref{lem:PreresM} and \eqref{eq:WeylF_LP}, we get
$\sA(z)=\begin{bmatrix}
          m(z) & 1 \\
          1 & 0
        \end{bmatrix}$.
        Now the formula~\eqref{eq:intermed-WrLP} follows from~\eqref{eq:LresM}.
\end{proof}
The next theorem follows from Theorem~\ref{thm:Descr_Pseudo_IS} and
Proposition~\ref{prop:RM_for_LP}.
\begin{theorem}
  \label{thm:Descr_Pseudo_LP}
  Let the expression   \eqref{eq:can,eq-n} be quasi-regular at $a$ and limit point at $b$ and let $A=S_{\min}$.
  Let $\Pi = (\dC, \Gamma_0, \Gamma_1)$ be the boundary triple for $S_{\max}$
  given by~\eqref{eq:BT_Can},
let 
a gauge $\sL$ for $S_{\min}$ and the mapping
$L\in\cB(\dC,\sL)$ be defined by \eqref{eq:gauge-LP}.
Then
\begin{enumerate}
  \item [\rm(i)]
The formula
\begin{equation}\label{eq:IntRepLP}
  \int_{\dR}\left(\frac{1}{{\lambda}-z}-\frac{{\lambda}}{1+{\lambda}^2}\right)d\sigma({\lambda})=
  -(\tau(z)+m(z))^{-1}+k
\end{equation}
establishes a bijective correspondence between
all pseudo-spectral functions $\sigma$  of $\langle A,\sL\rangle$
and all $\Pi$-admissible families $\tau\in\wt\cR$, where $k\in\dR$ depends on $\sigma$.
\item [\rm(ii)]	
For every pseudo-spectral function $\sigma$ the generalized Fourier transform
\begin{equation}\label{eq:FourierLP}
\cF:\sH\ni f\mapsto F(\lambda)= \int_a^b y^\#(s,\lambda)\cH(s)f(s)ds
\end{equation}
is a partial isometry from $\sH$ to $L^2(d\sigma)$
such that $\ker\cF=\mul A$ and for all $f,g\in \sH\ominus\mul A$ the Parseval equality holds
\begin{equation}\label{eq:ParsevalLP}
  \int_a^b g(s)^*\cH(s)f(s)ds=\int_\dR G({\lambda})^*d\sigma({\lambda})F({\lambda}),\quad\text{where}\quad
  G=\cF g.
\end{equation}
\end{enumerate}
\end{theorem}
\begin{proposition}
Let in the assumptions of Theorem \ref{thm:Descr_Pseudo_LP} $p=1$. Then
\begin{enumerate}
  \item [\rm(i)]
 If the equation~\eqref{eq:can,eq-n} has no $\cH$-i.i. and only in this case,  the set of spectral functions $\sigma$  of $\langle S_{\min},\sL\rangle$ is non-empty and
 $\Sigma_{sf}(S_{\min},\sL)=\Sigma_{psf}(S_{\min},\sL)$.

  \item [(ii)]    If the equation~\eqref{eq:can,eq-n} has  no $\cH$-i.i.  at the endpoint $a$, then all  $\tau\in\wt\cR$ are $\Pi$-admissible.

\item [(iii)]   If the system~\eqref{eq:can,eq-n} has  $\cH$-i.i.  at $a$  of type $\pi/2$, then  the parameter $\tau$ in~\eqref{eq:IntRepLP} is $\Pi$-admissible
    if and only if $\tau\in\cR$ and
 $
      \tau(iy)=o(y)$ as $ y\to\infty$.
     \item [(iv)]
  If the system \eqref{eq:can,eq-n} has an $\cH$-i.i.  at $a$ of type $\gamma$, then  the parameter $\tau\in\wt\cR$ in~\eqref{eq:IntRepLP} is $\Pi$-admissible if and only if  
\[
  (\sin\gamma\tau(iy)+\cos\gamma)(-\cos\gamma\tau(iy)+\sin\gamma)^{-1}=o(y)\quad\text{ as }\quad y\to\infty.
\]
\end{enumerate}

\end{proposition}
\begin{remark}
Basic results  in the theory of spectral functions for
 canonical systems with $p=1$ are discussed in the monograph
 by F. V. Atkinson~\cite{Atk64}. Direct  and inverse spectral problems for canonical systems were studied by L. de Branges~\cite{dB68}, for strings by
 I.S. Kac and M.G. Krein~\cite{KacK68}.
Spectral and pseudo-spectral functions of a symmetric relation  $S_{\gamma\cdot}$  for canonical differential systems  (with  $p>1$) in the regular case were studied by A.~Sakhnovich \cite{Sak92},
in the general case  by L.~Sakhnovich~\cite{SakL99}, and
by D. Arov and H. Dym~\cite{ArDy12}. Criteria for the existence of spectral and pseudo-spectral functions
of the symmetric relation $S_{\gamma\cdot}$  for canonical  systems in the regular and singular case
in terms of the $\cH$-indivisible intervals were found by I.S.~Kac~\cite{Kac03}. Description of LT-spectral functions
of symmetric relations associated with  canonical  systems were presented
by H. Langer and B. Textorius in~\cite{LaTe85}.
Eigenfunction expansions of first-order
 symmetric systems with non-equal defect numbers were studied by
 V. Mogilevskii~\cite{Mog15}.
\end{remark}

\section*{Acknowledgements}
The author expresses deep gratitude to Yu.M. Arlinskii and D.Z. Arov for their valuable comments and fruitful discussions.

%
%
%

\end{document}